\documentclass[10pt]{report}
\usepackage{amsmath,amssymb,amsthm}
\usepackage[arrow,matrix,curve]{xy}

\usepackage{mathrsfs,bbm,eucal}
\theoremstyle{plain}
\newtheorem{thm}{Theorem}[section]
\newtheorem*{thm*}{Theorem}
\newtheorem{lem}[thm]{Lemma}
\theoremstyle{plain}
\newtheorem{prop}[thm]{Proposition}
\newtheorem{conj}[thm]{Conjecture}
\theoremstyle{plain}
\newtheorem{example}[thm]{Example}
\newtheorem{cor}[thm]{Corollary}
\theoremstyle{definition}
\newtheorem{defn}[thm]{Definition}
\theoremstyle{remark}
\newtheorem{rem}[thm]{Remark}

\DeclareMathAlphabet{\mathpzc}{OT1}{pzc}{m}{it}

\newcommand{\bem}[1]{}
\newcommand{\dirac}{\mbox{$\mathcal{D}\!\!\!\!\!\:/\!\;$}}
\newcommand{\spinor}{\mbox{$S\!\!\!\!\!\:/\;\!$}}
\newcommand{\sspinor}{\mbox{$\scriptstyle S\!\!\!\!\!\:/\;\!$}}
\newcommand{\bundle}[1]{\CMcal{#1}}
\newcommand{\Cl}{\mathit{Cl}}
\newcommand{\R}{\mathbbm{R}}
\newcommand{\C}{\mathbbm{C}}

\newcommand{\Q}{\mathbbm{Q}}
\renewcommand{\H}{\mathscr{H}}
\newcommand{\He}{\mathscr{H}^\partial }
\newcommand{\Z}{\mathbbm{Z}}
\newcommand{\N}{\mathbbm{N}}
\newcommand{\id}{\mathbbm{1}}
\newcommand{\BU}{\mathrm{B}\mathrm{U}}
\newcommand{\group}{G}
\newcommand{\met}{\mathrm{Met}}
\newcommand{\scal}{\mathpzc{sc_{min}}}

\newcommand{\D}{\mathcal{D}}
\newcommand{\curv}[1]{\CMcal{#1}}
\newcommand{\proj}[1]{\mbox{$#1\mathrm{P}$}}

\newcommand{\ke}{\mathpzc{k}}
\newcommand{\ce}{\mathpzc{c}}

\newcommand{\cgv}{\mathpzc{cgv}}
\newcommand{\cg}{\mathpzc{cg}}
\newcommand{\ccg}{\mathpzc{ccg}}
\newcommand{\yam}{\mathcal{Y}}

\begin{document}

\begin{titlepage}
\thispagestyle{empty}
\begin{center}
\vspace*{2cm}
{\Huge \bf Scalar curvature and vector bundles }\\
\vspace{2cm}

{\Large Mario Listing}
\vspace{2cm}

{\bf Abstract}\\
\end{center}
In the first part we use Gromov's K--area to define the K--area homology which stabilizes into singular homology on the category of pairs of compact smooth manifolds. The second part treats the questions of certain curvature gaps. For instance, the $L^\infty $--curvature gap of complex vector bundles on a compact manifold is positive if and only if the K--area homology coincides with the reduced singular homology in all even degrees. In the third part we give some upper bounds of the scalar curvature on compact manifolds. In particular, we generalize results by Llarull and Goette, Semmelmann. 
\vspace{1cm}

\noindent The present work is my habilitation thesis which was accepted by the ''Fakult\"at f\"ur Mathematik und Physik der Albert--Ludwigs--Universit\"at Freiburg''. 

\tableofcontents 
\end{titlepage}
 
\chapter*{Introduction}
One of the main concerns in differential geometry is the relation of curvature and topology. Classical examples of this interaction are the Bonnet--Myers theorem or the theorem by Hadamard--Cartan. Latest research results on these subjects are the classification of manifolds with positive curvature operator \cite{BoWi} and the $1/4$--pinched differentiable sphere theorem \cite{BrSch}. However, methods of Riemannian geometry also supply important result in topology as the proof of the Poincar\'e conjecture in dimension $n=3$ shows \cite{pre_Per1,pre_Per2}. In this case the Ricci curvature flow determines an Einstein metric on a simply connected closed $3$--manifold which can only exist on $S^3$. In $4$--dimensions Yang--Mills theory and Seiberg--Witten theory are powerful tools to relate Riemannian geometry to differential topology. Last but not least the Atiyah--Singer index theorem and the Lichnerowicz formula determine a topological obstruction to the existence of positive scalar curvature on spin manifolds. 

In this work we are interested in the interaction of the curvature of vector bundles with the topology of the underlying manifold. This leads eventually to questions of positive scalar curvature which is  investigated in many papers. In the first chapter we use the curvature of complex vector bundles to define a semi homology theory which also serves as a well accessible obstruction against positive scalar curvature. Using a recent result by Hanke \cite{Hanke}, these obstructions have been known before in a different framework namely the Rosenberg index. This semi homology theory has further applications in the second chapter which is mostly concerned with the question of how small the curvature of a vector bundle can be before it has to be stably rational trivial. Remember that a flat vector bundle on a simply connected manifold is trivial, but this does not generalize to manifolds with nontrivial fundamental group. In the last chapter we show how to bound the scalar curvature from above if the manifold admits a metric of positive scalar curvature.

\subsubsection*{Chapter 1: Homology of finite K--area}
Gromov introduced in \cite{Gr01} the notion of K--area for Riemannian manifolds and proved that this area is finite for closed spin manifold of positive scalar curvature. In this result the K--area replaces the concept of enlargeability considered by Gromov and Lawson in \cite{GrLa2,GrLa3}. Apart from this interesting fact, the K--area is not  entirely bound to positive scalar curvature, because finite K--area of a compact manifold depends only on the homotopy type of the manifold whereas the existence of positive scalar curvature depends on the differentiable structure. So far the best obstruction for positive scalar curvature on spin manifolds is the vanishing of an invariant $\alpha ^\R _\mathrm{max}(M)\in \mathrm{KO}_n(C^*_{\mathrm{max},\R }\pi _1(M))$ introduced by Rosenberg in \cite{Ros1,Ros2}. In particular, $\alpha ^\R _\mathrm{max}(M)$ generalizes the Atiyah--Milnor--Singer invariant and does not vanish for enlargeable spin manifolds which was proved by Hanke and Schick in \cite{HaSch1,HaSch2}. Recently, Hanke generalized in \cite{Hanke} the concept of K--area to infinite dimensional bundles over $C^*$--algebras and proved that infinite K--area implies $\alpha ^\R _\mathrm{max}(M)\neq 0$. This justifies the observation that the Rosenberg index is the best index theoretic obstruction for positive scalar curvature.

The K--area defined in \cite{Gr01} and used below is given by $\sup \frac{1}{\| R^\bundle{E}\| }$ where the supremum is taken over a certain set of (finite dimensional) Hermitian vector bundles $\bundle{E}$ endowed with a Hermitian connection. The usage of finite dimensional bundles is the main difference to the preprint by Hanke \cite{Hanke}, but it has the advantage to get  explicit values for the K--area which becomes important in chapters 2 and 3. Below we introduce the K--area of a compact Riemannian manifold $(M,g)$ w.r.t.~a homology class $\theta \in H_{*}(M;G)$ where $H_*(M;G)$ means singular homology of $M$ with coefficients in an abelian group $G$. This leads to the definition of the \emph{homology groups with finite K--area} respectively the \emph{K--area homology}. In fact, the set of homology classes $\theta \in H_k(M;G)$ with finite K--area determine a subgroup $\H _k(M;G)\subseteq H_k(M;G)$ which is independent on the choice of the Riemannian metric on $M$. Moreover, the induced homomorphism of a continuous map $f:M\to N$ restricts to homomorphisms $f_*:\H _k(M;G)\to \H _k(N;G)$ which proves that $\H _k(M;G)$ depends only on the homotopy type of the compact manifold $M$. Below we give some examples of the K--area homology which show its nontrivial character. For instance, $\H _*(T^n)=\{ 0\} $ and $\H _*(N)=\{ 0\} $ if $N$ is a closed orientable surface with Euler characteristic $\chi (N)\leq 0$ whereas $\H _{2k}(M)=H_{2k}(M)$ for all $k>0$ if $M$ is a closed orientable manifold with finite fundamental group. Moreover, we construct manifolds with isomorphic fundamental group, isomorphic singular homology but different K--area homology and we give  examples of manifolds with isomorphic cohomology ring but different K--area homology. In fact, if $M^3$ is a hyperbolic homology sphere and $f:M\to S^3$ is a map of degree $\pm 1$, then $f$ induces isomorphisms in singular homology and cohomology but not in K--area homology since $\H _3(M)=\{ 0\} $ whereas $\H _3(S^3)=\Z $. Because the fundamental class of the sphere $S^k $ has finite K--area for all $k\geq 2$ we conclude as an application that the image of the Hurewicz homomorphism $h_k:\pi _k(M)\to H _k(M)$ is contained in $\H _k(M)$ for all $k\geq 2$. Conversely, the fundamental class of a connected compactly enlargeable manifold $M^n$ has infinite K--area which means $\H _n(M)=\{ 0\} $.  Moreover, if $M$ is connected and compactly $\widehat{A}$--enlargeable, then the Poincar\'e dual of the total $\widehat{A}$--class of $M$ has infinite K--area, i.e.~$\widehat{\mathbf{A}}(M)\cap [M]\notin \H _*(M;\Q )$. Here, compactly enlargeable respectively compactly $\widehat{A}$--enlargeable refer to finite coverings in the definition of enlargeability (cf.~\cite{LaMi}). Using the definition of infinite K--area in \cite{Hanke}, enlargeable manifolds have infinite K--area. Moreover, we show that a closed connected spin manifold $M^n$ of positive scalar curvature satisfies $\widehat{\mathbf{A}}(M)\cap [M]\in \H _*(M;\Q )$ and $\H _{n-1}(M)=H _{n-1}(M)$ which together with the above observation can be seen as a restatement of some classic results about positive scalar curvature. Note that $\H _0(M;\Q )=\H _1(M;\Q )=\{ 0\} $ for any compact manifold, i.e.~this result covers the original obstruction of the vanishing $\widehat{A}$--genus. Analogous results have been observed by Brunnbauer and Hanke in \cite{BrHa} where the small group homology $H_*^{\mathrm{sm}(P)}(M;\Q )$ has been introduced. In fact,  one can show that $\H _*(M;\Q )\subseteq H_*^{\mathrm{sm}(P)}(M;\Q )$ if $P$ denotes the largeness property ''compactly enlargeable''. Another analogy can be found in \cite{Hanke} where a subspace of the K--group is considered. In section \ref{sec15} we introduce the relative K--area of a class $\theta \in H_k(M,A)$ where $A$ is a compact submanifold of $M$. This leads to the subgroups $\H _k(M,A)\subseteq H_k(M,A)$ of homology classes with finite relative K--area which can be useful to relate $\H _*(M)$ and $\H _*(M')$ if $M'$ is obtained from $M$ by surgery. In fact, this relative version satisfies the excision property and $\H _k(M,A)$ is isomorphic to $\H _k(M/A)$ for all $k$ if $M/A$ is a smooth manifold. 
\begin{thm*}
The functor $\H $ yields a semi homology theory on pairs of compact smooth manifolds and continuous maps in the sense that $\H $  satisfies the Eilenberg--Steenrod axioms up to exact homology sequences. 

Moreover, $\H $ determines the functor of singular homology on the category pairs of compact smooth manifolds and continuous maps.
\end{thm*}
The second assertion in the theorem seems rather unexpected considering the fact $\H _*(T^n)=\{ 0\}$. The key ingredient for this observation is that  $\theta \times [S^2]$ has finite K--area for all $\theta \in H_k(M,A)$. Hence, we can recover $H_*(M,A)$ from $\H _*(M\times S^2,A\times S^2)$ and $\H _*(M,A)$ using the K\"unneth formula.

\subsubsection*{Chapter 2: Curvature Gaps}
A \emph{curvature gap} is a nonnegative number which bounds the curvature in a certain norm from below. Of course, this number depends on the choice of the norm, but we will show that there are interesting choices related to question of topology and Yamabe invariants. Section \ref{curv_gap} treats the case of $L^\infty $--curvature gaps of Hermitian bundles in analogy to the curvature norm used for the definition of the K--area in chapter 1. In fact, let $g$ be a Riemannian metric on the compact manifold $M$ and $\bundle{E}\to M$ be a Hermitian bundle with Hermitian connection, then $\| R^\bundle{E}\| _{g}:=\max |R^\bundle{E}(v\wedge w)|_{op}$ determines a norm  where $| \, .\, |_{op}$ means the fibrewise operator norm on $\mathrm{End}(\bundle{E})$ and the maximum is taken over all orthonormal vectors $v,w\in TM$. Hence, we define the $L^\infty $--vector bundle curvature gap on a Riemannian manifold by $\cg (M_g):=\inf \| R^\bundle{E}\| _g$ where the infimum is taken over all Hermitian bundles $\bundle{E}$ with Hermitian connections such that $\bundle{E}$ is not stably rational trivial. Note that a flat bundle is always stably rational trivial, i.e.~the condition not stably rational trivial guarantees $\| R^\bundle{E}\| _{g}>0$. The main result in section \ref{curv_gap} is the following theorem:
\begin{thm*}
$\cg (M_g)>0$ respectively $\cg (M_g)=0$ depend only on the homotopy type of $M$ and $\cg (M_g)>0$ holds if and only if $\H _{2k}(M)=H_{2k}(M)$ for all $k>0$. 
\end{thm*}
Hence, a finite fundamental group is sufficient for $\cg (M_g)>0$, but not necessary as the example $M=S^3\times S^1$ shows. In the remainder of chapter 2 we consider suitable $L^{n/2}$--norms of the curvature where $n$ is the dimension of the underlying manifold. The lower bounds for the corresponding curvature gaps are invariants of conformal classes and based on the concept of \emph{modified scalar curvature} introduced by Gursky and LeBrun in \cite{GurLe2} and generalized by Itoh in \cite{Itoh}. In fact, Itoh observed that the Yamabe problem for the modified scalar curvature $\mathrm{scal}^F:=\mathrm{scal}-F$ can be treated in the same way as the original Yamabe problem. In particular, if the corresponding modified Yamabe invariant satisfies $\yam ^F(M_{[g]}^n)<\yam (S^n)$, then there is a metric $h\in [g]$ with constant modified scalar curvature. Here, $F:\met (M)\to C^{0,\alpha }(M)$ satisfies $F(\rho \cdot g)=F(g)/\rho $ for all metrics $g$ and smooth functions $\rho :M\to (0,\infty )$. In section \ref{sec234} we introduce the \emph{conformal curvature gap}
\[
\ccg (M_{[g]}):=\left[ \inf _{\mathrm{ch}(\bundle{E})\neq \mathrm{rk}(\bundle{E})}\int |R^\bundle{E}|_{g,op}^{n/2}\cdot \mathrm{vol}_g\right] ^{2/n}
\] 
of the conformal manifold $(M^n,[g])$ where $|R^\bundle{E}|_{g,op}$ is a function on $M$ determined by a suitable operator norm on the fibers of $\Lambda ^2T^*M\otimes \mathrm{End}(\bundle{E})$, and the infimum is taken over all Hermitian bundles which are not stably rational trivial. In $2$--dimensions the conformal curvature gap does not depend on the conformal class, i.e.~it is a topological invariant. Using the result about modified scalar curvature shows $\ccg (S^2)=2\pi $ whereas K--area methods yield $\ccg (N)=0$ for all orientable closed surfaces $N$ with $\chi (N)\leq 0$. In $4$--dimension we obtain a topological invariant by taking the supremum over all conformal classes. In particular, $\ccg (M)=\sup _{[g]}\ccg (M_{[g]})$ turns out to be a finite value if $\dim M=4$. In dimension $n>4$ finiteness of $\ccg (M)$ is an open problem except in the case $\ccg (M_{[g]})=0$. In order to compute a few estimates for $\ccg (M)$ we introduce the \emph{curvature gap volume}
\[
\cgv (M)=\sup _{[g]}\left[ \inf _{\left< \mathrm{ch}(\bundle{E}),[M]\right> \neq 0}\int |R^\bundle{E}|_{g,op}^{n/2}\cdot \mathrm{vol}_g\right] ^{2/n}
\]
which is a finite topological invariant of closed $4$--manifolds. $\cgv (M^4)<\infty $ follows from the existence of self dual connections on certain $\mathrm{SU}(2)$--bundles proved by Taubes in \cite{Taub}. Moreover, $\cgv (M)$ is positive if $M^4$ is a spin manifold which admits a metric of positive scalar curvature, in fact $\yam (M)\leq 4\cgv (M)$ with equality for $S^4$, $S^3\times S^1$ and $T^4$ of course. Section \ref{sec244} is devoted to the computation of estimates respectively precise values for $\ccg (M)$ and $\cgv (M)$.

In the last section of chapter 2 we consider gap problems of the Weyl curvature. These gaps follow again by the method of modified scalar curvature. In $4$--dimension there are plenty of $L^2$--bounds for the Weyl curvature (for instance \cite{ABKS,Gur1,Itoh,LeB3}), but to our knowledge in dimension $n>4$ there are no precise results on this subject. In \cite{ABKS} Akutagawa et al.~prove that for each $\epsilon >0$ and $C>0$ there is a conformal class $[g]$ on $M$ with
\[
\int |W|_g^{n/2}\cdot \mathrm{vol}_g\geq C\qquad \text{and}\quad  \yam (M_{[g]}) \geq \yam (M)-\epsilon .
\]
We will show a lower bound of $\int |W|_g^{n/2}\mathrm{vol}_g$ in terms of $\yam (M_{[g]})$ if $M^n$ is even dimensional with Betti number $b_{n/2}> 0$. In fact, this yields certain estimates of the Weyl invariants introduced in \cite{ABKS}. In order to conclude optimal results we use a slightly different $L^{n/2}$--norm of the Weyl tensor (of course these are equivalent to the above norm and depend only on the conformal class). In fact with our choice, we obtain optimal estimates for locally symmetric Einstein spaces of nonnegative scalar curvature. The proof of the following theorem is based on a precise estimate of the curvature operator which appears in the Bochner--Weitenb\"ock formula on $m$--forms. 
\begin{thm*}
Suppose that $M^{2m}$ is a closed oriented manifold with Betti number $b_m>0$. Then for any conformal class $[g]$ on $M$
\[
\yam (M_{[g]})\leq 2m(2m-1)\left[ \int _M |\lambda _-(W)|^m\cdot \mathrm{vol}_g\right] ^{1/m}
\]
with equality if $(M,g)$ is a locally symmetric Einstein space with $\mathrm{scal}_g\geq 0$. Here, $\lambda _-(W)$ denotes the pointwise minimal eigenvalue of the Weyl curvature $W:\Lambda ^2TM\to \Lambda ^2TM$.
\end{thm*} 

\subsubsection*{Chapter 3: Upper bounds of scalar curvature}
In the third chapter we consider the problem of ''maximal'' scalar curvature on a given manifold. If $M^n$ does not admit a metric of positive scalar curvature, the Yamabe invariant $\yam (M)$ is an upper bound of the scalar curvature for metrics of unit volume, i.e.~there is no metric $g$ with $\mathrm{scal}_g>\yam (M)\cdot \mathrm{Vol}(M,g)^{-n/2}$ and assuming ''$\geq $'' implies equality as well as $g$ is Einstein. This fails in general for manifolds with $\yam (M)>0$ which means that the scalar curvature of metrics with unit volume can be arbitrary large. Hence, in order to get an analogous statement for manifolds with $\yam (M)>0$, we consider Riemannian functionals $\mu :\met (M)\to \R $ with the properties
\begin{itemize}
\item $\mu (c\cdot g)=c^{-1}\mu (g)$ for all $c>0$ and $g\in \met (M)$.
\item $\mathrm{scal}_g\geq \mu (g)$ implies $\mathrm{scal}_g=\mu (g)$. 
\end{itemize}
These functionals are called \emph{upper bounds of the scalar curvature}. A simple example of an upper bound on spin manifolds is determined by Gromov's K--area. Other important examples are provided by \emph{area extremal metrics} which were introduced in \cite{Gr01,Llarull}. In section \ref{sec_functionals} we formalize these constructions and review some results about area extremal metrics proved by Llarull \cite{Llarull} and Goette, Semmelmann \cite{GoSe1,GoSe2}. The remainder of chapter 3 treats the question of \emph{conform area extremal metrics} which generalizes the concept of area extremality. In fact, a metric $g_0$ on $M$ is conform area extremal if $g\geq \rho \cdot g_0$ on $\Lambda ^2TM$ and $\mathrm{scal}_g\geq \frac{1}{\rho }\mathrm{scal}_{g_0}\geq 0$ imply $\mathrm{scal}_g=\frac{1}{\rho }\mathrm{scal}_{g_0} $ for all $g\in \met (M)$ and $\rho :M\to (0,\infty )$. The first examples were found in \cite{List9} where we proved the conform area extremality of locally symmetric spaces with positive Ricci curvature and nontrivial Euler characteristic. In section \ref{sec_conf_extremal} we show how conform area extremality can be described by modified Yamabe invariants and moreover, we give extensions of the main results in \cite{List9}. This is done using K--area methods as well as an extension of the usual Bochner argument to almost nonnegative curvature operators (cf.~proposition \ref{lem1342}). In section \ref{strict_conf} we obtain by analogous techniques the conform area extremality of space forms in odd dimensions. The difficulty in this case arises for manifolds with trivial Euler characteristic and trivial Kervaire semi--characteristic because the usual Dirac operator has trivial index in these situations. 
\begin{thm*}
Let $(M^n,g_0)$, $n\geq 3$, be a closed spin manifold of constant sectional curvature $K_{g_0}>0$, then $g_0$ is strict conform area extremal.
\end{thm*}
Here, we say that $g_0$ is \emph{strict conform area extremal} if $g\geq \rho \cdot g_0$ on $\Lambda ^2TM$ and $\mathrm{scal}_g\geq \frac{1}{\rho }\mathrm{scal}_{g_0}$ imply $\rho =\mathrm{const}$ and $g=\rho \cdot g_0$. In order to prove the theorem we need a precise value for the K--area of odd dimensional positive space forms which is obtained performing a $0$--dimensional surgery on $S^{2n}$. Since the twisting bundles are nonsymmetric we need again a generalized Bochner argument.

\subsubsection*{Pre--Published results}
The first chapter has been published in \cite{List10} with some minor modifications. For instance, we improved the estimate in proposition \ref{prop_scalar} which was necessary for the main result in section \ref{strict_conf}. Moreover, we added proposition \ref{prop162}, lemma \ref{lem163}, remark \ref{rem_covering}, remark \ref{rem151},  example \ref{exam136}, section \ref{versus} and a small discussion of the intersection product to the original work. The main difference of the first chapter to \cite{List10} is indeed the result that the K--area homology determines the singular homology after stabilizing with $S^2$ which follows from lemma \ref{lem163} and the arguments in section \ref{versus}. The results in the second chapter are new observations and extensions of previous works with the exception of section \ref{sec2} which collects and adapts known facts about modified scalar curvature to our situation. In the third chapter we summarize some known results about upper bounds for the scalar curvature and extend previous results. For instance, proposition \ref{main_prop} and its proof can be found in \cite{List9}, but theorem \ref{area_thm} is an extension of the main result in \cite{List9}. Moreover, theorem \ref{thm343} is a new observation and is proved using K--area methods as well as a generalized Bochner technique provided by proposition \ref{lem1342}.

\subsubsection*{Acknowledgements}
I am very grateful to my advisor Sebastian Goette for his support in any kind of questions and problems. Moreover, I would like to thank Jan Schl\"uter for conversations in matters of topology. Last but not least I am grateful to the DFG--Schwerpunktprogramm ''Globale Differentialgeometrie'' and the SFB Tr 71 for financial and travel support.

\chapter{Homology of finite K--area}
\label{chp2}

\section{The K--area of a homology class}
\label{sec21}
Let $M_g:=(M,g)$ be a compact Riemannian manifold possibly disconnected and with nonempty boundary. In order to obtain the additivity axiom  for the K--area homology and in view of the topological K--theory the fiber dimension of a vector bundle on $M$ is not assumed to be a global constant, it is only constant on connected components. Suppose that $\mathrm{BU}_n$ is the classifying space of $\mathrm{U}(n)$, then $n$--dimensional Hermitian vector bundles on $M$ are classified up to isomorphism by $[M,\mathrm{BU}_n]$.  Hence, $[M,\BU ] $ classifies finite dimensional Hermitian vector bundles on $M$ up to isomorphism where $\BU =\coprod _n\mathrm{BU}_n$  is the disjoint (topological) sum. Note that we do not consider the stabilized picture of vector bundles which means that vector bundles have a definite rank and our $\BU $ differs from the usual space in the literature. If $\bundle{E}\to M$ is a (smooth) Hermitian vector bundle with Hermitian connection, we denote by $\rho ^\bundle{E}:M\to \BU $ the classifying map and define
\[
\| R^\bundle{E}\| _g:=\max _{x\in M}\max _{v, w\in T_xM}\frac{|R^\bundle{E}(v\wedge w)|_{op}}{|v\wedge w |_g}
\] 
where $R^\bundle{E}:\Lambda ^2TM\to \mathrm{End}(\bundle{E})$ means the curvature of $\bundle{E}$ and $|\, .\, | _{op}$ is the pointwise operator norm on $\mathrm{End}(\bundle{E})$. It is quite essential for the theory to use the operator norm on $\mathrm{End}(\bundle{E})$ because the equivalence of norms on finite dimensional vector spaces includes a constant usually depending on the dimension and in case of infinite K--area, the rank of the interesting bundles tends to infinity. Suppose $\theta \in H_{2*}(M;\group )$ for a coefficient group $\group $ where omitting the coefficient group means as usual $G=\Z $. Then $\mathscr{V}(M;\theta )$ denotes the set of all Hermitian bundles $\bundle{E}\to M$ endowed with a Hermitian connection such that the image of $\theta $ under the induced homomorphism $\rho ^\bundle{E}_*:H_*(M;\group )\to H_*(\BU ;\group )$ is nontrivial: $\rho ^\bundle{E}_*(\theta )\neq 0$. Since the classifying map is uniquely determined up to homotopy, $\mathscr{V}(M;\theta )$ does not depend on the choice of $\rho ^\bundle{E}$. The $\Z $--cohomology ring of $\mathrm{BU}_n$ is a $\Z $--polynomial ring  generated by the Chern classes which supplies the following alternative definition of $\mathscr{V}(M;\theta )$. Suppose that $M$ is connected, $2n\geq \dim M$ and $\theta \in H_{2*}(M)$, then $\mathscr{V} (M;\theta )$ is the set of Hermitian bundles $\bundle{E}\to M$ endowed with a Hermitian connection such that there is a polynomial $p\in \Z [c_1,\ldots ,c_n] $ with $\left< p(c(\bundle{E})),\theta \right> \neq 0$ where $c(\bundle{E})$ is the total Chern class of $\bundle{E}$ and
\[
p(c(\bundle{E})):=p(c_1(\bundle{E}),\ldots ,c_n(\bundle{E}))\in H^{2*}(M;\Z  ).
\]
Of course, if $\bundle{E}\in \mathscr{V} (M;\theta )$, then $\left< p(c(\bundle{E})),\theta \right> \neq 0$ for a monomial $p$, i.e.~there is a nonvanishing $\theta $--Chern number:
\[
\left< c_{1}(\bundle{E})^{j_1}\cdots c_{n}(\bundle{E})^{j_n},\theta \right> \neq 0
\]
for certain nonnegative integers $j_1\ldots ,j_n$. The equivalence of the two descriptions is a simple exercise in algebraic topology because $H^*(\mathrm{BU}_m;\Z )=\Z [c_1,\ldots ,c_m]$, $m=\mathrm{rk}_\C (\bundle{E})$, yields the nondegeneracy of the pairing $\left< .,.\right> $ and $\rho ^\bundle{E}_*(\theta )\neq 0$ implies the existence of a characteristic class $u\in H^*(\mathrm{BU}_m;\Z )$ with $\left< u,\rho ^\bundle{E}_*(\theta )\right> \neq 0$, i.e.~we choose $p(c(\bundle{E}))=(\rho ^\bundle{E})^*u\in H^{2*}(M;\Z )$.  

If $\mathscr{V} (M;\theta )\neq \emptyset $, the \emph{K--area of a compact Riemannian manifold $M_g=(M,g)$ w.r.t.~the homology class} $\theta \in H_{2*}(M;\group )$ is defined by
\begin{equation}
\label{defn_k}
\ke (M_g;\theta ):=\left( \inf _{\bundle{E}\in \mathscr{V}(M;\theta )}\| R^\bundle{E}\| _g \right) ^{-1}\in (0,\infty ].
\end{equation}
Moreover, we define for monotonicity reasons $\ke (M_g;\theta )=0$ in case $\mathscr{V} (M;\theta )=\emptyset $ (for instance $\theta =0$). We will frequently use this definition to introduce various K--areas by taking the infimum over different sets of vector bundles. 
\begin{rem}
Because $M$ is compact, any element of the K--group $K(M)$ can be represented by $[\bundle{E}]-[\C ^N]$ where $\bundle{E}$ is a complex vector bundle and $\C ^N$ is a flat bundle on $M$ and moreover, the Chern character map $\mathrm{ch}:K(M)\otimes \Q \to H^{2*}(M;\Q )$ is an isomorphism (cf.~\cite{AH}). Thus,  $\mathscr{V}(M;\theta )$ is nonempty if $\theta $ is a nontrivial element in $H_{2*}(M;\Q )$. 
\end{rem}
Since $\mathscr{V}(M;a\cdot \theta )= \mathscr{V}(M;\theta )$ is independent on the choice of the metric, we conclude the following scaling invariance of the K--area:
\[
\ke (M_{\mu \cdot g};a\cdot \theta )= \mu \cdot \ke (M_g;\theta )
\]
where $\mu $ is a positive constant and $a\in \Z \setminus \{ 0\} $. In order to extend the above definition to odd homology classes, we add large circles. In fact, suppose $ \theta \in H_{2*+1}(M;\group )$ then $\theta \times [S^1 ] \in H_{2*}(M\times S^1;\group )$ for a fundamental class of $S^1$. Thus, we define
\[
\ke (M_g;\theta ):=\sup _{\mathrm{d}t^2}\ke (M_g\times S^1_{\mathrm{d}t^2};\theta \times [ S^1] )
\]
where the supremum runs over all line elements $\mathrm{d}t^2$ of $S^1$ and $M_g\times S^1_{\mathrm{d}t^2}$ is endowed with the product metric. Note that the K--area on the right hand side does not depend on the choice of the fundamental class $[S^1]$. Replacing $\mathrm{d}t^2$ by $\mu \cdot \mathrm{d}t^2=\mathrm{d}\tilde t^2$ we obtain the above scaling invariance for odd homology classes. The K--area of a general class $\theta \in H_*(M;\group )$ is given by
\[
\ke (M_g;\theta ):=\max \{ \ke (M_g;\theta _\mathrm{even}),\ke (M_g;\theta _\mathrm{odd})\} .
\]
In case $\theta =[ M] $ we omit $\theta $ and simply write $\ke (M_g)$ which is the total K--area of $M$ introduced by Gromov in \cite{Gr01} and considered in \cite{Dav,pre_Lima,Entov,Polt,pre_Saval}. 
\begin{prop}
\label{proposition1}
Suppose $\theta =\sum _i\theta _i$ with $\theta _i\in H_{i}(M ^n;\group )$, then
\[
\ke (M_g;\theta )=\max \{ \ke (M_g;\theta _i)\ |\ i=0\ldots n\} .
\] 
Moreover, we obtain for $\theta ,\eta \in H_k(M ;\group )$
\[
\ke (M_g;\theta +\eta )\leq \max \{ \ke (M_g;\theta ),\ke (M_g;\eta )\} .
\]
\end{prop}
\begin{proof}
We start with the case $\theta \in H_{2*}(M;\group )$. Observe that $\mathscr{V}(M;\theta _i)\subseteq \mathscr{V}(M;\theta )$ because $0\neq \rho ^\bundle{E}_*(\theta _i) \in H_i(\BU ;\group )$ implies $\rho _*^\bundle{E}(\theta )\neq 0$. Conversely, if $\rho _*^\bundle{E}(\theta )\neq 0$, then $\rho ^\bundle{E}_*(\theta _i)\neq 0$ for at least one $\theta _i$ which completes the proof for even homology classes: $\mathscr{V}(M;\theta )=\bigcup _i\mathscr{V}(M;\theta _i)$. For the second statement, we obtain by the same argument $\mathscr{V}(M;\theta +\eta )\subseteq \mathscr{V}(M;\theta )\cup \mathscr{V}(M;\eta )$ which shows the inequality if $k$ is even. Now suppose that $\theta \in H_{2*+1}(M;\group )$, then:
\[
\begin{split}
\ke (M_g;\theta )&=\sup _{\mathrm{d}t^2}\max \{ \ke (M_g\times S^1_{\mathrm{d}t^2};\theta _i\times [S^1])\ |\ i=1\ldots n\} \\
&=\max \Bigl\{ \sup _{\mathrm{d}t^2}\ke (M_g\times S^1_{\mathrm{d}t^2};\theta _i\times [S^1])\ |\ i=1\ldots n \Bigl\} .
\end{split}
\] 
The general case is an easy consequence. 
\end{proof}
Since $M$ is compact, for any two Riemannian metrics $g$ and $h$ on $M$ there is a constant $0<C=C(g,h)<+\infty $ such that $C^{-1}\cdot \| R^\bundle{E}\| _g\leq \| R^\bundle{E}\| _h\leq C\cdot \| R^\bundle{E}\| _g$ for all bundles $\bundle{E}$. Hence, the condition $\ke (M_g;\theta )<\infty $ does not depend on the choice of the Riemannian metric on $M$ which yields the following: For each $j$, the set
\[
\H _j(M;\group ):=\{ \theta \in H_j(M;\group )\ | \ \ke (M_g;\theta )<\infty \} \subseteq H_j(M;\group )
\]
is  a subgroup independent on the choice of the metric $g$ and satisfies
\[
\H _*(M;\group )=\{ \theta \in H_*(M ;\group )\ | \ \ke (M_g;\theta )<\infty \} =\bigoplus \H _j(M;\group ).
\]
If $\theta \in H_0(M;\group )$ does not vanish, there are trivial bundles $\bundle{E}$ with $\rho ^\bundle{E}_*(\theta )\neq 0$. In this case we use that the fiber dimension of vector bundles is not assumed to be globally constant. Because trivial bundles admit flat connections, the K--area of $0\neq \theta \in H_0(M;\group )$ is infinite which implies $\H _0(M;\group )=\{ 0\} $. Moreover, if $\group $ is a ring, $\mathscr{V}(M;a\cdot \theta )\subseteq \mathscr{V}(M;\theta )$ for all $\theta \in H_{2*}(M;G)$ yields
\[
\ke (M_g;a\cdot \eta )\leq \ke (M_g;\eta )
\]
for all $\eta \in H_*(M;\group )$ and $a\in \group $. Since $H_*(\mathrm{BU}_n;\group )$ is a free $\group $--module, this is an equality if $a\neq 0$ and $\group $ has no zero divisors. Hence, for any coefficient ring $\group $, $\H _j(M;\group )$ is a $\group $--submodule of $H _j(M;\group )$. 

Let $M\coprod N$ be the disjoint sum of $M$ and $N$, then for any $\theta \in H_{2*}(M\coprod N;G)$ the interesting bundles in $\mathscr{V}(M\coprod N;\theta )$ are determined by $\mathscr{V}(M;\theta _{|M})\cup \mathscr{V}(N;\theta _{|N})$  where $\theta _M$ and $\theta _N$ mean the restriction of $\theta $ to $M$ and $N$. Thus, the K--area of $\theta \in H_k(M\coprod  N;G)$ equals the maximum of the K--area of $\theta _M$ and the K--area of $\theta _N$ which proves the additivity axiom
\[
\H _k\left( M\coprod N;G\right) \cong \H _k(M;G)\oplus \H _k(N;G).
\]
\begin{prop}
\label{proposition2}
Let $f:(M,g)\to (N,h)$ be a smooth map with $g\geq f^*h$ on $\Lambda ^2TM$, then
\[
\ke (M_g;\theta )\geq \ke (N_h;f_*\theta ).
\]
In fact, for each continuous map $f:M\to N$, the induced homomorphism on singular homology yields homomorphisms $f_*:\H _j(M;\group )\to \H _j(N;\group )$.

Moreover, if $M\stackrel{i}{\hookrightarrow} N$ is a compact submanifold and a retract, then for suitable metrics $h$ on $N$
\[
\ke (M_h;\theta )=\ke (N_h;i_*\theta ),
\]
here suitable means that the smooth retraction map $r:N\to M$ is $1$--Lipschitz, i.e.~$r^*h_{|M}\leq h$ on $\Lambda ^2TN$. 
\end{prop}
\begin{proof}
We start with the case $\theta \in H_{2*}(M;\group )$. Since
\[
\bigl( \rho ^{f^*\bundle{E}}\bigl) _*(\theta )=\left( \rho ^\bundle{E}\circ f\right) _*(\theta )=\rho ^\bundle{E}_*(f_*\theta ),
\]
the pull back of vector bundles yields a map $f^*:\mathscr{V}(N;f_*\theta )\to \mathscr{V}(M;\theta )$. Moreover, $g\geq f^*h$ on $\Lambda ^2TM$ supplies
\[
\| R^{f^*\bundle{E}}\| _g\leq \| R^\bundle{E}\| _h
\]
which proves the inequality. If $\theta \in H_{2*+1}(M;\group )$ we consider the map $f\times \mathrm{id}:M\times S^1\to N\times S^1$ and apply the case for even homology classes. The necessary inequality on $\Lambda ^2T(M\times S^1)$ follows from the compactness of $M$ because for any line element $\mathrm{d}t^2$, there is some $\mathrm{d}\tilde t^2$ such that $g\oplus \mathrm{d}\tilde t^2\geq  (f\times \mathrm{id})^*(h\oplus \mathrm{d}t^2)$ on $\Lambda ^2T(M\times S^1)$. The second observation is proved for smooth $f:M\to N$ by the inequality and in case of continuous $f$ we use the smooth approximation theorem and the homotopy invariance of the induced homomorphism. For the last statement it remains to show ''$\leq $'' but this follows by considering a smooth retraction map $r:N\to M$ (using the smooth approximation theorem: every retract $M\subseteq N$ for compact manifolds $M$ and $N$ is also a smooth retract). 
\end{proof}

\begin{thm}
$\H _*(\, .\, ;G)$ is a functor on the category of compact smooth manifolds and continuous maps into the category of graded abelian groups which satisfies the homotopy, dimension and additivity axiom. In fact, $\H _*(M;G)$ depends only on the homotopy type of $M$. 
\end{thm}
Gromov proved in \cite{Gr01} that the total K--area of simply connected manifolds and spin manifolds of positive scalar curvature is finite which means $\H _n(M^n)=H_n(M) $ for these closed manifolds. Furthermore, using the above proposition and the observation that $H_{2j}(\C P^n )$ is generated by the fundamental class of $\C P^j\subseteq \C P^n$ we conclude for the complex projective spaces
\[
\H _k(\C P^n)=\biggl\{ \begin{array}{cl}
\Z &if \ \ k\in \{ 2,4,\ldots ,2n\}\\
0& otherwise.
\end{array}
\]
Conversely, every closed connected orientable surface $M$ of positive genus has infinite total K--area, i.e.~$\H _2(M)=\{ 0\}$ (cf.~\cite{Gr01}). If $\H _k(N)$ is known, the previous proposition is one of best ways to compute $\H _k(M)$ by considering maps $M\to N$ respectively $N\to M$.  
\begin{cor}
The Hurewicz homomorphism satisfies for all $j\geq 2$:
\[
h_j:\pi _j(M)\to \H _j(M)\subseteq H_j(M).
\]
\end{cor}
\begin{proof}
Let $f:(S^j,e)\to (M,x)$ be a representative of $\alpha \in \pi _j(M,x)$, then
\[
\infty > \ke (S^j_g;[S^j])\geq \ke (M_{\overline{g}};f_*[S^j])=\ke (M_{\overline{g}};h_j(\alpha ))
\]
for suitable metrics $g$, $\overline{g}$ on $S^j$ and $M$. 
\end{proof}

\begin{prop}
\begin{enumerate}
\item $\mathrm{Tor}(H_*(M))\subseteq \H _*(M)$.
\item $\H _*(M;\Q )\cong \H _*(M)\otimes \Q $.
\item $\H _1(M;\Q )=\{ 0\} $, i.e.~$\H _1(M)=\mathrm{Tor}(H_1(M))$.
\end{enumerate}
\end{prop}
\begin{proof}
Since $H_*(\mathrm{BU}_n)$ is free, each induced homomorphism $H_*(M)\to H_*(\BU )$ maps torsion classes to $0$. Hence, for any $\theta \in \mathrm{Tor}(H_*(M))$, $\mathscr{V}(M;\theta )=\emptyset $ and $\mathscr{V}(M\times S^1;\theta \times [S^1])=\emptyset $ supply $\ke (M_g;\theta )=0$, i.e.~$\theta \in \H _*(M)$. For the second claim we consider the commutative diagram
\[
\begin{xy}
\xymatrix{H_{2k}(M)\otimes \Q \ar[d]^{\lambda_M }\ar[r]^{\rho _*^\bundle{E}\otimes \mathrm{id}} &H_{2k}(\BU )\otimes \Q \ar[d]^\lambda \\
H_{2k}(M;\Q )\ar[r]^{\rho _*^\bundle{E}}&H_{2k}(\BU ;\Q )}
\end{xy}
\]
where $\lambda $ and $\lambda _M$ are isomorphism. Hence, $\mathscr{V}(M;\theta )=\mathscr{V}(M;\lambda _M(\theta \otimes x))$ for all $\theta \in H_{2k}(M)$ and $x\in \Q \setminus \{ 0\}$ prove that $\lambda _M:\H _{2k}(M)\otimes \Q \to \H _{2k}(M;\Q )$ is an isomorphism. In order to see $\H _{2k+1}(M)\otimes \Q \cong \H _{2k+1}(M;\Q )$ replace $M$ by $M\times S^1$ in the above diagram. Suppose $\theta \in H_1(M;\Q )$ is nontrivial, then there is some $\alpha $ in the lattice $H^1(M;\Z )\subseteq  H^1(M;\Q )$ with $\left< \alpha ,\theta \right> \neq 0$. Moreover, by the Hopf theorem and the smooth approximation theorem there is a smooth map $f:M\to S^1$ uniquely determined up to homotopy with $f^*\omega =\alpha $ where $\omega \in H^1(S^1;\Z )$ is the orientation class.  Consider the map $f\times \mathrm{id}:M\times S^1\to T^2$, then
\[
\left< \omega ,f_*\theta \right> =\left< f^*\omega ,\theta\right> =\left< \alpha ,\theta \right> \neq 0
\]
proves $f_*\theta \neq 0$ which yields $[T^2]=x\cdot f_*\theta \times [S^1]$ for some $x\in \Q \setminus \{ 0\} $. Hence, the scaling invariance in the homology class and proposition \ref{proposition2} show
\[
\ke (M_g\times S^1_{ \mathrm{d}t^2};\theta \times [S^1])\geq \ke (T^2_0;[T^2])=\infty .
\]
where $T_0^2$ is a suitable flat torus ($T^2$ has infinite total K--area by the remarks below or use the result in \cite{Gr01}).
\end{proof}
Although the torsion subgroup remains unchanged for $\Z $ coefficients, this does not have to hold for arbitrary coefficient groups. For instance, consider the real projective space $\R P^n$ in dimension $n\geq 2$ and let $\bundle{E}$ be the canonical complex line bundle with $0\neq c_1(\bundle{E})\in H^2(\R P^n;\Z )= \Z _2 $. The image of $c_1(\bundle{E})$ in $H^2(\R P^n;\Z _2)$ is the second Stiefel--Whitney class $w_2(\bundle{E})$ which is nontrivial. Hence, if $\theta $ denotes the generator of $ H_2(\R P^n;\Z _2)$, $\left< w_2(\bundle{E}),\theta \right> \neq 0$ implies $\bundle{E}\in \mathscr{V}(\R P^n;\theta )$. But the real Chern class of $\bundle{E}$ vanishes which means that $\bundle{E}$ admits a flat connection. This shows $\ke (\R P^n;\theta )=\infty $ and $\H _2(\R P^n;\Z _2)=\{ 0\} $ whereas  $H _2(\R P^n;\Z _2)=\Z _2$.   

\section{K--area for the Chern character}
In the above definition of the K--area we considered Hermitian bundles which have a nontrivial $\theta $--Chern number. However, we can also restrict to bundles with a specific nontrivial $\theta $--Chern number respectively fix a characteristic class $u\in H^*(\BU ;\group )$ and consider bundles with $\left< u(\bundle{E}),\theta \right> \neq 0$ where $u(\bundle{E}):=(\rho ^\bundle{E})^*u$. This leads to the $\mathrm{K} _u$--area and results similar to the one presented above with a few exceptions. But we still obtains subgroups $\H _k(M;u)\subseteq H_k(M;\group )$ of homology classes with finite $\mathrm{K}_u$--area and moreover, a continuous map $f:M\to N$ induces homomorphisms $f_*:\H _k(M;u)\to \H _k(N;u )$. Hence, each characteristic class $u\in H^*(\BU ;G)$ determines a functor $\H _*(\, .\, ;u)$ on the category of compact smooth manifolds and continuous maps into the category of graded abelian groups which satisfies the homotopy, dimension and additivity axiom. 

Another important K--area is the K--area for the Chern character. This area is particularly interesting for scalar curvature results. If $\theta \in H_{2*}(M;\Q )$, we denote by $\mathscr{V}_\mathrm{ch} (M;\theta )$ the set of Hermitian bundles $\bundle{E}\to M$ endowed with a Hermitian connection such that $\left< \mathrm{ch}(\bundle{E}),\theta \right> \neq 0$ for the Chern character $\mathrm{ch}(\bundle{E})$. As already remarked $\mathscr{V}_\mathrm{ch} (M;\theta )$ is nonempty if $\theta \neq 0$, i.e.~we define $\ke _\mathrm{ch}(M_g;\theta )$ as above but take the infimum in (\ref{defn_k}) over bundles in $\mathscr{V}_\mathrm{ch} (M;\theta )$. We set $\ke _\mathrm{ch}(M_g;0 )=0$ and add large circles to define the $\mathrm{K} _\mathrm{ch}$--area for odd homology classes:
\[
\ke _\mathrm{ch}(M_g;\theta ):=\sup _{\mathrm{d}t^2}\ke _\mathrm{ch}(M_g\times S^1_{\mathrm{d}t^2};\theta \times [S^1]),\qquad \theta \in H_{2*+1}(M;\Q ).
\]
The $\mathrm{K} _\mathrm{ch}$--area of a general class is the maximum of the even and the odd part. Since $\mathscr{V}_\mathrm{ch}(M;\theta )\subseteq \mathscr{V}(M;\theta )$, we conclude
\[
\ke _\mathrm{ch}(M_g;\theta )\leq \ke (M_g;\theta )
\]
for all $\theta \in H_*(M;\Q )$. Equality seems to hold only in few examples, for instance $\ke _\mathrm{ch}(M_g)=\ke (M_g)$ if $M$ is homeomorphic to $S^n$ or $\ke _\mathrm{ch}(M_g;\theta )=\ke (M_g;\theta )$ if $\theta \in H_2(M;\Q )$.  Thus, if $\H _j(M;\mathrm{ch})$ denotes the set of all homology classes $\theta \in H_j(M;\Q )$ with $\ke _\mathrm{ch}(M_g;\theta )<\infty $, then $\H _j(M;\Q )\subseteq \H _j(M;\mathrm{ch})$ are linear subspaces of $H_j(M;\Q )$ for each $j$.

\begin{prop}
\label{proposition6}
Suppose $\theta \in H_{2*}(M;\Q )$ is of total degree $2m>0$, then
\[
\ke (M_g;\theta )\leq m^2\cdot \ke _\mathrm{ch}(M_g;\theta ).
\]
In particular, $\H _j(M;\Q )=\H _j(M;\mathrm{ch})$ for each $j$.
\end{prop}
\begin{proof}
The proof follows mainly the idea in \cite{Gr01}, only a couple of changes are necessary to compute the precise estimate. We show that for any bundle $\bundle{E}\in \mathscr{V}(M;\theta )$ there is an associated bundle $\bundle{X}\in \mathscr{V}_\mathrm{ch} (M;\theta )$ in such a way that for all $x,y\in TM$:
\[
| R^\bundle{X}(x\wedge y )| _{op}\leq m^2\cdot | R^\bundle{E}(x\wedge y )| _{op}.
\]
Then the above inequality is an easy consequence. We assume without loss of generality $\theta _0=0$, because otherwise the K--and $\mathrm{K}_\mathrm{ch}$--area are infinite and the statement is trivial. Essential to us will be the fact that the Adams operation $\psi _k$  applied to a vector bundle is nothing but a polynomial in the exterior powers of this bundle. In fact, $\psi _k(\bundle{E})\in K(M) $ is a $\Z $--linear combination of
\[
\Lambda ^\alpha \bundle{E}=\Lambda ^{\alpha _1}\bundle{E}\otimes \cdots \otimes \Lambda ^{\alpha _l}\bundle{E}
\]
where $\alpha \in \N ^l_0$ is a multi index with $|\alpha |=k$. The $\Z $--coefficient for $\alpha $ in this linear combination is simply determined by the corresponding coefficient in the representation of the $k$th power sum by elementary symmetric polynomials.   

\emph{Claim 1:} Suppose $\bundle{E}$ is a vector bundle with $\left< \mathrm{ch}(\Lambda ^\alpha \bundle{E}),\theta \right> =0$ for all multi indexes $\alpha $ with $1\leq |\alpha |\leq m$, then
\[
\left< \mathrm{ch}_i(\bundle{E}),\theta _{2i}\right> =0
\]
holds for all $i$. Since $\theta _{2i}=0$ for $i>m$, it suffices to prove the case $i\leq m$. This can be seen by using the Adams operation $\psi _k$ for $k\in \{ 1,\ldots ,m\}$. Under the assumption on $\bundle{E}$, $\left< \mathrm{ch}(\psi _k\bundle{E}),\theta \right> $ vanishes for all $k\in \{ 1,\ldots ,m\} $. Since $\mathrm{ch}(\psi _k\bundle{E})=\rho _k(\mathrm{ch}(\bundle{E}))$ with $\rho _k=k^j$ on $\mathrm{H}^{2j}(M,\Q )$, we obtain
\[
0=\left< \mathrm{ch}(\psi _k\bundle{E}),\theta \right> =\sum _{i=1}^mk^i\left< \mathrm{ch}_i(\bundle{E}),\theta _{2i}\right> =\left< a_k,b\right> _{Euc}
\]
for all $k\in \{ 1,\ldots ,m\} $ where
\[
a_k=(k^1,k^2,\ldots ,k^m)\quad \text{and}\quad b=\bigl(\left< \mathrm{ch}_1(\bundle{E}),\theta _2\right> ,\ldots ,\left< \mathrm{ch}_m(\bundle{E}),\theta _{2m}\right> \bigl) .
\]
However, the vectors $a_ 1,\ldots ,a_m$ form a basis of $\Q ^m$ which implies that $b=0$.

\emph{Claim 2:} Let $\bundle{E}$ be a vector bundle and fix some $i\in \{ 1,\ldots ,m\} $. If $\left< \mathrm{ch}_i(\Lambda ^\alpha \bundle{E}),\theta _{2i}\right> $ vanishes for all multi indexes $\alpha =(k_1,\ldots ,k_l)\in \N ^l$ with $|\alpha |=\sum k_j=i$, then for each polynomial $p$: $\left< p(c(\bundle{E})),\theta _{2i}\right> =0$. We use again the Adams operation in order to see this. Define
\[
B_i^\alpha :=\left< \mathrm{ch}_i(\psi _\alpha (\bundle{E})),\theta _{2i}\right> =\left< \mathrm{ch}_i(\psi _{k_1}(\bundle{E})\otimes \cdots \otimes \psi _{k_l}(\bundle{E})),\theta _{2i}\right>  
\]
for multi indexes $\alpha =(k_1,\ldots ,k_l)\in \N ^l$ with $|\alpha |=\sum k_j=i$. Since $\psi _{\alpha }(\bundle{E})$ is a linear combination of $\Lambda ^\beta \bundle{E}$ with $|\beta |\leq i$, $B^\alpha _i$ vanishes for all $\alpha $ under the above assumption on $\bundle{E}$. We will show that $B^\alpha _i=0$ for all $\alpha $ with $|\alpha |=i$ implies
\[
\left< \prod _{j=1}^i\mathrm{ch}_j(\bundle{E})^{\beta _j},\theta _{2i}\right> =0
\]
for all $\beta _1,\ldots ,\beta _i\in \N _0$ with $i=\sum j\cdot \beta _j$. Since the Chern classes of degree $\leq j$ are polynomials in the Chern characters $\mathrm{ch}_1,\ldots ,\mathrm{ch}_j$, this concludes for all polynomials $p$: $\left< p(c(\bundle{E})),\theta _{2i}\right> =0$. Consider a formal power series in $t$: $a(t)=\sum _ja_j\cdot t^j$ with $a_0\neq 0$ and let $\alpha =(k_1,\ldots ,k_i)\in (\N _0)^i$, then
\[
A^\alpha (t):= \prod _{j=1}^i a(k_j\cdot t)
\]
is symmetric in $k_j$ and we obtain an expansion $A^\alpha (t)=\sum A^\alpha _it^i$ where $A_i^\alpha $ does not depend on $t$ and is a polynomial in $a_1,\ldots ,a_i$. A little exercise in linear algebra shows
\[
\Q [a_1,\ldots ,a_i]_{i}=\{ a\in \Q [a_1,\ldots ,a_i]\ |\ \deg a=i\} =\mathrm{span}_\Q  \{ A^\alpha _i\ |\  |\alpha |=i\} 
\]
where $\deg a_1^{j_1}\cdots a_i^{j_i}=\sum \limits_{s=1}^is\cdot j_s$. Note that the inclusion ''$\supseteq $'' follows by definition and ''$\subseteq $'' follows from dimension reasons because if $\{ \alpha _s\in (\N _0)^i|s\in S\}$ is a set of multi indexes such that for all $r\neq s$ there is at least one elementary symmetric polynomial $\sigma _l$ with $\sigma _l(\alpha _r)\neq \sigma _l(\alpha _s)$, then $\{ A^{\alpha _s}_i\ |\ s\in S\} $ is linear independent in $\Q [a_1,\ldots ,a_i]_i$. If we identify $a_j$ with $\mathrm{ch}_j(\bundle{E}) $ for all $j$, the properties of the Adams operation and the Chern character imply that $A^\alpha _i$ is determined by
\[
\mathrm{ch}_i(\psi _{k_1}(\bundle{E})\otimes \cdots \otimes \psi _{k_l}(\bundle{E}))=\mathrm{ch}_i(\psi _\alpha (\bundle{E})).
\]
Thus, $B^\alpha _i=\left< A_i^\alpha ,\theta _{2i}\right> =0$ for all $\alpha $ proves $\left< b,\theta _{2i}\right> =0$ for all $b\in \Q [a_1,\ldots ,a_i]_i$.

\emph{Conclusion:} Suppose $\bundle{E}\in \mathscr{V}(M;\theta )$ and consider the bundle $\bundle{X}:=\Lambda ^\alpha \left( \Lambda ^{\beta }\bundle{E}\right) $ for multi indexes $\alpha $, $\beta $ with $|\alpha |, |\beta |\leq m$, then $\| R^\bundle{X}\| _g\leq m^2\| R^\bundle{E}\| _g$ is satisfied. Moreover, we conclude $\bundle{X}\in \mathscr{V}_\mathrm{ch}(M;\theta )$ for some $\alpha $, $\beta $ by the following contradiction argument. If $\left< \mathrm{ch}(\bundle{X}),\theta \right> $ vanishes for all $\alpha $ and $\beta $, we apply claim 1 to the bundle $\bundle{Y}=\Lambda ^\beta \bundle{E}$ and conclude $\left< \mathrm{ch}_i(\Lambda ^\beta \bundle{E}),\theta _{2i}\right> =0$ for all $i$ and all $\beta $. Hence, claim 2 proves that for any polynomial $p$: $\left< p(c(\bundle{E})),\theta \right> =0$, i.e.~$\bundle{E}\notin \mathscr{V}(M;\theta )$ a contradiction to our assumption. 
\end{proof}

\begin{prop}
\label{proposition7}
Let $f:(\tilde M,\tilde g)\to (M,g)$ be a Riemannian covering with $M$, $\tilde M$ closed and oriented, then
\[
\ke _\mathrm{ch}(\tilde M_{\tilde g};f^!\theta )=\ke _\mathrm{ch}(M_g;\theta )
\]
where $f^!:H_*(M;\Q )\to H_*(\tilde M;\Q )$ is the transfer homomorphism. 
\end{prop}
\begin{proof}
Observe that $f$ is a finite covering and $0<|\deg f|$ is the number of sheets. Since $f_*f^!\theta =\deg f\cdot \theta $ we conclude ''$\geq $'' from the $\mathrm{K}_\mathrm{ch}$--version of proposition \ref{proposition2} (pull back of vector bundles). Furthermore, each bundle $\bundle{E}\to \tilde M$ induces a bundle $\bundle{E}'\to M$ by
\[
\bundle{E}'_x=\bigoplus _{y\in f^{-1}(x)}\bundle{E}_y.
\]
The connection on $\bundle{E}$ yields a connection on $\bundle{E}'$ with $\| R^{\bundle{E}'}\| _g=\| R^\bundle{E}\| _{\tilde g}$. Thus, we conclude ''$\leq $'' for $0\neq \theta \in H_{2*}(M;\Q )$ if $\bundle{E}\in \mathscr{V}_\mathrm{ch}(\tilde M;f^!\theta )$ implies $\bundle{E}'\in \mathscr{V}_\mathrm{ch}(M;\theta )$ (note that $f^!$ is injective). Assume that $f$ is a normal covering, then $f^*\bundle{E}'$ is isomorphic to $\bigoplus h^*\bundle{E}$ where the sum is taken over all deck transformations $h$. Hence, we obtain
\[
\begin{split}
\deg f\cdot \left< \mathrm{ch}(\bundle{E}'),\theta \right> &=\left< \mathrm{ch}(\bundle{E}'),f_*f^!\theta \right> =\left< \mathrm{ch}(f^*\bundle{E}'),f^!\theta \right> \\
&=\sum _h\left< h^*\mathrm{ch}( \bundle{E})  ,f^!\theta \right> =|\deg f|\cdot \left< \mathrm{ch}(\bundle{E}),f^!\theta \right> 
\end{split}
\]
because for a deck transformation $h:\tilde M\to \tilde M$, $h_*$ acts as $\mathrm{Id}$ on the image of $f^!$: $f=f\circ h$ and $h_*h^!=\mathrm{Id}$ yield $h_*f^!=h_*(fh)^!=h_*(h^!f^!)=f^!$. This proves the claim if $\theta \in H_{2*}(M;\Q )$ and $f$ is a normal covering. If $f$ is not a normal covering, let $H\subseteq f_\# \pi _1(\tilde M,\tilde x)$ be a normal subgroup of $\pi _1(M,f(\tilde x))$ such that $\pi _1(M,f(\tilde x))/H$ is a finite group [$H$ always exists because $f_\# \pi _1(\tilde M,\tilde x)$ has finite index in $\pi _1(M,f(\tilde x))$]. Hence, the associated (smooth) covering $l:(N,y)\to (M,f(\tilde x))$ with $l_\# \pi _1(N,y)=H$ is normal and finite, and there is a unique (smooth) map $p:(N,y)\to (\tilde M,\tilde x)$ which is a normal covering and satisfies $l=f\circ p$. Thus, the previous case proves for $\bar g:=l^*g$ and $\theta \in H_{2*}(M;\Q )$:
\[
\ke _\mathrm{ch}(M_g;\theta )=\ke _\mathrm{ch}(N_{\bar g};l^!\theta )=\ke _\mathrm{ch}(N_{\bar g};p^!f^!\theta )=\ke _\mathrm{ch}(\tilde M_{\tilde g};f^!\theta ).
\]
In order to show the proposition for $\theta \in H_{2*+1}(M;\Q )$ consider the covering $b:=f\times \mathrm{id}:\tilde M\times S^1\to M\times S^1$, use $b^!(\theta \times [S^1])=(f^!\theta )\times [S^1]$ and apply the case for even classes to the covering $b$ and the class $\theta \times [S^1]\in H_{2*}(M;\Q )$.
\end{proof}
This proposition allows us to compute the K--areas for the torus respectively for the connected sum of ''nice'' manifolds with a torus. Consider the $2^n$--fold Riemannian covering $f:(T^n,4g)\to (T^n,g),\ \alpha \mapsto 2\alpha $, then by the scaling property
\[
\ke _\mathrm{ch}(T^n_g;\theta )=\ke _\mathrm{ch}(T^n_{4g};f^!\theta )=4\cdot \ke _\mathrm{ch}(T^n_g;\theta )
\]
which shows $ \ke (T^n_g;\theta )=\ke _\mathrm{ch}(T^n_g;\theta )=\infty $ for each $\theta \neq 0$. Here we use that the $\mathrm{K}_\mathrm{ch}$--area of a nontrivial class $\theta \in H_*(M;\Q )$ is in $(0,\infty ]$. This and the following remark prove that $\H _*(T^n\# \cdots \# T^n)=\{ 0\} $ for any finite number of tori. 

\begin{rem}
\label{zusammen}
Suppose that $M^n$ and $N^n$ are closed, oriented and connected. Consider the projection maps $M\# N\to M$ and $M\# N \to N$ by identifying the summand $M$ respectively $N$ to a point, then $\H _k(M\# N)\to \H _k(M)\oplus \H _k(N)$ is well defined and injective, i.e.~$\H _k(M\# N) $ can be considered as a subgroup of $\H _k(M)\oplus \H _k(N)$ for all $k<n$. Moreover, $\H _n(M\# N)=\Z $ implies $\H _n(M)=\Z $ and $\H _n(N)=\Z $. Conversely, if $X\subseteq M$ or $Y\subseteq N$ are compact submanifolds of codimension at least one which give a finite upper bound for K--areas on $M$ respectively $N$, then $X,Y\subseteq M\# N$ determine upper bounds for the corresponding K--areas on $M\# N$. For instance, if $0<j<n$, $\C P^j\subseteq \C P^n\# T^{2n}$ supplies finite K--area for the homology class induced by the fundamental class of $\C P^j$ which means    
\[
\H _k(\C P^n\# T^{2n})=\biggl\{ \begin{array}{cl}
\Z &if \ \ k\in \{ 2,4,\ldots ,2n-2\}\\
0& otherwise.
\end{array}
\]
\end{rem}
\begin{rem}
\label{rem_covering}
Some minor modifications of the above proof extend the covering result to certain product manifolds: Let $f:\tilde N_{\tilde h}\to N_h$ be a Riemannian covering with $N$, $\tilde N$ closed and oriented and suppose that $M_g$ is a compact Riemannian manifold. Then $\mathrm{id}\times f:M_g\times \tilde N_{\tilde h}\to M_g\times N_h$ is a Riemannian covering and
\[
\ke _\mathrm{ch}(M_g\times \tilde N_{\tilde h};\eta \times f^!\theta )=\ke _\mathrm{ch}(M_g\times N_h;\eta \times \theta )
\]
holds for all $\eta \in H _*(M;\Q )$ and $\theta \in H_*(N;\Q )$. Hence, if $h$ is a metric on $T^n$ and $h_0$ is a standard flat metric on $T^n$ with $h_0\leq h\leq k^2h_0$ on $TT^n$ for some $k\in \N $, then $f:T^n_{k^2\cdot h_0}\to T^n_{h_0}$, $x\mapsto x^k$ is a $k^n$--fold Riemannian covering which yields
\[
\ke _\mathrm{ch}(M_g\times T_{k^2 h_0}^n;\eta \times \theta )=\ke _\mathrm{ch}(M_g\times T_h^n;\eta \times \theta )=\ke _\mathrm{ch}(M_g\times T_{h_0}^n;\eta \times \theta )
\] 
for all $\eta \in H_*(M;\Q )$ and $\theta \in H_*(T^n;\Q )$ (use proposition \ref{proposition2}).
\end{rem}
\begin{rem}
\label{funda}
Consider a simply connected and closed manifold $N$, then there is some $\epsilon >0$ such that any $\epsilon $--flat bundle on $N$ is already trivial (cf.~Gromov \cite{Gr01}). Hence, the K--area of a class $\theta \in H_{2k}(N)$ is finite if $k>0$ which means $\H _{2k}(N)=H_{2k}(N)$ for all $k>0$. Moreover, the last proposition and the above observations show that $\H _{2k}(M)=H _{2k}(M)$ for all $k>0$ if $M$ is a closed orientable manifold with finite fundamental group. The equality might also hold for odd homology classes, but in this case a more detailed estimate is necessary.  
\end{rem}
\section{Stabilizing K--area and products}
\label{stab_area}
The K--area has a further generalization which depends on taking products with standard tori $T^i=S^1\times \cdots \times S^1$. Suppose that $\theta \in H_{2*-i}(M;\Q )$ for $i\geq 0$, then the $\mathrm{K}_\mathrm{ch}^i$--area of $\theta $ is defined by
\[
\ke _\mathrm{ch}^i(M_g;\theta ):=\ke _\mathrm{ch}(M_g\times T^i_h;\theta \times [T^i])
\]
where the right hand side does not depend on the choice of the metric $h$ by remark \ref{rem_covering}. Note that $\ke ^0_\mathrm{ch}(M_g;\theta )$ is the above $\mathrm{K}_\mathrm{ch}$--area for $\theta \in H_{2*}(M;\Q )$ and  $\ke ^1_\mathrm{ch}(M_g;\theta )$ coincides with the above definition of the $\mathrm{K}_\mathrm{ch}$--area for $\theta \in H_{2*+1}(M;\Q )$. The proposition below shows that $\ke _\mathrm{ch}^i(M;\theta )$ is nondecreasing in $i$ which means $\ke _\mathrm{ch}^{i+2}(M_g;\theta )\geq \ke ^i_\mathrm{ch}(M_g;\theta )$ for all $\theta \in H_{2*-i}(M;\Q )$. The more challenging question is to compute estimates in the opposite direction. We also consider the corresponding stabilized version of this area:
\[
\ke _\mathrm{ch}^\mathrm{st}(M_g;\theta ):=\sup _i\ke _\mathrm{ch}^i(M_g;\theta )\geq \ke _\mathrm{ch}(M_g;\theta ).
\]
We could introduce similar objects for the ordinary K--area, however, it is uncertain if a stabilized version of the K--area can be estimated by the stabilized $\mathrm{K} _\mathrm{ch}$--area (compare the constant in proposition \ref{proposition6}). Since the stabilized $\mathrm{K}_\mathrm{ch}$--area inherits most of the properties of the $\mathrm{K}_\mathrm{ch}$--area, the set $\H _k^\mathrm{st}(M;\Q )\subseteq H _k(M;\Q )$ consisting of rational homology classes with finite stabilized $\mathrm{K}_\mathrm{ch}$--area is a linear subspace of $\H _k(M;\Q )$ for all $k$. Moreover, a continuous map $f:M\to N$ yields homomorphisms $f_*:\H _k^\mathrm{st}(M;\Q )\to \H _k^\mathrm{st}(N;\Q )$. Hence, the stabilized K--area homology $\H ^\mathrm{st}_k(M;\Q )$ depends only on the homotopy type of $M$. There are plenty of manifolds $M$ with $\H _*(M;\Q )=\H _*^\mathrm{st}(M;\Q )$, however it seems a rather difficult question if this is true for general $M$. The reason for introducing the stabilized version is the more consistent behavior in the product case:  
\begin{prop}
Suppose that $\eta \in H_k(M;\Q )$ and $\theta \in H_l(N;\Q )$, then
\[
\begin{split}
\ke _\mathrm{ch}^{i+j}(M_g\times N_h;\eta \times \theta )&\geq \min \left\{ \ke _\mathrm{ch}^i(M_g;\eta ),\ke _\mathrm{ch}^j(N_h; \theta )\right\}\\
\ke _\mathrm{ch}^{\mathrm{st}}(M_g\times N_h;\eta \times \theta )&\geq \min \left\{ \ke _\mathrm{ch}^\mathrm{st}(M_g;\eta ),\ke _\mathrm{ch}^\mathrm{st}(N_h; \theta )\right\}
\end{split}
\]
where $i+k$ and $j+l$ are assumed to be even integers. 
\end{prop}
\begin{proof}
The inequality is obvious if one of the homology classes vanishes, hence $\eta \neq 0$ and $\theta \neq 0$. We set $\tilde \eta :=\eta \times [T^i]$ and $\tilde \theta :=\theta \times [T^j]$ for notational simplicity. Suppose that  $\bundle{E}\in \mathscr{V}_\mathrm{ch}(M\times T^i;\tilde \eta )$ and $\bundle{F}\in \mathscr{V}_\mathrm{ch}(N\times T^j;\tilde \theta )$, then
\[
0\neq \bigl< \mathrm{ch}(\bundle{E}),\tilde \eta \bigl> \cdot \bigl< \mathrm{ch}(\bundle{F}),\tilde \theta \bigl> =\bigl< \mathrm{ch}(\bundle{E})\times \mathrm{ch}(\bundle{F}),\tilde \eta \times \tilde \theta \bigl>  =\bigl< \mathrm{ch}(\pi _M^*\bundle{E}\otimes \pi _N^*\bundle{F}),\tilde \eta \times \tilde \theta \bigl>  
\]
shows $\pi _M^*\bundle{E}\otimes \pi _N^*\bundle{F}\in \mathscr{V}_\mathrm{ch}(M\times T^i\times N\times T^j;\tilde \eta \times \tilde \theta )$. Moreover, we obtain for the tensor product connection on $\bundle{G}:=\pi _M^*\bundle{E}\otimes \pi _N^*\bundle{F}$:
\[
\| R^\bundle{G}\| _{g\oplus h}=\max \left\{ \| R^\bundle{E}\| _g,\| R^\bundle{F}\| _h\right\} .
\]
Hence, considering the pull back of $\bundle{G}$ by the coordinate transposition map $M\times N\times T^{i+j}\to M\times T^i\times N\times T^j$ proves the inequality.
\end{proof}
This proposition yields additional examples for finite K--area. For instance, if $M^m$ and $N^n$ are connected, closed and spin and $M\times N$ admits a metric of positive scalar curvature, then $\H _m(M)=\Z $ or $\H _n(N)=\Z $. Note that in this case neither $M$ nor $N$ need to carry a metric of positive scalar curvature. In general it seems rather difficult to see whether $\theta \times \eta $ has finite K--area even if both $\theta $ and $\eta $ have finite K--area. We observe equality in the previous proposition for $\eta \in H_0(M;\Q )$ or $\theta \in H_0(N;\Q )$ by considering the inclusion maps  (cf.~proposition \ref{proposition2}). Moreover, the following proposition proves finite K--area for homology classes $[M]\times \theta $ if $M$ is a closed spin manifold of positive scalar curvature. 
\begin{prop}
\label{prop_scalar}
Let $(M^n,g)$ be a connected closed spin manifold of positive scalar curvature and $\theta \in H_l(N;\Q )$ where $N$ is a compact manifold, then 
\[
\ke _\mathrm{ch}^\mathrm{st}\bigl( M_g\times N_h;\bigl(\widehat{\mathbf{A}}(M)\cap [M]\bigl)\times \theta \bigl)\leq \frac{2n(n-1)}{\min \mathrm{scal}_g}.
\]
\end{prop}
\begin{proof}
It is sufficient to prove the proposition for the $\mathrm{K}_\mathrm{ch}$--area, the stabilized version follows from this case by replacing $(N,\theta )$ with $(N\times T^i,\theta \times [T^i])$. The inequality is obvious for $\theta =0$, hence suppose $\theta \neq 0$. We start with the case that $n$ is even. Without loss of generality $l$ is even, otherwise go over to $N\times S^1$ and $\theta \times [S^1]$. A bundle $\bundle{E}\in \mathscr{V}_\mathrm{ch}\bigl( M\times N;\bigl(\widehat{\mathbf{A}}(M)\cap [M]\bigl)\times \theta \bigl) $ yields a family of smooth vector bundles over $M$ parameterized by $N$ (cf.~\cite{AS4,BGV,LaMi}). Thus, twisting the (constant family of) complex spinor bundle $\spinor M$ with the family of bundles $\bundle{E}$ yields a family of Dirac operators $\dirac ^+_\bundle{E}:\Gamma (\spinor ^+M\otimes \bundle{E})\to \Gamma (\spinor ^-M\otimes \bundle{E})$ over $M$ parameterized by $N$ where the corresponding index $\mathrm{ind}\, \dirac ^+_\bundle{E}\in K(N)$ satisfies
\[
\begin{split}
\left< \mathrm{ch}(\mathrm{ind}\, \dirac ^+_\bundle{E} ),\theta \right> &=\left< \int _M\bigl( \widehat{\mathbf{A}}(M)\times 1\bigl) \cdot \mathrm{ch}(\bundle{E}), \theta \right> \\
&=\left< \mathrm{ch}(\bundle{E}),\bigl( \widehat{\mathbf{A}}(M)\cap [M]\bigl) \times \theta \right> \neq 0
\end{split}
\]
Hence, there is a point $x\in N$ such that the associated twisted Dirac operator $\dirac _{\bundle{E}_x}:\Gamma (\spinor M\otimes \bundle{E}_x)\to \Gamma (\spinor M\otimes \bundle{E}_x)$ has nontrivial kernel where $\bundle{E}_x$ is the induced bundle over $M\times \{ x\}\subseteq M\times N$. The integrated version of the Bochner Lichnerowicz formula
\[
\dirac ^2_{\bundle{E}_x}=\nabla ^*\nabla +\frac{\mathrm{scal}_g}{4}+\frak{R}^{\bundle{E}_x}
\]
shows that there is a point on $M$ where the minimal eigenvalue of $\frak{R}^{\bundle{E}_x}$ is less or equal to $ -\mathrm{scal}_g/4$. However, the maximal absolute eigenvalue of $\frak{R}^{\bundle{E}_x}$ is bounded by $\frac{n(n-1)}{2}\| R^{\bundle{E}_x}\| _g$ which implies
\[
\| R^\bundle{E}\| _{g\oplus h}\geq \| R^{\bundle{E}_x}\| _g\geq \frac{\min  \mathrm{scal}_g}{2n(n-1)}
\]
for all $\bundle{E}\in \mathscr{V}_\mathrm{ch}\bigl( M\times N;\bigl( \widehat{\mathbf{A}}(M)\cap [M]\bigl) \times \theta \bigl) $ and an arbitrary Riemannian metric $h$ on $N$. If $M^n$ is odd dimensional, we apply the even dimensional case to the manifold $M_g\times S^1_{c\cdot \mathrm{d}t^2}$ for constants $c>0$. Again without loss of generality $l$ is even, i.e.~$\theta $ is an even homology class. In particular, nontrivial kernel of $\dirac _{\bundle{E}_x}$ yields
\[
\frac{\min \mathrm{scal}_g}{4}=\frac{\min \mathrm{scal}_{g\oplus c\cdot \mathrm{d}t^2}}{4}\leq \| \frak{R}^{\bundle{E}_x}\| _{op}\leq \left( \frac{n(n-1)}{2}+\frac{n}{c}\right) \cdot \| R^{\bundle{E}_x}\| _{g\oplus \mathrm{d}t^2}
\]
for all constants $c>0$. Hence, considering $c\to \infty $ shows
\[
\min \mathrm{scal}_g\leq 2n(n-1)\| R^{\bundle{E}_x}\| _{g\oplus \mathrm{d}t^2}\leq 2n(n-1)\| R^\bundle{E}\| _{g\oplus \mathrm{d}t^2\oplus h}
\]
on $M_g\times S^1_{\mathrm{d}t^2}\times N_h$ for all $\bundle{E}\in \mathscr{V}_\mathrm{ch}\bigl( M\times S^1\times N;\bigl( \widehat{\mathbf{A}}(M)\cap [M]\bigl) \times [S^1] \times \theta \bigl) $. Note that the inequality is independent on the size of $S^1$ and the choice of the metric $h$ on $N$.
\end{proof}
\begin{rem}
\label{rem243}
Equality in the above situation can only hold if $g$ has constant scalar curvature. As an example, we obtain equality for even dimensional spheres of constant curvature. In fact, if $S^{2n}_0$ denotes the standard sphere of constant sectional curvature $K=1$, then the complex spinor bundle of $S^{2n}$ satisfies $\| R^{\sspinor ^+}\| =\frac{1}{2}$ and $c_n(\spinor ^+)\neq 0$ which proves
\[
\ke (S^{2n}_0)=\ke _\mathrm{ch}(S^{2n}_0)=\ke _\mathrm{ch}^\mathrm{st}(S^{2n}_0)=2.
\] 
In corollary \ref{cor345} we shall see that this is also true for odd dimensional spheres. Hence, we obtain from proposition \ref{proposition7} that a closed orientable manifold $M_g^n$ with constant sectional curvature $K\geq 0$, satisfies $\ke (M_g)=\ke _\mathrm{ch}(M_g)=2/K$. 
\end{rem}
\begin{rem}
\label{rem8}If $M$ is a compact manifold, the above proposition shows that $\theta \times [S^n]$ has finite K--area for all $\theta \in H_k(M)$ and $n\geq 2$. Hence, we obtain from singular theory that
\[
\H _{k}(M)\oplus H _{k-n}(M)\to \H _{k}(M\times S^n), \ (\eta ,\theta )\mapsto \eta \times \mathbbm{1}+\theta \times [S^n]
\]
is an isomorphism for all $k$ and $n\geq 2$. The same statement holds of course for the stabilized K--area homology. 
\end{rem}

\begin{example}
There are manifolds with isomorphic fundamental group, isomorphic singular homology but different K--area homology as the following example shows. Suppose $n>3$ and define
\[
\begin{split}
M&=(T^3\times S^n)\# (T^3\times S^n)\\
N&=\bigl( (T^3\# T^3)\times S^n\bigl) \# (S^3\times S^n),
\end{split}
\]
then $\pi _1(M)=\pi _1(N)=\Z ^3*\Z ^3$ by Seifert--van Kampen. We leave it to the reader to compute the singular homology. In order to determine the K--area homology of $M$ and $N$ we use the methods presented in remark \ref{zusammen} and results about positive scalar curvature on closed spin manifolds, in fact we obtain
\[
\begin{split}
\H _k(M)&=\H _k(N)=H_k(M)=H_k(N)\quad  k\geq 4\\
\H _j(M)&=\H _j(N)=\{ 0\}\qquad \ \ j=0,1,2
\end{split}
\]
whereas $\H _3(N)=\Z $ and $\H _3(M)=\{ 0\} $.
\end{example}
\begin{example}
\label{exam136}
The manifolds in the previous example have different intersection product, but it seems rather difficult to construct manifolds with isomorphic fundamental group, isomorphic cohomology ring and compute their respective K--area homology. However, if we drop the assumption on the fundamental group, interesting examples already exist in $3$--dimensions. For instance, let $X$ be a closed connected oriented $3$--manifold which is a hyperbolic homology sphere, i.e.~$X$ admits a hyperbolic metric and $H_*(X)=H_*(S^3)$. Certain Brieskorn homology spheres (cf.~\cite[Chp.~1]{Sav}) and the main theorem in \cite{My} supply examples of hyperbolic homology spheres. Hence, the cohomology ring of $X$ is isomorphic to the cohomology ring of $S^3$, but $\H _*(S^3)=\H _3(S^3)=\Z $ whereas $\H _*(X)=\{ 0\} $ by \cite[$4\frac{3}{5}$(v')]{Gr01}. Here we use that $X$ has residually finite fundamental group ($\dim X=3$) and that $X$ admits a metric of nonpositive sectional curvature. This leads to further $3$--dimensional examples. For instance, let $N$ be a closed connected oriented $3$--manifold which admits a metric of positive scalar curvature (like $N=S^2\times S^1\# \cdots \# S^2\times S^1\# \proj{\R }^3$), then $N$ and $N\# X$ have isomorphic cohomology ring. However, the above results show $\H _3(N)=\Z $ whereas $\H _3(N\# X)=\{ 0\}$ by remark \ref{zusammen}. 
\end{example}
In general, the intersection product $\bullet $ does not preserve finite K--area. For instance, $1\times [S^2],[S^2]\times 1\in H_2(S^2\times S^2)$ have finite K--area, but
\[
(1\times [S^2])\bullet ([S^2]\times 1)=1\in H_0(S^2\times S^2)
\]
has infinite K--area. However, the intersection product is well behaved by stabilizing with spheres $S^n$ for $n\geq 2$:
\[
(\theta \times [S^n])\bullet (\eta \times [S^n])=(\theta \bullet \eta )\times [S^n]\in \H _*(M\times S^n)
\]
for all $\theta ,\eta \in H_*(M)$.

\begin{prop}
Suppose that $\theta \in H_k(M;\Q )$, $\alpha \in H^1(M;\Q )$ and $i+k+1\in 2\Z $, then:
\[
\ke _\mathrm{ch}^i(M_g;\alpha \cap \theta )\leq \ke _\mathrm{ch}^{i+1}(M_g;\theta ).
\]
\end{prop}
\begin{proof}
Consider a bundle $\bundle{E}\in \mathscr{V}_\mathrm{ch}(M\times T^i;(\alpha \cap \theta )\times [T^i])$ and a smooth map $f:M\to S^1$ with $f^*([S^1]^*) =\alpha $ for the orientation class $[S^1]^* \in H^1(S^1;\Q )$. Denote by $\pi :M\times T^{i}\times S^1\to M\times T^i$ the projection and by $\sigma :M\times T^i\times S^1\to S^1\times S^1, (x,y,t)\mapsto (f(x),t)$. For any $\epsilon >0$ there is a bundle $\bundle{F}\in \mathscr{V}_\mathrm{ch}(S^1\times S^1;[S^1]\times [S^1])$ with $\| R^\bundle{F}\| \leq \epsilon $. Hence, $\bundle{G}=\pi ^*\bundle{E}\otimes \sigma ^*\bundle{F}$ satisfies $\| R^\bundle{G}\| \leq \| R^\bundle{E}\| +\epsilon $ and since $\mathrm{ch}(\bundle{F})=\mathrm{rk}(\bundle{F})+x\cdot ([S^1]^* \times [S^1]^*)$ for some $x\in \Q \setminus \{ 0\} $ we conclude
\[
\begin{split}
\left< \mathrm{ch}(\bundle{G}),\theta \times [T^{i+1}]\right> &=x\cdot \left< \pi ^*\mathrm{ch}(\bundle{E})\cdot (\alpha \times \mathbbm{1}\times [S^1]^*),\theta \times [T^{i}]\times [S^1]\right> \\
&=x\cdot \left< \pi ^*(\mathrm{ch}(\bundle{E})\cdot \alpha \times \mathbbm{1}),\theta \times [T^{i}]\times \mathbbm{1}\right>\\
&=x\cdot \left< \mathrm{ch}(\bundle{E}),(\alpha \cap \theta )\times [T^i]\right> \neq 0 
\end{split}
\]
which proves the inequality.
\end{proof}
We obtain as a corollary that capping with cohomology classes of degree one yields maps:
\[
\begin{split}
&H^1(M;\Z )\times \H _{2k+1}(M)\to \H _{2k}(M),\ (\alpha ,\theta )\mapsto \alpha \cap \theta \\
&H^1(M;\Q )\times \H_k^\mathrm{st}(M;\Q )\to \H _{k-1}^\mathrm{st}(M;\Q ),\ (\alpha ,\theta )\mapsto \alpha \cap \theta .
\end{split}
\]
In particular, a connected closed manifold $M^n$  with $\H ^\mathrm{st}_n(M;\Q )=\Q $ satisfies $\H ^\mathrm{st}_{n-1}(M;\Q )=H_{n-1}(M;\Q )$ and therefore $\H _{n-1}(M)=H_{n-1}(M)$. This may be seen as the K--area analog of Schoen and Yau's result about positive scalar curvature on minimal hypersurfaces (cf.~\cite{SchY6}). The corresponding statement for manifolds of positive scalar curvature was used by Schick in \cite{Schick1} to give a counterexample to the unstable Gromov--Lawson--Rosenberg conjecture.
\section{K--area on noncompact manifolds}
A noncompact manifold has no fundamental class and in view of the relative index theorem we are rather interested on the K--area defined with bundles of compact support. Hence, in order to get a theory of finite K--area on noncompact manifolds we can consider subgroups of the Borel--Moore homology or subgroups of singular cohomology. Moreover, in order to simplify statements about positive scalar curvature and the $\widehat{A}$--class we choose the cohomology approach. In this section all manifolds are assumed to be connected, oriented and without boundary. Remember that singular homology is dual isomorphic to cohomology with compact support and $H^k(M;\Q )$ is the dual space of $H_k(M;\Q )$. Thus, the bilinear pairing
\[
H^{m-k}(M;\Q )\times H^{k}_\mathrm{cpt}(M;\Q )\to \Q \ ,(\alpha ,\beta )\mapsto \int  _M\alpha \cup \beta 
\]
is well defined and nondegenerate where $m=\dim M$. Moreover, if $f:M^m\to N^n$ is a proper map, the pull back $f^*:H_\mathrm{cpt}^k(N;\Q )\to H_\mathrm{cpt}^k(M;\Q )$ induces the transfer homomorphism
\[
f_!:H^{m-k}(M;\Q )\to H^{n-k}(N;\Q ),\quad \int _N (f_!\alpha )\cup \gamma :=\int _M\alpha \cup f^*\gamma .
\]
If $\alpha \in H^{m-2k}(M^m;\Q )$ is a cohomology class, $\mathscr{V} _{\mathrm{cpt}}(M;\alpha )$ denotes the set of Hermitian vector bundles $\bundle{E}\to M$ endowed with a Hermitian connection such that $\bundle{E}$ and its connection are trivial outside of a compact set and such that there are nonnegative integers $i_1,\ldots ,i_k$ with $\sum j\cdot i_j=k$ and
\[
0\neq \int _M\alpha \cdot c_1(\bundle{E})^{i_1}\cdots c_k(\bundle{E})^{i_k} .
\]
Note that the right hand side is well defined because $c_i(\bundle{E})$ has compact support. The K--area $\ke (M_g;\alpha )$ of the Riemannian manifold $(M,g)$ w.r.t.~the cohomology class $\alpha $ is defined like in (\ref{defn_k}) by taking the infimum over bundles in $\mathscr{V} _\mathrm{cpt}(M;\alpha )$. We set $\ke  (M_g;\alpha )=0$ in case of $\mathscr{V} _\mathrm{cpt}(M;\alpha )=\emptyset $. If $k$ is odd and $\alpha \in H^{m-k}(M^m;\Q )$, define
\[
\ke (M_g;\alpha ):=\sup _{\mathrm{d}t^2}\ke (M_g\times S^1_{\mathrm{d}t^2};\pi ^*\alpha )
\]     
for the projection map $\pi :M\times S^1\to M$. In accordance with proposition \ref{proposition1} the K--area of a general cohomology class $\alpha =\sum \alpha _i$ is the maximum of the K--areas of the $\alpha _i\in H^i(M;\Q )$. If $f:(M^m,g)\to (N^n,h)$ is a smooth proper Lipschitz map in the sense that $g\geq f^*h$ on $TM$, then
\[
\ke (M_g;\alpha )\geq \ke (N_h;f_!\alpha )
\]
for all $\alpha \in H^{*}(M;\Q )$. If $\alpha \in H^{m-2*}(M;\Q )$, this inequality is still true for $\Lambda ^2$--Lipschitz maps $f$ (i.e.~$g\geq f^*h$ on $\Lambda ^2TM$). However, it fails for general $\alpha $ because the condition $\Lambda ^2$--Lipschitz does not extend to the map $f\times \mathrm{id}:M\times S^1\to N\times S^1$ if $M$ is not compact (cf.~proof of proposition \ref{proposition2}). 
Because proposition \ref{proposition1} holds for the cohomology K--area,
\[
\H ^j(M_g;\Q ):=\{ \alpha \in H^j(M;\Q )\ |\ \ke (M_g;\alpha )<\infty \}
\] 
is a linear subspace for each $j$ which of course depends on the asymptotic geometry of the metric $g$ on noncompact manifolds except for $j=m=\dim M$ since flat connections on trivial bundles mean $\H ^m(M;\Q )=\{ 0\}$. The subscript on $M$ refers to the choice of a geometry at infinity. However, if $g\sim h$ in the sense that there is a constant $C>0$ with $C^{-1}\cdot g\leq h\leq C\cdot g$ on $TM$, then $\H ^j(M_g;\Q )=\H ^j(M_h;\Q )$. In fact, on compact manifolds the subspaces $\H ^j(M;\Q )$ are independent on the choice of the metric. Moreover, if $f:(M^m,g)\to (N^n,h)$ is a smooth proper Lipschitz map, the transfer homomorphisms restrict to  homomorphisms on the respective cohomology groups with finite K--area:
\[
f_!:\H ^{m-k}(M_g;\Q )\to \H ^{n-k}(N_h;\Q ).
\]
If $M^m$ is a closed oriented manifold, the above definitions coincide up to Poincar\'e duality, i.e.
\[
\ke (M_g;\alpha )=\ke (M_g;\alpha \cap [M])
\]
where $[M]\in H_{m}(M)$ is the fundamental class. In particular, Poincar\'e duality yields isomorphisms
\[
\cap [M]:\H ^j(M;\Q )\to \H _{m-j}(M;\Q ).
\]
At this point we should clarify that in general $\H ^j(M;\Q )$ is no longer the dual space of $\H _j(M;\Q )$. 

The definitions of \emph{compactly enlargeable}, \emph{compactly} $\widehat{A}$\emph{--enlargeable} and \emph{weakly enlargeable} are taken from the survey book \cite{LaMi}. These versions of enlargeability are also used in \cite{GrLa2,HaSch1}. Using the generalized K--area concept introduced by Hanke  one can show that the fundamental homology class of enlargeable manifolds has infinite K--area (cf.~\cite[Proposition 3.8]{Hanke}).

\begin{prop}
If $M$ is a compactly $\widehat{A}$--enlargeable manifold, then the total $\widehat{A}$--class of $M$ has infinite K--area: $\widehat{\mathbf{A}}(M)\notin \H ^*(M;\Q )$. In particular, a compactly enlargeable manifold $M^m$ satisfies
\[
\H ^0(M;\Q )=\H _m(M;\Q )=\{ 0\} .
\]
Moreover, if $M$ is weakly enlargeable, then for each metric $g$ on $M$ and every constant $C>0$ there is a Riemannian covering $\tilde{M}_{\tilde g}$ with $\ke (\tilde M_{\tilde g};\widehat{\mathbf{A}}(\tilde{M}))>C$. 
\end{prop}
\begin{proof}
Consider the case of a compactly $\widehat{A}$--enlargeable manifold $M_g$, the enlargeable case follows in the same way by replacing the $\widehat{A}$--class with $\mathbbm{1}\in H^0(M;\Q )$ and the $\widehat{A}$--degree with the ordinary degree of a map. $M$ is closed and for each $\epsilon >0$ there is a finite Riemannian covering $(\tilde{M},\tilde{g})\to (M,g)$ and a map $f:(\tilde M,\tilde g)\to (S^k,g_0)$ of nontrivial $\widehat{A}$--degree and with $\epsilon \cdot \tilde g\geq f^* g_0$ on $T\tilde M$. Then the Poincar\'e dual of $f_!(\widehat{\mathbf{A}}(\tilde M))$ in the top degree is determined by $\deg _{\widehat{A}}f\cdot [S^k]\in H_k(S^k;\Q )$. Since $f$ is proper, we obtain for the standard sphere $S^k_0$
\[
\epsilon \cdot \ke (\tilde M_{\tilde g};\widehat{\mathbf{A}}(\tilde M))\geq \ke (S^k_0;f_!(\widehat{\mathbf{A}}(M))) \geq \ke (S^k_0;\deg _{\widehat{A}}f\cdot [S^k])=\ke (S^k_0).
\]
Because $\epsilon >0$ is arbitrary and $\ke (S^k_0)\in (0,\infty )$, proposition \ref{proposition6} and \ref{proposition7} yield
\[
\ke (M_g;\widehat{\mathbf{A}}(M))=\ke (\tilde{M}_{\tilde {g}};\widehat{\mathbf{A}}(\tilde M))=\infty 
\]
which proves $\widehat{\mathbf{A}}(M)\notin \H ^*(M;\Q )$ (here we used $\pi ^!(\alpha \cap [M])=\pi ^*\alpha \cap [\tilde M]$ for the covering map $\pi :\tilde M\to M$ and a cohomology class on $M$).

The second claim follows similar. Consider a map $f:\tilde M\to S^{2k}$ of nonzero $\widehat{A}$--degree which is constant at infinity. Although $f$ is not proper if $M$ is not compact, the pull back by $f$ yields a map $f^*:\mathscr{V} (S^{2k};\mathbbm{1})\to \mathscr{V}_\mathrm{cpt}(\tilde M;\widehat{\mathbf{A}}(\tilde M))$, here $\mathbbm{1}\in H^0(S^{2k})$ is the unit. This supplies for $\epsilon ^2\cdot \tilde g\geq f^*g_0$ on $\Lambda ^2T\tilde M$:
\[
\epsilon \cdot \ke (\tilde M_{\tilde g};\widehat{\mathbf{A}}(\tilde M))\geq \ke (S^{2k}_0)\in (0,\infty )
\]
and since $\epsilon >0$ is arbitrary, we obtain the claim if $M$ is even dimensional. If $M$ is odd dimensional, consider the composition of the maps
\[
\tilde M\times S^1\stackrel{f\times \mathrm{id}}{\longrightarrow }S^{2k-1}\times S^1\stackrel{\mathrm{pr}}{\longrightarrow} (S^{2k-1}\times S^1)/(S^{2k-1}\vee S^1)\cong S^{2k}
\]
where the one point union $S^{2k-1}\vee S^1$ takes place at a point $(f(x),z)\in S^{2k-1}\times S^1$ and $x \in M$ means a point sufficiently close to infinity. This new map has nonzero $\widehat{A}$--degree and is constant at infinity, i.e.~the even dimensional case provides the claim. 
\end{proof}
The proof of this proposition supplies the relation
\[
\H _k(M;\Q )\subseteq H_k^{\mathrm{sm}(P)}(M;\Q )
\]
between the K--area homology and the small group homology introduced in \cite{BrHa} if $P$ denotes the largeness property compactly enlargeable. In order to see this simply use the above proof to show that a class $\theta \in H_k(M;\Q )$ which has finite K--area can not be (compactly) enlargeable in the sense given in \cite[Def.~3.1]{BrHa}.
\begin{prop}[\cite{GrLa3,Gr01}]
Suppose a spin manifold $M^n$ admits a complete metric $g$ of uniformly positive scalar curvature, then
\[
\ke ( M_g;\widehat{\mathbf{A}}( M))\leq \frac{n(n+1)^3}{2\cdot \inf \, \mathrm{scal}(g)}
\]
which implies $\widehat{\mathbf{A}}(M)\in \H ^*(M_g;\Q )$. 
\end{prop}
Note that the constant on the right hand side does not change for Riemannian covering spaces of $(M,g)$.  The proof of the proposition is a simple consequence of the relative index theorem. If $\ke ( M_g;\widehat{\mathbf{A}}(M))$ is not bounded by the constant on the right hand side, we obtain a contradiction to the relative index theorem. The factor $n(n+1)^3/2$ is very rough and includes the estimate of the K--area by the $\mathrm{K}_\mathrm{ch}$--area as well as the estimate of the twisted curvature term in the Bochner formula. If $n$ is even, this factor can be improved to $(n-1)n^3/2$. Index theory can also be used to show finite K--area of the fundamental homology class on manifolds which are not spin. For instance, suppose that $(M^n,g)$ is a closed connected spin$^c$ manifold with associated class $c\in H^2(M;\Z )$, and let $\Omega $ be a 2--form representing the real Chern class of $c$. If $2 | \Omega | _\mathrm{op} <\mathrm{scal}_g$, then
\[
\ke _\mathrm{ch}(M;(e^{c/2}\cdot \widehat{\mathbf{A}}(M))\cap [M])\leq \frac{2n(n-1)}{\min (\mathrm{scal}_g-2|\Omega |_\mathrm{op})}\ ,
\]
i.e.~$e^{c/2}\cdot \widehat{\mathbf{A}}(M)\in \H ^*(M;\Q )$ implies $\H _n(M;\Q )=\H ^0(M;\Q )=\Q $. In this case $| \Omega |_\mathrm{op}$ denotes the pointwise operator norm of Clifford multiplication by $\Omega $ which means  $| \Omega |_\mathrm{op}=\sum |\lambda _j|$ for a diagonalization $\Omega =\sum \lambda _je_{2j-1}\wedge e_{2j}$. In order to see the above estimate for even $n$, consider the Dirac operator on bundles $\spinor ^cM\otimes \bundle{E}$ where $\spinor ^cM$ is the spin$^c$ bundle and $\bundle{E}\in \mathscr{V}_\mathrm{ch}(M;e^{c/2}\widehat{\mathbf{A}}(M)\cap [M])$ (cf.~\cite{LaMi}). Further examples are obtained by twisting a fixed Dirac bundle $\bundle{S}$ over $M$ with bundles $\bundle{E}$. If $\bundle{S}$ and its connection split locally into $\spinor M\otimes \bundle{X}$ and $4\cdot | \frak{R}^\bundle{X}|_\mathrm{op}<\mathrm{scal}_g$, then the K--areas of certain characteristic classes associated to $\bundle{S}$ are finite. In this case $\frak{R}^\bundle{X}$ means the twisted curvature endomorphism in the Bochner Lichnerowicz formula.
\section{Relative K--area}
\label{sec15}
We are interested in K--areas for relative homology classes. A particular example is the K--area of the fundamental class of a compact manifold $M$ with nontrivial boundary which gives the total K--area of $M$. In order to simplify the notation, we consider only the coefficient group $\group =\Z $ but most of the statements can be adapted to arbitrary coefficient groups. A \emph{pair of compact manifolds} $(M,A)$ consists of a compact manifold $M$ and a compact submanifold $A\subseteq M$ where both $M$ and $A$ may have nontrivial boundary. If $g$ is a Riemannian metric on $M$, there is some $\delta  >0$ in such a way that $A$ is a strong deformation retract of the open set
\[
U:=\{ x\in M\ |\ \mathrm{dist}_g(x,A)<\delta \}\subseteq M.
\]
Thus, the inclusion $A\hookrightarrow M$ is a cofibration and a pair of compact manifolds satisfies the homotopy extension property which supplies two important facts for us:
\begin{enumerate}
\item[(i)] A continuous map $f:(M,A)\to (N,B)$ is homotopic to a smooth map  $h:(M,A)\to (N,B)$.
\item[(ii)] Let $\bundle{E}\to M$ be a bundle and $M=\coprod M_i$ be a disjoint decomposition in compact manifolds such that $\mathrm{rk}(\bundle{E}_{|M_i})=\mathrm{rk}(\bundle{E}_{|M_j})$ implies $i=j$. Choose points $x_i\in M_i\cap A$ and define the (discrete finite) set $P=\{ x_i\ |\ i\} $. If the restriction of a bundle $\bundle{E}\to M$ to $A \subseteq M$ is trivial, then there is a classifying map $\rho ^\bundle{E}:M\to \BU $ which is constant on $M_i\cap A$ for all $i$. In particular, $\rho ^\bundle{E}$ induces homomorphisms
\[
\rho ^\bundle{E}_*:H_k(M,A)\to H_k(\BU ,\rho ^\bundle{E}(P) )
\] 
and $\rho ^\bundle{E}(P)\cap \mathrm{BU}_n\subseteq \BU $ consists of at most one point which means that the inclusion map $H_k(\BU )\to H_k(\BU ,\rho ^\bundle{E}(P))$ is an isomorphism for all $k>0$.
\end{enumerate}
In order to see (i) we apply the smooth approximation theorem to $f_{|A}$ and obtain a smooth map $\tilde h:A\to B$ which is homotopic to $f_{|A}$. Using the homotopy extension property, this homotopy extends to all of $M$ and yields a map $\tilde h:M\to N$ whose restriction to $A $ is smooth and $f\simeq \tilde h:(M,A)\to (N,B)$. Applying the smooth approximation theorem (cf.~\cite[Chp.~II, theorem 11.8]{Bredon}) to $\tilde h$ yields a smooth map $h:(M,A)\to (N,B)$ which is homotopic to $f$. In order to see (ii) let $\tilde \rho ^\bundle{E}:M\to \BU $ be a classifying map, then $\tilde \rho ^\bundle{E}_{\ |A}:A\to \BU $ is a classifying map for the induced bundle on $A\subseteq M$. Since $\bundle{E}_{|A}$ is trivial, $\tilde \rho ^\bundle{E}_{\ |A}$ is homotopic to a map $\rho ^{\bundle{E}}:A\to  \BU $ which is constant on $M_i\cap A$ for all $i$. Since this homotopy extends to all of $M$, $\tilde \rho ^\bundle{E}$ is homotopic to a map $\rho ^\bundle{E}:M\to \BU $ which is constant on $M_i\cap A$. 

For a pair $(M,A)$ and $\theta \in H_{2*}(M,A)$ respectively $\theta \in H_{2*}(M)$ we denote by $\mathscr{V} (M,A;\theta )$ the set of Hermitian vector bundles $\bundle{E}\to M$ endowed with a Hermitian connection such that $\bundle{E}$ and its connection are trivial on an open neighborhood of $A$ and such that $\rho ^\bundle{E}_*(\theta )\neq 0$. If $\mathscr{V} (M,A;\theta ) $ is not empty, the \emph{relative K--area of} $\theta $ is defined by
\[
\ke (M_g,A;\theta ):=\left( \inf _{\bundle{E}\in \mathscr{V}(M,A;\theta )}\| R^\bundle{E}\| _g \right) ^{-1}\in (0,\infty ].
\]
We set $\ke (M_g,A;\theta )=0$ if $\mathscr{V} (M,A;\theta )=\emptyset $. The relative K--area of an odd class $\theta \in H_{2*+1}(M,A)$ [respectively $\theta \in H_{2*+1}(M)$] is defined by the relative K--area of the class $\theta \times [S^1]\in H_{2*}(M\times S^1,A\times S^1)$ [respectively $\theta \times [S^1]\in H_{2*}(M\times S^1)$]:
\[
\ke (M_g,A;\theta ):=\sup _{\mathrm{d}t^2}\ke (M_g\times S^1_{\mathrm{d}t^2},A\times S^1;\theta \times [S^1])
\]
and the relative K--area of a general class is the maximum of the even and the odd part. 
\begin{rem}
\label{rem151}
For a relative class of odd degree it is quite essential to take the supremum over all line elements on $S^1$. Remember that $\theta \in H_{2k+1}(M)$ has finite K--area iff $\ke (M_g\times S^1_{\mathrm{d}t^2};\theta \times [S^1])$ is finite for some line element $\mathrm{d}t^2$. This is an immediate consequence of remark \ref{rem_covering}. In general this is no longer true for $\theta \in H_{2k+1}(M,A)$ with $A\neq \emptyset $. For instance, consider the pair $(S^1,\{ x\} )$ where $x$ is a point on $S^1$ and suppose that $0\neq \theta \in H_1(S^1,\{ x\}  )$. Then $(S^1\setminus \{ x\} )\times S^1$ is an open cylinder which has an embedding via $f$ into $S^2$ in such a way that the closure of the cylinder in $S^2$ does not contain the north pole and the south pole. If $A\subseteq S^2$ is the complementary set to the open cylinder, $f:(S^1_{\mathrm{d}s^2}\times S^1_{\mathrm{d}t^2},\{ x\} \times S^1)\to (S^2_g,A)$ is a relative diffeomorphism with $\frac{1}{C}f^*g\leq \mathrm{d}s^2\oplus \mathrm{d}t^2\leq C\cdot f^*g$ on $(S^1\setminus \{ x\} )\times S^1$ for some constant $C>0$ where $g$ is the standard metric on $S^2$. Using $\ke (S^2_g,A;[S^2])\leq \ke (S^2_g)= 2$ shows
\[
\ke (S^1_{\mathrm{d}s^2}\times S^1_{\mathrm{d}t^2},\{ x\} \times S^1;\theta \times [S^1])\leq 2\cdot C.
\]
However, taking the supremum over all $\mathrm{d}t^2$ yields $\ke (S^1_{\mathrm{d}s^2},\{ x\};\theta )=\infty $.
\end{rem}
Considering the induced homomorphism $j_*:H_k(M )\to H_k(M,A )$ for the inclusion $j:(M,\emptyset )\to (M,A)$ we conclude from the commutative diagram
\[
\begin{xy}
\xymatrix{H_k(M)\ar[d]^{\rho _*^\bundle{E}}\ar[r]^{j_*}&H_k(M,A)\ar[d]^{\rho _*^\bundle{E}}\\
H_k(\BU )\ar[r]^{j_*'}&H_k(\BU ,\rho ^\bundle{E}(P) )}
\end{xy}
\] 
that $\mathscr{V}(M,A;j_*\theta )=\mathscr{V}(M,A;\theta )$ for all $\theta \in H_{k}(M)$ and even $k>0$, here we use that $j'_*$ is an isomorphism for all $k>0$. Hence, the relative K--area of $\theta \in H_k(M)$ and $j_*\theta $ coincide if $k>0$. The case $k=0$ is different because $j_*'$ is not injective. However, for any  $0\neq \theta \in H_0(M,A)$ there is a trivial bundle on $M$ with $\rho ^\bundle{E}(\theta )\neq 0$ which means that $\ke (M_g,A;\theta )=\infty $. Moreover, $\mathscr{V}(M,A;\theta )\subseteq \mathscr{V}(M;\theta )$ implies:
\[
 \ke (M_g;\theta )\geq \ke (M_g,A;\theta )=\ke (M_g,A;j_*\theta )
\]
for all $\theta \in H_k(M)$ and $k>0$. In some cases there is up to a constant an inequality in the opposite direction as the proposition below shows. We leave it to the reader to verify proposition \ref{proposition1} and \ref{proposition2} for the relative K--area. In particular, if $f:(M,A)\to (N,B)$ is a smooth map with $g\geq f^*h$ on $\Lambda ^2TM$, then the pull back by $f$ yields a map $f^*:\mathscr{V}(N,B;f_*\theta )\to \mathscr{V}(M,A;\theta )$ for even homology classes which implies
\[
\ke (M_g,A;\theta )\geq \ke (N_h,B;f_*\theta )
\]
for all $\theta \in H_*(M,A)$ respectively $\theta \in H_*(M)$. Since $M$ and $A$ are compact, finiteness of the relative K--area does not depend on the Riemannian metric $g$ and we obtain subgroups
\[
\H _k(M,A):=\{ \theta \in H_k(M,A)\ |\ \ke (M_g,A;\theta )<\infty \} .
\]
We have already observed that the K--area of a nontrivial class $\theta \in H_0(M,A)$ is infinite which implies $\H _0(M,A)=\{ 0\} $. Moreover, the above considerations show that the relative and the absolute K--area coincide if $A=\emptyset $, i.e.~$\H _k(M,\emptyset )=\H _k(M)$ for all $k$. Assuming that $f:(M,A)\to (N,B)$ is a continuous map it is homotopic to a smooth map, i.e.~the above inequality for the relative K--area shows
\[
f_*:\H _k(M,A)\to \H _k(N,B).
\]
Furthermore, $\H _*(M,A)$ depends only on the homotopy type of the pair $(M,A)$ and we can generalize this to relative diffeomorphisms:
\begin{prop}
\label{prop162}
Let $f:(M,A)\to (N,B)$ be a relative diffeomorphism of pairs of compact manifolds, then $f_*:\H _k(M,A)\to \H _k(N,B)$ is an isomorphism for all $k$.
\end{prop}
\begin{proof}
We assume at first that there is a constant $C>0$ with $\frac{1}{C}g\leq f^*h\leq C\cdot g$ on $T(M\setminus A)$ for metrics $g$ on $M$ and $h$ on $N$. The homomorphisms $f_*$ are well defined by the above consideration and injective because $f_*:H_k(M,A)\to H_k(N,B)$ are isomorphisms [note that $f:M/A\to N/B$ is a homeomorphism and $H_k(M,A)\cong $ $H_k(M/A,\left< A\right> )$]. Moreover, $f^*:\mathscr{V}(N,B;f_*\theta )\to \mathscr{V}(M,A;\theta )$ is bijective since any bundle $\bundle{E}\to M\setminus A$ which is trivial and has trivial connection on a neighborhood of $A$ extends to a bundle $\bundle{E}\to N$ which on a neighborhood of $B\subseteq N$ is trivial with trivial connection. Hence, the K--areas satisfy
\[
\frac{1}{C}\ke (M_g,A;\theta )\leq \ke (N_h,B;f_*\theta )\leq C\cdot \ke (M_g,A;\theta )
\]
which supplies the surjectivity of $f_*$ if $k$ is even. The statement for odd $k$ follows analogous from the even case and the relative diffeomorphism $f\times \mathrm{id}:(M\times S^1,A\times S^1)\to (N\times S^1,B\times S^1)$. Now lets come back to the general situation. Since $A\subseteq M$ is a compact submanifold which is a strong deformation retract of an open neighborhood in $M$, there is a compact submanifold $X\subseteq M$ with $A\subseteq \mathrm{int}(X)$ and a homotopy $H:X\times [0,1]\to X$ with $H(.,t)_{|A}=\mathrm{id}_A$, $H(.,0)=\mathrm{id}_X$ and $H(X,1)=A$. Since $f:(M,A)\to (N,B)$ is a relative diffeomorphism, $Y:=f(X)\cup B\subseteq N$ is a compact submanifold and $B\subseteq Y$ is a strong deformation retract with $B\subseteq \mathrm{int}(Y)$. Hence, $f:(M_g,X)\to (N_h,Y)$ is a relative diffeomorphism and there is a constant $C>0$  with $\frac{1}{C}g\leq f^*h\leq C\cdot g$ on $T(M\setminus X)$ which implies the isomorphism $f_*:\H _k(M,X)\to \H _k(N,Y)$. The constant $C>0$ exists because $f:M\setminus \mathrm{int}(X)\to N\setminus \mathrm{int}(Y)$ is a diffeomorphism of compact manifolds with boundary. Moreover, the inclusions $(M,A)\to (M,X)$ and $(N,B)\to (N,Y)$ induce isomorphisms on relative K--area which completes the proof.
\end{proof}
\begin{cor}
Let $(M,A)$ be a pair of compact manifolds and $U\subseteq M$ be an open set such that $\overline{U}\subseteq \mathrm{int}(A)$ and $(M\setminus U,A\setminus U)$ is a pair of compact manifolds. Then the inclusion induces isomorphisms
\[
i_*:\H _k(M\setminus U,A\setminus U)\to \H _k(M,A).
\] 
\end{cor}
Considering rational coefficients and the relative K--area for the Chern character we can extend the proof of proposition \ref{proposition7} to compact manifolds with boundary and obtain:
\begin{rem}
Let $f:\tilde M_{\tilde g}\to M_g$ be a normal Riemannian covering between oriented compact manifolds with boundary such that $f$ is trivial at the boundary, then
\[
\ke _\mathrm{ch}(\tilde M_{\tilde g},\partial \tilde M;f^!\theta )=\ke _\mathrm{ch}(M_g,\partial M;\theta )
\]
for all $\theta \in H_k(M,\partial M;\Q )$ where $f^!:H_k(M,\partial M;\Q )\to H_k(\tilde M,\partial \tilde M;\Q )$ is the transfer homomorphism defined by the pull back of cohomology classes and Poincar\'e duality using the orientation classes: $f^!=(\cap [\tilde M])\circ f^*\circ (\cap [M])^{-1}$.
\end{rem}

\begin{lem}
\label{lem163}
Let $(M,A)$ be a pair of compact manifolds and $\theta \in H_*(M,A)$, then for all $n\geq 2$:
\[
\theta \times [S^n]\in \H _*(M\times S^n,A\times S^n).
\]
\end{lem}
\begin{proof}
We use as in proposition \ref{prop_scalar} the index theorem for families of Dirac operators, but need to take care of the fact that the base space $M/A$ will not be a manifold in general. Without loss of generality $\theta \in H_k(M,A)$. Using the above results and considering suitable inclusion maps $S^n\to M\times S^n$, the claim follows for $\theta \in H_0(M,A)$ (note $H_0(M)\to H_0(M,A)$ is surjective). Hence, we suppose $k>0$ and start with the case that $k$ and $n\geq 2$ are even. Choose a Riemannian metric $h$ on $M$ and endow $S^n$ with the metric of constant sectional curvature $K_0=1$, then $g$ denotes the product metric on $M\times S^n$. If $\theta \times [S^n]$ is a torsion class, the relative K--area of $\theta \times [S^n]$ vanishes which proves the claim. Thus, suppose $\theta \times [S^n]$ is not a torsion class, then $\theta \times [S^n]$ yields a nonvanishing rational homology class. As above we will use a slight modification of the relative K--area to show the result. We denote by $\widetilde {\mathrm{ch}}(\bundle{F})\in H^{2*}(M,A;\Q )$ the relative Chern character of a relative bundle $\bundle{F}\to (M,A)$. Note that in our situation, $\bundle{F}\to (M,A)$ is endowed with a connection that is flat on a neighborhood of $A$, i.e.~we may use the curvature to define the relative Chern character:
\[
\widetilde{\mathrm{ch}}(\bundle{F}):=\left[ \left( \mathrm{tr}\left( e^{-\frac{R^\bundle{F}}{2\pi \mathbf{i}}} \right) ,0\right) \right]
\in H^{2*}(M,A;\Q )\subseteq H^{2*}_{dR}(M,A).
\]
By the arguments in the proof of proposition \ref{proposition6} we can find for any $\bundle{E}\in \mathscr{V}(M\times S^n,A\times S^n;\theta \times [S^n])$  multi indices $\alpha ,\beta $ with length $|\alpha |,|\beta |\leq \frac{k+n}{2}$ such that $\left< \widetilde{\mathrm{ch}}(\Lambda ^\alpha \Lambda ^\beta \bundle{E}),\theta \times [S^n]\right> \neq 0$. Because 
\[
\frac{(k+n)^2}{4}\| R^\bundle{E}\| _g\geq \| R^{\Lambda ^\alpha \Lambda ^\beta \bundle{E}}\| _g,
\] 
it suffices to show the existence of a constant $c>0$ with the property: Any relative bundle $\bundle{E}\to (M\times S^n,A\times S^n)$ with $\left< \widetilde{\mathrm{ch}}(\bundle{E}),\theta \times [S^n]\right> \neq 0$ satisfies $\| R^\bundle{E}\| _g\geq c$. In order to show the lower bound consider the following maps
\[
(M\times S^n,A\times S^n)\stackrel{q\times \mathrm{id}}{\longrightarrow } (M/A\times S^n,\left< A\right> \times S^n)\stackrel{j}{\longleftarrow }M/A\times S^n
\]
where $q$ denotes the quotient map $q:M\to M/A$ and $j=j'\times \mathrm{id}$ with $j':M/A\to (M/A,\left< A\right> )$. Then $q\times \mathrm{id}$ induces isomorphisms on $H_l$ with $(q\times \mathrm{id})_*=q_*\times \mathrm{id}$. Moreover, $j'_*:H_l(M/A)\to H_l(M/A,\left< A\right> )$ is an isomorphism for $l>0$, i.e.~the class $\eta :=(j_*')^{-1}q_*(\theta )\in H_k(M/A )$ satisfies
\[
j_*(\eta \times [S^n])=(q\times \mathrm{id})_*(\theta \times [S^n]).
\]
Because $\bundle{E}\to (M\times S^n,A\times S^n)$ is trivial on a neighborhood of $A\times S^n$ with trivial connection, there is an open neighborhood $U\subseteq M/A$ of $\left< A\right> $ such that $\bundle{E}_{|V_0}$ is trivial with trivial connection on $V_0:=U\times S^n$. Moreover, there are open sets $V_i\subseteq M/A\times S^n$ and maps $\psi _i$ such that $(V_i,\psi _i)$ trivialize $\bundle{E}$ and $V_i\cap (A\times S^n)=\emptyset $ for $i>0$. Hence, replacing $V_0$ by $U/A\times S^n$, $\bundle{E}$ determines a topological vector bundle $\bundle{E}'\to (M/A\times S^n,\left< A\right> \times S^n)$ with the property $(q\times \mathrm{id})^*\bundle{E}'\cong \bundle{E}$. The metric and the connection on $\bundle{E}$ induce a metric and connection on $\bundle{E}'$ for a suitable interpretation on the singular set $\left< A\right> \times S^n$. Using the general fact $j^*\widetilde{\mathrm{ch}}(-)=\mathrm{ch}(-)$ shows
\[
\left< \widetilde{\mathrm{ch}}(\bundle{E}),\theta  \times [S^n]\right> =\left< \widetilde{\mathrm{ch}}(\bundle{E}'), (q\times \mathrm{id})_*(\theta \times [S^n])\right> =\left< \mathrm{ch}(\bundle{E}'),\eta \times [S^n]\right> .
\]
Now let $(U_i)\subseteq M/A$ be a finite covering of $M/A$ by open contractible sets $U_i$. If $\bundle{E}'_{|U_i\times S^n}$ is stably nontrivial for some $U_i$, the smooth Hermitian vector bundle $\bundle{E}'_x:=\bundle{E}'_{|\{ x\} \times S^n}\to S^n$ satisfies $\left< \mathrm{ch}(\bundle{E}'_x) ,[S^n]\right> \neq 0$ for $x\in U_i\subseteq  M/A$. In fact, the results in the previous sections supply
\[
\| R^\bundle{E}\| _g\geq \| R^{\bundle{E}'_x}\| _0\geq \frac{1}{2}.
\] 
Thus, we can assume that $\bundle{E}'_{|U_i\times S^n}$ is stably trivial for all $U_i$. Define the bundle $\bundle{E}'':=\bundle{E}'\oplus \C ^m$ with the induced metric and connection where $m<\infty $ is chosen in such a way that $\bundle{E}''_{|U_i\times S^n}$ is trivial (note there are only finitely many $U_i$). Moreover, we obtain trivializations
\[
\bundle{E}''_{|U_i\times S^n}\to U_i\times S^n\times \C ^N,
\]
which shows that the structure group of the fiber bundle $\bundle{E}''\to M/A$ is $\mathrm{Diff}(S^n\times \C ^N; S^n)$. Hence, $M/A\times S^n$ is a smooth manifold over $M/A$ and $\bundle{E}''\to M/A\times S^n$ is a smooth vector bundle over $M/A\times S^n$ (cf.~\cite{AS4}). In fact, $\bundle{E}''$ is a family of smooth vector bundles over $S^n$ parameterized by the compact Hausdorff space $M/A$. If $\pi :M/A\times S^n\to S^n$ is the projection, the bundle $\pi ^*\spinor S^n$ is a family of spinor bundles over $S^n$ parameterized by $M/A$. In particular, $\pi ^*\spinor S^n\otimes \bundle{E}''$ is a family of Dirac bundles on $S^n$ parameterized by $M/A$ and the index of the associated Dirac operator $\dirac _{\bundle{E}''}^+:\Gamma (\pi ^*\spinor ^+S^n\otimes \bundle{E}'')\to \Gamma (\pi ^*\spinor ^-S^n\otimes \bundle{E}'')$ satisfies (cf.~\cite{AS4,LaMi}):
\[
\left< \mathrm{ch}(\mathrm{ind}(\dirac ^+_{\bundle{E}''})),\eta \right> =\left< \int _{S^n}\mathrm{ch}(\bundle{E}''),\eta \right> =\left< \mathrm{ch}(\bundle{E}'),\eta \times [S^n]\right> \neq 0 
\]
(note $\mathrm{ind}(\dirac ^+_{\bundle{E}''})\in K(M/A)$). This proves that there is a point $x\in M/A$ such that the kernel of the Dirac operator $\dirac _{\bundle{E}''_x}:\Gamma (\spinor S^n\otimes \bundle{E}_x'')\to \Gamma (\spinor S^n\otimes \bundle{E}_x'')$ is nontrivial where $\bundle{E}''_x$ is the smooth vector bundle on $\{ x\}\times S^n\subseteq M/A\times S^n$. Thus, the usual Bochner argument yields
\[
\| R^\bundle{E}\| _g\geq \| R^{\bundle{E}''_x}\| _0\geq \frac{1}{2}
\]
which completes the proof for even $k$, $n$. If $n\geq 2$ is odd, we apply the above methods to the manifold $S^n_0\times S^1_{\mathrm{d}t^2} $ and obtain by the arguments in the proof of proposition \ref{prop_scalar}:
\[
\| R^\bundle{E}\| _g\geq \| R^{\bundle{E}_x''}\| _{g_0\oplus \mathrm{d}t^2}\geq \frac{1}{2}.
\]
If $\theta $ is an odd homology class and $n$ even, we apply the even case to the manifold $(M\times S^1,A\times S^1)$ and the homology class $\theta \times [S^1]$. Note that the lower bound of the curvature does not depend on the choice of the metric $h$ on $M$, i.e.~the K--area remains finite by taking the supremum over all line elements of $S^1$. The remaining case ($k$ and $n$ odd) follows by applying the previous cases to the manifolds $(M\times S^1,A\times S^1)$ and $S^n_0\times S^1_{\mathrm{d}t^2}$ and using
\[
\ke _\mathrm{ch}(M_h\times S^n_0,-;\theta \times [S^n])\leq \ke _\mathrm{ch}(M_h\times T^2_{h'}\times S^n_0,-;\theta \times [T^2]\times [S^n])\leq 2
\]
for all metrics $h$, $h'$ on $M$ respectively $T^2$ (cf.~section \ref{stab_area}).
\end{proof}

\begin{prop}
Suppose $x\in M$, then the inclusion map $j:M\to (M,x)$ yields isomorphisms
\[
j_*:\H _k(M)\to \H _k(M,x)
\]
for all $k$. Moreover, if $(M,A)$ is a pair of compact manifolds such that $M/A$ admits a smooth structure, then
\[
q_*:\H _k(M,A)\to \H _k(M/A,\left< A\right> )\cong \H _k(M/A)
\]
is an isomorphism for all $k$ where $q:M\to M/A$ is the quotient map.
\end{prop}
\begin{proof}
By the above considerations, for each point $x\in M$ the inclusion  $j_*:\H _k(M)\to \H _k(M,x)$ is well defined and injective for all $k$, i.e.~the surjectivity remains to show. In particular, we obtain the claim if there is a constant $C>0$ with
\[
C\cdot \ke  (M_g,x;\theta )\geq \ke (M_g;j_*^{-1}\theta ).
\]
for any $\theta \in H_k(M,x)$ and $k>0$. Let $A=\overline{B_{\epsilon }(x)}\subseteq M$ be a small closed ball around the point $x$, then $A$ is contractible which means that there is a smooth map $f:M\to M$ which is constant on $A$ but satisfies $f\simeq \mathrm{id}$, i.e.~$f_*=\mathrm{id}$ on homology. Since $f$ is constant on $A$, the pull back by $f$ yields a bundle which has trivial connection on a neighborhood of $x$, i.e.~$f^*:\mathscr{V}(M;\eta )\to \mathscr{V}(M,x;\eta )$. There is a constant $C\geq 1$ with $C\cdot g\geq f^*g$ on $TM$ which implies the curvature estimate for the induced bundle and therefore,
\[
C\cdot \ke  (M_g,x;\eta )\geq \ke (M_g;\eta )
\]
for all $\eta \in H_{2*}(M)$. But the relative K--areas of $\eta \in H_k(M)$ and $j_*\eta \in H_k(M,x)$ coincide for $k>0$ which proves the first claim if $k$ is even. In order to conclude the same estimate for $\eta \in H_{2*+1}(M)$ we use the map $f:M\to M$ from above and consider $b:=f\times \mathrm{id}:M\times S^1\to M\times S^1$. Since $\mathrm{d}b$ has rank one on a neighborhood of $x\times S^1$, the pull back of a bundle by $b$ is flat on a neighborhood of $x\times S^1$ but the connection is in general not trivial. Consider the restriction of $b^*\bundle{E}$ to $x\times S^1$, then the connection differs from the trivial connection by a section $\alpha $ in $\mathrm{End}(b^*\bundle{E}_{|x\times S^1})\otimes T^*S^1$. Choosing a cut off function for $B_\epsilon (x)\times S^1\subseteq M\times S^1$, $\alpha $ extends to a section $M\times S^1\to \mathrm{End}(b^*\bundle{E})\otimes T^*(M\times S^1)$. In fact, $\nabla '=b^*\nabla -\alpha $ is a Hermitian connection on $b^*\bundle{E}$ which on a neighborhood of $x\times S^1$ is compatible with the trivialization of $b^*\bundle{E}$, i.e.~$(b^*\bundle{E},\nabla ')\in \mathscr{V}(M\times S^1,x\times S^1;\eta \times [S^1])$ if $(\bundle{E},\nabla )\in \mathscr{V}(M\times S^1;\eta \times [S^1])$. Since the $C^0$--bound of $\alpha $ depends on $\bundle{E}$, we need to rescale one of the line elements, i.e.~we fix at first $\mathrm{d}t^2$ and choose $\mathrm{d}s^2=y^2 \cdot \mathrm{d}t^2$ where $y$ satisfies $y\geq \max \{ 1, r\cdot \| \alpha \|  _{\mathrm{d}t^2}\} $ and $r$ means the radius of $S^1_{\mathrm{d}t^2}$ (note that $C\cdot (g\oplus \mathrm{d}s^2)\geq b^*(g\oplus \mathrm{d}t^2)$ on $T(M\times S^1)$). Then for $\bundle{E}\in \mathscr{V}(M\times S^1;\eta \times [S^1])$ the maximum of the curvatures are related by
\[
\| R^{\nabla '}\| _{g\oplus \mathrm{d}s^2}\leq \| R^{b^*\bundle{E}}\| _{g\oplus \mathrm{d}s^2}+C'\cdot \| \alpha \| _{\mathrm{d}s^2}\leq 2\max \left\{ C\cdot \| R^\bundle{E}\| _{g\oplus \mathrm{d}t^2},C'/r\right\}
\]
where $C'<\infty $ depends only on the choice of the cut off function for $B_\epsilon (x)\times S^1\subseteq M\times S^1$ and $C$ is the constant from above, i.e.~$C$ depends only on $f$. Hence, we conclude 
\[
\sup _{\mathrm{d}s^2}\ke (M_g\times S^1_{\mathrm{d}s^2},x\times S^1;\eta \times [S^1])\geq  \min \left\{ \frac{\ke (M_g\times S^1_{\mathrm{d}t^2} ,\eta \times [S^1])}{2C}, \frac{r}{2C'} \right\} 
\]
which together with the definition for odd homology classes proves the assertion (because $r\to \infty $ by taking the supremum over $\mathrm{d}t^2$ on the right hand side).

The second statement is an easy application of proposition \ref{prop162} because $q:(M,A)\to (M/A,\left< A\right> )$ is a relative diffeomorphism. 
\end{proof}

In general the connecting homomorphism $\partial _*:H_k(M,A )\to H_{k-1}(A )$ does not restrict to a homomorphism $\H _k(M,A)\to \H _{k-1}(A)$ as the following example shows. Let $M=\overline{B^2}$ be the 2--dimensional closed disk and $A=\partial M=S^1$ be its boundary, then
\[
\H _2(M,A)=\H _2(M/A)=\H _2(S^2)=\Z 
\] 
whereas $\H _*(M)=\{ 0\} $ and $\H _*(S^1)=\{ 0\} $, but $\partial _*:H_2(M,A)\to H_1(S^1)$ is an isomorphism. Thus, one may also consider the subgroup of $\H _*(M,A)$ whose image w.r.t.~the connecting homomorphism is contained in $\H _*(A)$. However, for this subgroup excision fails in general. Even if the connecting homomorphism restricts for all $k$ to homomorphisms $\partial _*:\H _k(M,A)\to \H _{k-1}(A)$, the corresponding homology sequence
\[
\cdots \longrightarrow \H _{k}(A)\stackrel{i_*}{\longrightarrow }\H _k(M)\stackrel{j_*}{\longrightarrow }\H _k(M,A)\stackrel{\partial _*}{\longrightarrow }\H _{k-1}(A)\longrightarrow \cdots 
\]
is only a chain complex and not exact in general. In order to see this, let $M=T^n$ be a torus and define $A:=M\setminus B_\epsilon (x)$ for a small open ball $B_\epsilon (x)\subseteq M$, then $\H _*(M)=\{ 0\}$. Moreover, $M/A\cong S^n$ supplies $\H _*(M,A)=\H _*(S^n)=\Z $ if $n\geq 2$. The homology sequence for the pair $(M,A)$ shows that $i_*:H_k(A)\to H_k(M)$ is injective for all $k<n$ and that the connecting homomorphism is trivial $\partial _*=0$ (the case $k=n-1$ uses the fact that $H_{n-1}(A)$ is torsion free and $M$ closed). Hence, $i_*:\H _k(A)\to \H _k(M)$ is injective and $\H _*(A)=\{ 0\} $ which proves that the homology sequence of this particular pair $(M,A)$ is not exact. 

\begin{defn}
Let $(X,Y)$ be a pair of topological spaces, then $\H _k(X,Y)$ denotes the set of all $\theta \in H_k(X,Y)$ for which there is a pair of compact manifolds $(M,A)$ and a continuous map $f:(M,A)\to (X,Y)$ such that $\theta =f_*(\eta )$ for some $\eta \in \H _k(M,A)$. Furthermore, $\He _k(X,Y)$ is the set of all $\theta \in \H _k(X,Y)$ with $\partial _*\theta \in \H _{k-1}(Y,\emptyset )$ where $\partial _*:H_k(X,Y)\to H_{k-1}(Y,\emptyset )$ is the connecting homomorphism.
\end{defn}
\noindent 
Simple exercises prove that $\He _k(X,Y)\subseteq \H _k(X,Y)$ are subgroups of $H_k(X,Y)$ for all $k$, that $\H _0(X,Y)=\{ 0\}$ and that $\H (X):=\H (X,Y)=\He (X,Y)$ for $Y=\emptyset $. Moreover, $\H _k(X,Y)$ coincides with the above definition if $(X,Y)$ is a pair of compact manifolds. Suppose that $h:(X,Y)\to (Z,T)$ is a continuous map of pairs, then the induced homomorphism restricts to homomorphisms $h_*:\H _k(X,Y)\to \H _k(Z,T)$ which follows from functoriallity $(h\circ f)_*=h_*f_*$. The naturallity of the connecting homomorphism shows that $h_*$ restricts further to homomorphisms $h_*:\He _k(X,Y)\to \He _k(Z,T)$. Hence, the functors $\H $ and $\He $ satisfy the dimension, additivity and homotopy axiom but obviously not exactness by the examples above. Thus, it remains to consider the excision axiom. Given subspaces $U\subseteq Y\subseteq X$ such that $\overline{U}\subseteq \mathrm{int}(Y)$, then the inclusion map $i:(X\setminus U,Y\setminus U)\to (X,Y)$ yields a well defined injective homomorphism
\[
i_*:\H _k(X\setminus U,Y\setminus U)\to \H _k(X,Y)
\]
for all $k$. We do not know if this homomorphism is surjective in general, however, it is surjective for pairs of compact manifolds and there is much evidence in the case that $(X,Y)$ is a CW--pair. Hence, the above observations can be summarized as follows.
\begin{thm}
$\He $ and $ \H $ are functors which determine semi homology theories on the category pairs of topological spaces and continuous maps in the following sense: $\He $ and $\H $ satisfy the dimension, additivity and homotopy axiom. When restricted to the category pairs of compact smooth manifolds and continuous maps, $\H $ satisfies the excision axiom. Moreover, natural connecting homomorphism exist for both functors and the corresponding homology sequence for pairs is a chain complex: 
\begin{itemize}
\item Choose trivial connecting homomorphisms for $\H $.
\item Choose the connecting homomorphism of singular homology for $\He $. 
\end{itemize} 
\end{thm}

\begin{rem}
Let $\overline{B^m}\subseteq \R ^m$ be the standard ball with boundary $S^{m-1}$, then $\H _m(\overline{B^m},S^{m-1})=\H _m(S^m) $ by the above proposition. Thus, if $(X,Y)$ is a pair of topological spaces, the relative Hurewicz homomorphism satisfies for all $m\geq 2$:
\[
h_m:\pi _m(X,Y)\to \H _m(X,Y)\subseteq H_m(X,Y).
\]
If $m\geq 3$, $\He _m(\overline{B^m},S^{m-1})=\H _m(\overline{B^m},S^{m-1})$ supplies that $\mathrm{Im}(h_m)\subseteq \He _m(X,Y)$.
\end{rem}
\begin{rem}
The K--area homology $\H $ stabilizes for compact manifolds under the suspension operation into reduced singular homology. In order to see this let $M$ be a compact manifold, $\Sigma ^kM=S^k\wedge M$ be the $k$--fold reduced suspension of $M$ and consider the projection map $p:S^k\times M\to \Sigma ^kM $. Suppose that $\eta \in H_{k+i}(\Sigma ^kM)$, then $\eta =p_*([S^k]\times \theta )$ for some $\theta \in H_i(M)$, but $[S^k]\times \theta $ has finite K--area for all $k\geq 2$ (cf.~remark \ref{rem8}) which proves that $\eta \in \H _{k+i}(\Sigma ^kM)$.  Hence, $\H _*(\Sigma ^kM)$ is the reduced singular homology of $\Sigma ^kM$ for all $k\geq 2$. If $M$ is connected, this may still be true for $k=1$, however, it fails in general as the example $\Sigma ^1S^0=S^1$ shows. Note that $\Sigma M$ is simply connected for connected manifolds and that $\H _j(\Sigma T^n)=H_j(\Sigma T^n)$ for all $j>0$, $n>0$ by an induction argument: $\Sigma S^1=S^2$ as well as $\Sigma T^n\simeq S^2\vee \Sigma T^{n-1}\vee \Sigma ^2T^{n-1}$.
\end{rem}
\section{K--area homology versus singular homology}
\label{versus}
In this section we show that the K--area homology determines the singular homology for pairs of compact manifolds. The key ingredient for this observation will be lemma \ref{lem163}. In fact, we will use the K\"unneth formula to recover $H_k(M,A)$ from $\H _{k+2}(M,A)$ and $\H _{k+2}(M\times S^2,A\times S^2)$ as well as to recover $H_k(f)$ from $\H _{k+2}(f)$ and $\H _{k+2}(f\times \mathrm{id}_{S^2})$. Before presenting the argument we show how to stabilize a semi homology theory by spheres. 

We denote by $\mathrm{Top}^2$ the category pairs of topological spaces with continuous maps and identify as usual the topological space $X$ with $(X,\emptyset )$. Moreover, $\mathrm{Ab}_{*}$ denotes the category of graded abelian groups with homomorphism. We assume that $P=(P_k)_{k\in \Z }:\mathrm{Top}^2\to \mathrm{Ab}_*$ is a sequence of functors which assigns to $(X,Y)\in \mathrm{Top}^2$ the abelian group $P_k(X,Y)$ and satisfies the following axioms
\begin{enumerate}
\item[(i)] $f,h:(X,Y)\to (Z,T)$ continuous with $f\simeq h$, then for all $k\in \Z $
\[
P(f)=P(h):P_k(X,Y)\to P_k(Z,T).
\]
\item[(ii)] The inclusions $X_\alpha \hookrightarrow \coprod _\alpha X_\alpha $ induce an isomorphism $\bigoplus _\alpha P_k(X_\alpha )\cong P_k(\coprod _\alpha X_\alpha )$ for all $k$.
\end{enumerate}
Moreover, we say that $P$ satisfies \emph{excision}, if for any $U\subseteq X$ with $\overline{U}\subseteq \mathrm{int}(Y)$, the inclusion map $j:(X\setminus U,Y\setminus U)\to (X,Y)$ yields isomorphisms
\[
P(j):P_k(X\setminus U,Y\setminus U)\to P_k(X,Y).
\]

Fix $n\geq 0$ and choose a base point $e_0\in S^n$. We denote by $\mathfrak{S}^n:\mathrm{Top}^2\to \mathrm{Top}^2$ the functor $(X,Y)\mapsto (X\times S^n,Y\times S^n)$ with morphisms $\mathfrak{S}^n(f):=f\times \mathrm{id}_{S^n}$, and define for notational simplicity the sequence of functors $P^n_k:=P_{k+n}\circ \mathfrak{S}^n:\mathrm{Top}^2\to \mathrm{Ab}_*$, i.e.~$P_k^n(X,Y)=P_{k+n}(X\times S^n,Y\times S^n)$ and $P^n(f):=P(f\times \mathrm{id}_{S^n})$ for continuous $f:(X,Y)\to (Z,T)$. Let
\[
\jmath :(X,Y)\to (X\times S^n,Y\times S^n),p\mapsto (p,e_0)
\]
be the inclusion map, then $P(\jmath ):P_{k+n}(X,Y)\to P^n_k(X,Y)$ determines a natural transformation. Thus, the quotient spaces
\[
\widehat{ P}^n_k(X,Y):=P^n_{k}(X,Y)/P(\jmath )(P_{k+n}(X,Y))
\]
are well defined abelian groups for all $k\in \Z $ and independent on the choice of $e_0$ if $n>0$ (by homotopy axiom). In the case $n=0$, $\widehat{P}_k^0$ will be naturally isomorphic to $P_k$ by the additivity axiom, i.e.~different choices of $e_0$ yield naturally isomorphic $\widehat{P}^0$. Suppose that $f:(X,Y)\to (Z,T)$ is continuous, then $P(\jmath )\circ P(f)=(P^nf)\circ P(\jmath )$ supplies that $P^n(f):P_k^n(X,Y)\to P_k^n(Z,T)$ descends to homomorphisms
\[
\widehat{P}^n(f):\widehat{P}_{k}^n(X,Y)\to \widehat{P}_{k}^n(Z,T)
\]
for all $k\in \Z $. Using these definitions, the functor $P:\mathrm{Top}^2\to \mathrm{Ab}_*$ determines a functor $\widehat{P}^n:\mathrm{Top}^2\to \mathrm{Ab}_*$ and the projection maps
\[
\Pr: P^n_k(X,Y)\to \widehat{P}^n_k(X,Y)
\]
are natural transformations of functors. The following remark is not vital for the results below, but worth mentioning to justify the stabilizing picture.
\begin{rem}
$\widehat{P}^n$ satisfies axioms (i) and (ii). Moreover if $P$ satisfies excision, then $\widehat{P}^n$ satisfies excision.
\end{rem}
\begin{proof}
The homotopy axiom is obvious, since $f\simeq h$ implies $f\times \mathrm{id}_{S^n}\simeq h\times \mathrm{id}_{S^n}$ and therefore $\widehat{P}(f)=\widehat{P}(h)$. The additivity follows from the additivity of $P$:
\[
\begin{split}
\bigoplus _\alpha \Bigl( P_k(X_\alpha\times S^n)/\jmath _*P_k(X_\alpha )\Bigl) &\cong \Bigl( \bigoplus _\alpha P_k(X_\alpha \times S^n){ \Bigl/}\bigoplus _\alpha \jmath _*P_k(X_\alpha )\Bigl) \\
&\cong P_k\Bigl(\coprod _\alpha X_\alpha \times S^n\Bigl) \Bigl/\jmath _*P_k\Bigl(\coprod _\alpha X_\alpha \Bigl) .
\end{split}
\]
Hence, it remains to show excision. Let $i:(X\setminus U,Y\setminus U)\to (X,Y)$ and $i':(X\times S^n\setminus U\times S^n,Y\times S^n\setminus U\times S^n)\to (X\times S^n,Y\times S^n)$ be the inclusion maps, then $P(i)$ and $P(i')$ are isomorphisms. Hence, $\jmath \circ i=(i\times \mathrm{id}_{S^n})\jmath $ and $(X\setminus U)\times S^n=X\times S^n\setminus U\times S^n$ yield the isomorphism 
\[
\widehat{P}^n(i):\widehat{P}_{k}^n(X\setminus U,Y\setminus U)\to \widehat{P}_{k}^n(X,Y)
\]
for all $k\in \Z $.
\end{proof}
We consider now the case of singular homology $P=H$. Fix a fundamental class $[S^n]\in H_n(S^n)$ where $[S^0]:=[S^0\setminus \{ e_0\} ]$, then the K\"unneth formula yields an isomorphism
\[
\begin{split}
H_{k+n}(X,Y)\oplus H_{k}(X,Y)&\to H_{k+n}(X\times S^n,Y\times S^n),\\
 (\eta ,\theta )&\mapsto \eta \times [\{ e_0\} ]+\theta \times [S^n].
\end{split}
\]
Since $\eta \times [\{ e_0\} ]=\jmath _*(\eta )$, the well defined homomorphism
\[
\Phi ^n_k:H_{k}(X,Y)\to \widehat{H}_{k}^n(X,Y),\ \theta \mapsto \Pr (\theta \times [S^n])
\]
is an isomorphism for all $k$. Moreover, by construction and using standard arguments in singular homology, $\Phi ^n_*$ becomes a natural isomorphism of functors $\mathrm{Top}^2\to \mathrm{Ab}_*$, i.e.
\[
\begin{xy}
\xymatrix{H_k(X,Y)\ar[r]^{f_*}\ar[d]^{\Phi ^n_k}&H_k(Z,T)\ar[d]^{\Phi ^n_k}\\
\widehat{H}_k^n(X,Y)\ar[r]^{f_{\widehat{*}}}&\widehat{ H}_k^n(Z,T)}
\end{xy}
\]
commutes for continuous $f:(X,Y)\to (Z,T)$ where $f_*=H_k(f)$ as usual and $f_{\widehat{*}}=\widehat{H}^n_k(f)$. The above construction works on suitable subcategories of $\mathrm{Top}^2$ just as well. In fact, we make the following assumptions on a subcategory $\mathrm{Cat}$ of $\mathrm{Top}^2$
\begin{itemize}
\item $\mathrm{Cat}$ is a \emph{full subcategory} of $\mathrm{Top}^2$, i.e.~the morphisms are continuous maps between objects of $\mathrm{Cat}$.
\item $(X,Y)\in \mathrm{Ob}(\mathrm{Cat})$ implies $X=(X,\emptyset )\in \mathrm{Ob}(\mathrm{Cat})$, $Y=(Y,\emptyset )\in \mathrm{Ob}(\mathrm{Cat})$ and $(X\times S^n,Y\times S^n)\in \mathrm{Ob}(\mathrm{Cat})$.
 \item $\{ e_0\}\in \mathrm{Ob}(\mathrm{Cat})$ where $e_0\in S^n$.
\end{itemize}
Of course we need to relax the additivity and excision axiom slightly to stay in the category $\mathrm{Cat}$. The category pairs of compact smooth manifolds and continuous maps, denoted by $\mathrm{Man}^2$, satisfies the above conditions. Now we consider the functor of K--area homology $\H _*: \mathrm{Man}^2\to \mathrm{Ab}_*$ and the functor $\widehat{\H }^n_*:\mathrm{Man}^2\to \mathrm{Ab}_*$ which is determined by $\H _*$, then $\widehat {\H }^0_k$ is naturally isomorphic to $\H _k$ for all $k\in \Z $. This is not true for $n>1$, in fact, the K--area homology stabilizes into singular homology: 
\begin{thm}
If $n\geq 2$, $\widehat{\H }^n_*$ is naturally isomorphic to the functor of singular homology on $\mathrm{Man}^2$. In particular,
\[
\Psi ^n_k:H_k(M,A)\to \widehat{\H }^n_k(M,A),\ \theta \mapsto \mathrm{Pr}(\theta \times [S^n])
\]
determines a natural isomorphism for all $n\geq 2$ and $k\in \Z $.
\end{thm}
The proof of this theorem is an easy application of the K\"unneth formula and lemma \ref{lem163}, because we deduce from lemma \ref{lem163} and the projection map $(M\times S^n,A\times S^n)\to (M,A)$ that
\[
\begin{split}
\H _{k+n}(M,A)\oplus H_{k}(M,A)&\to \H _{k+n}(M\times S^n,A\times S^n),\\
 (\eta ,\theta )&\mapsto \jmath _*(\eta )+\theta \times [S^n].
\end{split}
\]
is a well defined isomorphism if $n\geq 2$. A standard exercise yields $\Psi ^n_kf_*=f_{\widehat{*}}\Psi ^n_k$ which proves the naturallity of $\Psi ^n_k$. Thus, the K--area homology functor determines the functor of singular homology on $\mathrm{Man}^2$. In order to recover the connecting homomorphism we use the natural transformation
\[
\partial ^n  _*:\H _k(M,A)\to \H ^n_{k-1}(A),\ \theta \mapsto \partial _*\theta \times [S^n].
\]
Again $\partial ^n _*$ is well defined  by lemma \ref{lem163} if $n\geq 2$. Since $\partial ^n _*$ is natural, $\partial ^n _*$ descends to a natural homomorphism
\[
\partial _*=(\Psi ^n_{k-1})^{-1}\circ \Pr \circ \partial^n _*:\H _k(M,A)\to H _{k-1}(A),\ \theta \mapsto \partial _*\theta 
\]
which supplies the natural homomorphism
\[
\begin{split}
\overline{\partial }_*^n:\widehat{\H }_{k}^n(M,A)&\to \widehat{H}_{k-1}^n(A)\\
 &\theta +\H _{k+n}(M,A)\mapsto \partial _*(\theta )+H_{k+n-1}(A).
\end{split}
\]
In fact, we obtain from the definitions for the connecting homomorphism of singular homology $\Phi ^n_{k-1}\partial _*=\overline{\partial }^n_*\Psi ^n_k:H_k(M,A)\to \widehat{H}_{k-1}^n(A)$. Thus, we define the natural transformation
\[
\widehat{\partial }^n_*:=\Psi ^n_{k-1}\circ (\Phi ^n_{k-1})^{-1}\circ \overline{\partial }^n_*:\widehat{\H }_k^n(M,A)\to \widehat{\H }^n_{k-1}(A)
\]
and conclude $\Psi ^n_{k-1}\partial _*=\widehat{\partial }^n_*\Psi _k^n$ for all $k\in \Z $ and $n\geq 2$.
\begin{defn}
For $n\geq 2$ and all $k\in \Z $,
\[
\partial ^n  _*:\H _k(M,A)\to \H ^n_{k-1}(A),\ \theta \mapsto \partial _*\theta \times [S^n]
\]
are natural transformations and called the $n$--\emph{connecting homomorphism of K--area homology}
\end{defn}
Since $\H _*:\mathrm{Man}^2\to \mathrm{Ab}_*$ determines $\widehat{\H }^n_*:\mathrm{Man}^2\to \mathrm{Ab}_*$ and $\partial ^n_*$ determines $\widehat{\partial }_*^n$ we obtain by the above considerations:
\begin{thm}
The functor of K--area homology together with the natural transformation $\partial ^n _*$ determines the singular homology on $\mathrm{Man}^2$ for integer and rational coefficients. In fact for $n\geq 2$,
\[
\Psi ^n_k:(H_k,\partial _*)\to (\widehat{\H }^n_k,\widehat{\partial }^n_*)
\]
is a natural isomorphism of functors $\mathrm{Man}^2\to \mathrm{Ab}_*$ with $\Psi ^n_{k-1}\partial _*=\widehat{\partial }^n_*\Psi _k^n$ for all $k\in \Z $. 
\end{thm}

\chapter{Curvature gaps}
\begin{center}
\framebox{In this chapter all manifolds are connected and oriented.}
\end{center}
\section{$L^\infty$--curvature gaps of complex vector bundles}
\label{curv_gap}
A well known theorem is that flat vector bundles on a simply connected manifold are trivial. Using the  curvature norm $\| R^\bundle{E}\| _g$ of Hermitian vector bundles $\bundle{E}\to M$ introduced in section \ref{sec21} this observation can be generalized to almost flat vector bundles:
\begin{thm}[Gromov \cite{Gr01}] 
Let $(M,g)$ be a closed simply connected Riemannian manifold, then there is a constant $\epsilon >0$ with the following property: Any Hermitian vector bundle $\bundle{E}\to M$ with curvature $\| R^\bundle{E}\| _g<\epsilon $ is trivial.
\end{thm}
In general, one can not drop the condition on the fundamental group and still hope to conclude triviality of the vector bundle. Nevertheless, in the presence of a finite fundamental group we obtain certain positive curvature gaps meaning that a vector bundle with curvature below that curvature gap must be stably rational trivial. Gap phenomena for Yang--Mills fields have been investigated by Bourguignon et al.~in \cite{BLS}. In particular, they presented precise estimates for the curvature of Yang--Mills connections on $S^n$. In the results below we do not assume the existence of a Yang--Mills field on a vector bundle to compute curvature gaps. But, in order to avoid any confusion, we should point out that the $L^\infty $--norm in \cite{BLS} differs from our $L^\infty $--norm and moreover, we restrict ourselves to the case of complex vector bundles. Given a closed, oriented and connected Riemannian manifold $(M^n,g)$, and denote by $\| R^\bundle{E}\| _g$ the curvature norm of Hermitian vector bundles $\bundle{E}\to M$ introduced in section \ref{sec21}. We raise the following question: 
\begin{equation}
\label{gapproblem}
\begin{split} 
&\text{Is there an }\epsilon >0\text{ such that any Hermitian bundle }\bundle{E}\to M  \\ 
&\qquad \text{with curvature }\| R^\bundle{E}\| _g<\epsilon \text{ is stably rational trivial.}
\end{split}
\end{equation}
Note that a bundle $\bundle{E}$ is \emph{stably rational trivial} if there are integers $m>0$ and $k\geq 0$ in such a way that 
\[
\underbrace{\bundle{E}\oplus \cdots \oplus \bundle{E}}_{m-times}\oplus \C ^k
\]
is isomorphic to the trivial bundle. Of course if the $K$--group $K(M)$ is torsion free (for instance $M=S^n$, $M=\proj{\C }^m$, $M=T^n$,...), then stably rational trivial is the same as being stably trivial, where $\bundle{E}$ is stably trivial if $\bundle{E}\oplus \C ^k$ is trivial for some $k\geq 0$. Hence, if $K(M)$ is torsion free, we have to consider only the case $m=1$ in the above decomposition. But on $\proj {\R }^n$ for instance, the complex line bundle $\bundle{L}$ associated to $0\neq c\in H^2(\proj{\R} ^n;\Z )$ is a stably nontrivial bundle which admits a flat connection and is rational trivial ($\bundle{L}\oplus \bundle{L}$ is trivial). The question whether a bundle is trivial or only stably trivial can not be decided by standard characteristic classes. A very simple example of a stably trivial complex bundle on $S^5$ is determined by the nontrivial element in $[S^5,\mathrm{BU}_2] =\pi _4(\mathrm{U}_2)=\Z _2$. In fact, all characteristic classes associated to a stably trivial bundle vanish. Hence, in the presence of a nontrivial fundamental group we can not expect to conclude triviality of the bundle if it has almost flat curvature, and the property stably rational trivial seems to be the best choice to consider curvature gap problems. 

The largest $\epsilon $ with the property (\ref{gapproblem}) is called the \emph{vector bundle curvature gap} of $(M,g)$ and denoted by $\cg (M_g)$. We set $\cg (M_g)=0$ for monotonicity reasons if there is no such $\epsilon >0$. Conversely, if each complex vector bundle on $M$ is stably rational trivial, like on $S^{2n+1}$ or lens spaces, define $\cg (M_g)=\infty $. Note that questions $\cg (M_g)>0$ respectively $\cg (M_g)=0$ do not depend on the Riemannian metric $g$ or the usage of the $L^\infty $--norm on $\Omega ^2(M)$, however, these question depend on the choice of the norm on the fibers of $\mathrm{End}(\bundle{E})$. The reason for this more or less unexpected observation is that there is no uniform equivalence of norms in finite dimension. For instance, in case $\mathrm{rk}(\bundle{E})\to \infty $ we can not estimate the Yang--Mills term $\int |R^\bundle{E}|_g^2$ by $\| R^\bundle{E}\| _g^2$ although the first expression comes from an $L^2$--norm. To justify this statement simply consider the infima $\inf \| R^\bundle{E}\| _g$ and $\inf \int |R^\bundle{E}|^2_g$ over the set of bundles $\bundle{E}\to T^4$ which have nontrivial second Chern number $c_2(\bundle{E})\neq 0$ and vanishing first Chern class $c_1(\bundle{E})=0$. By the results in chapter \ref{chp2} we know $\inf \| R^\bundle{E}\| _g=0$ whereas $\inf \int |R^\bundle{E}|^2_g\geq 8\pi ^2$ by Chern--Weil theory. 

We denote by $\widetilde{H}_*(M;\Q )$ the reduced singular homology of $M$ with rational coefficients, i.e.~$\widetilde{H}_0(M;\Q )=\{ 0\} $ because $M$ is assumed to be connected.
\begin{lem}
\label{lem_gap}
Consider the subspace $V:=\widetilde{H}_{2*}(M;\Q )\subseteq \widetilde{H}_*(M;\Q )$, then
\[
\frac{1}{\cg (M_g)}=\max _{\theta \in V}\ke (M_g;\theta )=\max _{\theta \in V}\ke _\mathrm{ch}(M_g;\theta )\in [0,\infty ]
\]
where $\ke (M_g;.)$, $\ke _\mathrm{ch}(M_g;.)$ are the K--areas introduced above. In fact, the notation maximum indicates that the value $1/\cg (M_g)$ is always achieved by some $\theta \in V$.
\end{lem} 
\begin{proof}
If $V=\{ 0\}$, the K--areas vanish by definition, i.e.~it remains to show that $\cg (M_g)=\infty $. We know from $V=\{ 0\} $ that $H^{2*}(M;\Q )=H^0(M;\Q )=\Q $. Thus, the isomorphism $\mathrm{ch}:K(M)\otimes \Q \to H^{2*}(M;\Q )$ supplies that the reduced K--group is a torsion group which means that each complex bundle $\bundle{E}\to M$ is stably rational trivial proving $\cg (M_g)=\infty $. Now suppose that $V\neq \{ 0\}$, then there are complex vector bundles $\bundle{E}\to M$ which are not stably rational trivial by the same argument, in fact $\cg (M_g)<\infty $ and it suffices to consider classes $\theta \in V\setminus \{ 0\}$. Note that a bundle $\bundle{E}\to M$ is stably rational trivial if and only if one of the following equivalent conditions is satisfied
\begin{itemize}
\item $\mathrm{ch}(\bundle{E})=\mathrm{rk}(\bundle{E})\in H^{2*}(M;\Q )$.
\item The induced homomorphism $\rho ^\bundle{E}_*:\widetilde{H}_{2*}(M;\Q )\to \widetilde{H}_{2*}(\mathrm{BU};\Q )$ of the classifying map on reduced homology  is trivial.
\end{itemize}
This means that the curvature gap $\cg (M_g)$ is determined by
\[
\cg (M_g)=\inf _{\mathrm{rk}(\bundle{E})\neq \mathrm{ch}(\bundle{E})}\| R^\bundle{E}\| _g=\inf _{\rho ^\bundle{E}_*\neq 0}\| R^\bundle{E}\| _g
\]
where the infima are taken over all Hermitian vector bundles with Hermitian connections such that $\mathrm{ch}(\bundle{E})\neq \mathrm{rk}(\bundle{E})$ respectively $0\neq \rho ^\bundle{E}_*:\widetilde{H}_{2*}(M;\Q )\to \widetilde{H}_{2*}(\mathrm{BU};\Q )$. Thus,
\[
\left[ \sup _{\theta \in V}\ke _\mathrm{ch}(M_g;\theta )\right] ^{-1}=\inf _{\theta \in V}\left[ \inf _{\left< \mathrm{ch}(\bundle{E}),\theta \right> \neq 0}\| R^\bundle{E}\| _g\right] =\inf _{\mathrm{rk}(\bundle{E})\neq \mathrm{ch}(\bundle{E})}\| R^\bundle{E}\| _g=\cg (M_g)
\]
\[
\left[ \sup _{\theta \in V}\ke (M_g;\theta )\right] ^{-1}=\inf _{\theta \in V}\left[ \inf _{\rho ^\bundle{E}_*(\theta )\neq 0}\| R^\bundle{E}\| _g\right] =\inf _{\rho ^\bundle{E}_*\neq 0}\| R^\bundle{E}\| _g=\cg (M_g)
\]
prove the equalities and it remains to show that the values are attained by some $\theta \in V$. If $\mathscr{H}_{2*}(M;\Q )\neq V$, then we choose $\theta \in V\setminus \mathscr{H}_{2*}(M;\Q )$ and obtain
\[
\frac{1}{\cg (M_g)}=\ke (M_g;\theta )=\ke _\mathrm{ch}(M_g;\theta )=\infty .
\]
Hence, assume $\mathscr{H}_{2*}(M;\Q )=V$ and take a basis $\theta _1,\ldots ,\theta _s$ of the vector space $V$, then without loss of generality $\ke (M_g;\theta _i)\leq \ke (M_g;\theta _s)<\infty $ for all $i$. Now suppose that $ \theta \in V$ is arbitrary, then $\theta =\sum \lambda _i\theta _i$ for $\lambda _i\in \Q $. We conclude from proposition \ref{proposition1} and the scaling invariance in the homology class
\[
\ke (M_g;\theta )\leq \max \{ \ke (M_g;\lambda _i\theta _i)|i=1\ldots s\} \leq \ke (M_g;\theta _s)
\] 
which proves $\sup _\theta \ke (M_g;\theta )=\ke (M_g;\theta _s)$. The same argument applies to the K--area for the Chern character since the second part of proposition \ref{proposition1}  holds for $\ke _\mathrm{ch}(M_g;.)$ as well.
\end{proof}
\begin{thm}
\label{thm_gap}
A closed manifold $M$ has  curvature gap $\cg (M_g)>0$ if and only if $\mathscr{H}_{2*}(M)=\widetilde{H}_{2*}(M)$. In particular, $\cg (M_g)>0$ respectively $\cg (M_g)=0$ depend only on the homotopy type of $M$.
\end{thm}
Hence, a manifold with finite fundamental group has positive curvature gap (due to Gromov, cf.~\cite{Gr01} and remark \ref{funda}). On the other hand, there are plenty of manifolds with vanishing curvature gap. For instance, if there is a map of nonzero degree $M^{2k}\to T^{2k}$ for $k\geq 1$, then $\cg (M_g)=0$. Even positive scalar curvature on spin manifolds does not guarantee positivity of $\cg (M_g)$ as the example $M=S^2\times T^2$ shows. Here, $\mathscr{H}_2(S^2\times T^2)\neq H_2(S^2\times T^2)$ and the nontrivial almost flat bundles on $S^2\times T^2$ are induced by nontrivial almost flat bundles on $T^2$. Nevertheless, a finite fundamental group is not necessary for a manifold to have a positive curvature gap. In fact, the first Betti number can be arbitrarily large as the example
\[
N=S^k\times S^1\# \cdots \# S^k\times S^1,\ \ k>1
\]
shows. Because $N$ admits a metric of positive scalar curvature, $\mathscr{H}_{2*}(N)=\widetilde{H}_{2*}(N)$ by the results in chapter \ref{chp2} which implies $\cg (N_g)>0$ for all metrics $g$. Moreover, the existence of metrics with positive scalar curvature is not necessary for (spin) manifolds to have a positive curvature gap. For instance, let $X^{2n+1}$ be a rational homology sphere which admits a metric of nonpositive sectional curvature (like the Seifert--Weber space in $3$--dimensions), then $X$ does not admit a metric of positive scalar curvature. However, $\cg (X)=+\infty $ because any complex vector bundle on $X$ is stably rational trivial.   Conversely, if $X$ denotes the Seifert--Weber space, the manifold $X\times S^1$ satisfies $\cg (X\times S^1)=0$ since $\H _4 (X\times S^1)=\{ 0\} $. This follows from \cite[$4\frac{3}{5}$(v')]{Gr01} using the facts that $X\times S^1$ has residually finite fundamental group and $X\times S^1$ admits a metric of nonpositive sectional curvature. Nevertheless, the first cohomology group of manifolds with positive curvature gap is restricted by the following remark.
\begin{rem}
Suppose that $\cg (M_g)>0$, then $\theta \cup \eta =0$ for all $\theta ,\eta \in H^1(M;\Q )$.
\end{rem}

The proof of this remark is an easy application of Hopf's theorem. If $\theta \cup \eta \neq 0$, there is a map $f:M\to T^2$ with $f_*:H_2(M)\to H_2(T^2)$ surjective which implies $\H _2(M)\neq H_2(M)$. Remark \ref{rem243} yields now the following optimal result for the standard sphere $(S^{2n},g_0)$ of constant sectional curvature $K_{g_0}=1$: 
\begin{prop}
A Hermitian bundle $\bundle{E}\to S^{2n} $ with curvature $\| R^\bundle{E}\| _{g_0}<\frac{1}{2}$ has to be stably trivial, in fact $\cg (S^{2n}_{g_0})=\frac{1}{2}$. 
\end{prop}
Note that if $f:(M,g)\to (S^{2n},g_0)$ is a smooth map of nonzero degree and $g\geq f^*g_0$ on $\Lambda ^2TM$, then $\cg (M_g)\leq \frac{1}{2}$ which in some sense shows that the curvature gap on spheres is largest. However, in general it will be difficult to compute precise positive values for $\cg (M_g)$. A simple exercise supplies the following generalization of this inequality: Let $f:(M^m,g)\to (N^n,h)$ be a smooth map such that $f_*:\widetilde{H}_{2*}(M;\Q )\to \widetilde{H}_{2*}(N;\Q )$ is surjective and $g\geq f^*h$ on $\Lambda ^2TM$, then $\cg (M_g)\leq \cg (N_h)$. 
\section{Modified scalar curvatures}
\label{sec2}
In this section we review some observations made by Gursky, LeBrun in \cite{GurLe2} and generalized by Itoh in \cite{Itoh}. Let $M^n$ be a closed manifold and denote by $\met ^{2,\alpha }(M)$ the set of Riemannian $C^{2,\alpha }$--metrics on $M$ as well as by $C^{0,\alpha }(M)$ the space of $C^{0,\alpha }$--functions $M\to \R $ where $\alpha \in [0,1]$. Without further mentioning we assume throughout this section that
\[
F:\met ^{2,\alpha} (M)\to C^{0,\alpha }(M)
\]
satisfies
\begin{enumerate}
\item $F(g)\geq 0$ for all metrics $g$.
\item $F(\varphi \cdot g)=\frac{1}{\varphi }F(g)$ for all metrics $g$ and smooth functions $\varphi :M\to (0,\infty )$.
\end{enumerate}
Then $\mathrm{scal}^F:\met ^{2,\alpha }(M)\to C^{0,\alpha }(M)$ given by
\[
\mathrm{scal}^F_g:=\mathrm{scal}_g-F(g)
\]
is called a \emph{modified scalar curvature on} $M$ (cf.~\cite{GurLe2,Itoh}). The associated \emph{modified Yamabe invariant} is denoted by $\yam ^F(M_{[g]})$:
\[
\yam ^F(M_{[g]}):=\inf _{h\in [g]}\frac{\int \mathrm{scal}^F_h\cdot \mathrm{vol}_h}{\mathrm{Vol}(M,h)^{(n-2)/n}}.
\]
\begin{rem}
If $\dim M=n\geq 3$, then $h\in [g]$ is a critical point of
\[
h\mapsto  \frac{\int \mathrm{scal}_h^F\mathrm{vol}_h}{\mathrm{Vol}(M,h)^{(n-2)/n}}
\]
if and only if $\mathrm{scal}_h^F$ is a constant function on $M$.
\end{rem}
\begin{proof}
Because of the scaling invariance, it is sufficient to consider the problem for metrics $h\in [g]$ of total volume $1$ and the functional
\[
h\mapsto \int \mathrm{scal}^F_h\mathrm{vol}_h.
\]
The variation of $h\mapsto \mathrm{scal}_h$ and $h\mapsto \mathrm{vol}_h$ at $h$ in direction of a symmetric $(0,2)$ tensor $a$ is given by $\mathrm{scal}'_h(a)=-\left< \mathrm{Ric}_h,a\right> _h+\mathrm{div}_hZ$ (where $Z$ is a vector field) as well as by $\mathrm{vol}'_h(a)=\frac{1}{2}\mathrm{tr}_ha$ (cf.~\cite{Bes}). Since we vary only in the conformal class with metrics of constant volume, $a$ is given by $a=\kappa \cdot h$ where $\kappa :M\to \R $ is a smooth function with $\int \kappa \cdot \mathrm{vol}_h=0$. The conformal invariance of $F$ yields
\[
F'_h(a)=\frac{\mathrm{d}}{\mathrm{d}t}\Bigl| _{t=0} F(h+t\cdot \kappa h)=\frac{\mathrm{d}}{\mathrm{d}t}\Bigl| _{t=0} \frac{1}{1+t\kappa }F(h)=-\kappa \cdot F(h)
\]
which proves that $h$ is a critical point of the above functional if
\[
0=\int \left( \Bigl< \frac{\mathrm{scal}_h}{2}h-\mathrm{Ric}_h-\frac{F(h)}{2}h,a\Bigl> -F'_h(a)\right) \mathrm{vol}_h=\frac{n-2}{2}\int \kappa  \cdot \mathrm{scal}_h^F \mathrm{vol}_h
\]
for $a=\kappa \cdot h$ and all smooth functions $\kappa $ with vanishing mean value. 
\end{proof}
Hence, in order to understand $\yam ^F(M_{[g]})$ we are looking for metrics of constant modified scalar curvature. If $n=\dim M\geq 3$, a modified scalar curvature transforms under a conformal rescaling $h=\varphi ^{4/(n-2)}\cdot g$ by:
\[
\mathrm{scal}_h^F=\varphi ^{-\frac{n+2}{n-2}} \left( 4\frac{n-1}{n-2}\delta _g\mathrm{d}\varphi +\mathrm{scal}^F_g\cdot \varphi \right) 
\]
which means that the modified Yamabe problem can be treated in the same way as the original Yamabe problem. In fact, if $\yam ^F(M_{[g]}^n)<\yam (S^n)$, then there is a metric $h\in [g]$ of constant modified scalar curvature $\mathrm{scal}^F_h=\yam ^F({[g]})$. This was observed by Itoh in \cite{Itoh} who proved the existence of constant modified scalar curvature if $\dim M\geq 3$ and $\yam ^F(M_{[g]})<\yam (S^n)$. Gursky and LeBrun used in \cite{GurLe2} a $4$--dimensional version of modified scalar curvature to compute the Yamabe invariant of $\C \mathrm{P}^2$ and some other spaces. The following argument was pointed out to me by C.B.~Ndiaye. Now, $F(g)\geq 0$ yields  $\yam ^F(M_{[g]})\leq \yam (M_{[g]})$, i.e.~$\yam ^F(M_{[g]})<\yam (S^n)$ unless $M_{[g]}$ is conformal to the standard sphere (due to Aubin and Schoen cf.~\cite{LeePa}). However, if $M_{[g]}$ is the conformal standard sphere, choose the metric of constant scalar curvature $h$ which is a minimizer, i.e.~we conclude from $\mathrm{scal}_h\geq \mathrm{scal}_h^F$ and the definitions: $\yam ^F(M_{[g]})<\yam (M_{[g]})=\yam (S^n)$ unless $F\equiv 0$. But the case $F=0$ is the original Yamabe problem. Hence, we can summarize the above observations in the following proposition.
\begin{prop} [\cite{Itoh}]
\label{prop_itoh}
Suppose that $0<\alpha <1$. Then for each conformal class $[g]$ on $M$, there is a $C^{2,\alpha }$ metric $h\in [g]$ with constant modified scalar curvature:
\[
\mathrm{scal}^F_h=\yam ^F(M_{[g]}).
\]
\end{prop}
If $\dim M\geq 3$, this was proved by Itoh in \cite{Itoh}. The case $\dim M=2$ is a straightforward application of results by Kazdan and Warner (cf.~\cite{KazWar}). Note  that the modified scalar curvatures are related by
\[
e^{-2\phi }\mathrm{scal}_h^F=\mathrm{scal}_g^F-2\delta _g\mathrm{d}\phi 
\]
if  $h=e^{-2\phi }g$ is a conformal transformation on $M^2$. Hence, the $2$--dimensional case  follows from theorem 7.2 and lemma 9.3 in \cite{KazWar}.

Using H\"older's inequality, $F(h)\geq 0$ for all $h$ and $(L^\frac{n}{2})^*=L^\frac{n}{n-2}$, a standard exercise shows
\[
\mathcal{F}(M_{[g]}):=\sup _{h\in [g]}\frac{\int F(h)\cdot \mathrm{vol}_h}{\mathrm{vol}(M,h)^{(n-2)/n}}=\left[ \int _MF(h)^{n/2}\cdot \mathrm{vol}_h\right] ^{2/n}<+\infty 
\]
where the right hand side does not depend on the choice of $h\in [g]$. Thus, each $F$ yields an estimate of the standard Yamabe invariant:
\begin{equation}
\label{ineq12}
\yam ^F(M_{[g]})\leq \yam (M_{[g]})\leq \yam ^F(M_{[g]})+\mathcal{F}(M_{[g]}).
\end{equation}
The \emph{modified Yamabe invariant} of $M$ (respectively the \emph{modified sigma invariant}) is given by
\[
\yam ^F(M):=\sup _{[g]}\yam ^F(M_{[g]})\leq \yam (M).
\]
\begin{cor}
\label{cor_yam}
There is a metric $g$  with positive modified scalar curvature on $M$, i.e.~$\mathrm{scal}_g^F>0$, if and only if $
\yam ^F(M)>0$.
\end{cor}
\begin{rem}
\label{rem_yamabe}
Suppose $\yam ^F(M_{[g]})\leq 0$ and $h\in [g]$, then
\[
\mathrm{scal}_h^F\geq \yam ^F(M_{[g]})\cdot \mathrm{Vol}(M,h)^{-2/n}
\]
implies equality. Additionally, if $n=\dim M\geq 3$, there is a unique metric $h\in [g]$ with $\mathrm{scal}_h=\yam (M_{[g]})$ and $\mathrm{Vol}(M,h)=1$. 
\end{rem}
\begin{proof}
The $2$--dimensional case is an application of the Gau\ss --Bonnet theorem which yields $\int \mathrm{scal}_h^F\cdot \mathrm{vol}_h=\yam ^F(M_{[g]})$ for all $h$. Thus, suppose that $n\geq 3$. Let $g\in [g]$ be the minimizer of the modified Yamabe functional, i.e.~$g$ is a metric with constant modified scalar curvature $\mathrm{scal}_g^F=\yam ^F(M_{[g]})$ and $\mathrm{Vol}(M,g)=1$. Suppose that $h\in [g]$ satisfies $\mathrm{scal}_h^F\geq \yam ^F(M_{[g]})$ and $\mathrm{Vol}(M,h)=1$, then $h=\psi ^{-2}g$ and $\psi $ is bounded as follows $\int \psi ^{-2}\mathrm{vol}_g\leq [\int \psi ^{-n}\cdot \mathrm{vol}_g]^{2/n}=1$. We conclude from the behavior of the scalar curvature under conformal transformations:
\[
\begin{split}
\yam ^F(M_{[g]})\int \psi ^{-2}\cdot \mathrm{vol}_g\leq \int \psi ^{-2} &\mathrm{scal}_h^F\cdot \mathrm{vol}_g=\int \left[\mathrm{scal}_g^F-c(n)\frac{|\mathrm{d}\psi |^2}{\psi ^2}\right]\cdot \mathrm{vol}_g\\
&=\yam ^F(M_{[g]})-c(n)\int |\mathrm{d}\ln \psi |^2 \mathrm{vol}_g
\end{split}
\]
where $c(n)=(n-1)(n-2)$. Hence, we obtain from $\yam ^F(M_{[g]})\leq 0$, the claim $\psi =\mathrm{const}=1$, i.e.~a metric $h\in [g]$ which has total volume one satisfies $\mathrm{scal}_h^F\geq \yam ^F(M_{[h]})$ iff $h=g$.
\end{proof}

\section{$L^{n/2}$--curvature gaps of vector bundles}
\label{sec234}
In section \ref{curv_gap} we discussed the question of $L^\infty $--curvature gaps. To get analogous gap statements for conformal manifolds we consider a suitable $L^{n/2}$--norm of the curvature. In order to conclude optimal lower bounds for the curvature gaps we use a different norm on the $\Lambda ^2T^*M$ part of $\Lambda ^2T^*M\otimes \mathrm{End}(E)$. However, essential for the theory is again the usage of the operator norm on $\mathrm{End}(\bundle{E})$. 
Suppose that $(M^{n},g)$ is an oriented closed Riemannian manifold and $\bundle{E}\to M $ is a Hermitian vector bundle endowed with a Hermitian connection. Then $|R^\bundle{E}|_{g,op }:M\to [0,\infty )$ is the function which assigns to $x\in M$ the expression
\[
|R^\bundle{E}|_{g,op}(x):=\frac{1}{2}\inf _{e_1,\ldots ,e_n}\sum _{i,j=1}^n|R_{e_i,e_j}^\bundle{E}|_{op}
\]
where the infimum is taken over all orthonormal frames $e_1,\ldots ,e_n\in T_xM$.  Using $|R^\bundle{E}|^2_g=-\sum\limits _{i<j}\mathrm{tr}[(R_{e_i,e_j}^\bundle{E})^2]$, we obtain
\[
\frac{1}{\mathrm{rk}(\bundle{E})}\cdot |R^\bundle{E}|^2_g\leq |R^\bundle{E}|^2_{g,op}\leq \frac{n(n-1)}{2}\cdot |R^\bundle{E}|^2_g
\]
In Yang--Mills theory one minimizes $\int |R^\bundle{E}|^2_g$ respectively $\int |R^\bundle{E}|^{n/2}_g$ in the space of connections on $\bundle{E}$ where the bundle is fixed. In this section we are interested in minimizing $\int |R^\bundle{E}|_{g,op}^{n/2}$ over certain sets of Hermitian bundles. Since the functional
\[
\met (M)\to C^{0,\alpha }(M),\ g\mapsto |R^\bundle{E}|_{g,op} 
\]
satisfies the assumption in section \ref{sec2}, each Hermitian bundle $\bundle{E}$ gives rise to a modified scalar curvature.
\begin{lem}
\label{lem1}
Suppose that $M^n$ is a spin manifold and $\bundle{E}\to M$ satisfies
\[
\left< \widehat{\mathbf{A}}(M)\cdot \mathrm{ch}(\bundle{E}),[M]\right> \neq 0,
\] 
then the modified scalar curvature $\mathrm{scal}^\bundle{E}_g=\mathrm{scal}_g-4\cdot |R^\bundle{E}|_{g,op}$ can not be positive. Hence, the corresponding modified Yamabe invariant $\yam ^\bundle{E}(M_{[g]})$ is nonpositive for each conformal class $[g]$. 
\end{lem}
\begin{proof}
Let $\spinor M$ be the complex spinor bundle and consider the twisted Dirac bundle $\spinor M\otimes \bundle{E}$. Then the lowest eigenvalue of the twist curvature expression in the Bochner--Weitzenb\"ock formula of $\spinor M\otimes \bundle{E}$ satisfies
\[
\sum _{i<j}\gamma (e_i)\gamma (e_j)\otimes R^\bundle{E}_{e_i,e_j}\geq -|R^\bundle{E}|_{g,op}
\]
by a standard estimate.
Thus, $\mathrm{scal}_g^\bundle{E}>0$ implies $\ker (\dirac ^\bundle{E})=\{ 0\} $, a contradiction to the assumption on the index and corollary \ref{cor_yam} completes the proof. 
\end{proof}
\begin{rem}
\label{rem2}
The above modified scalar curvature is optimal for the standard sphere $S^{2m}$. In fact, consider the bundle $\bundle{E}:=\spinor ^+S^{n}$ with the connection induced by the standard metric $g_0$ which has sectional curvature $K_{g_0}=1$, then $\left< \mathrm{ch}(\bundle{E}),[S^{n}]\right> \neq 0$. Hence, $R^\bundle{E}_{e_i,e_j}=-\frac{1}{2}\gamma (e_i)\gamma (e_j)$ shows $|R^\bundle{E}|_{g,op}=\frac{1}{4}n(n-1)$ and $\mathrm{scal}_{g_0}^\bundle{E}=0$.
\end{rem}
\begin{defn}
Suppose $ \theta \in H_{2*}(M;\Q )$ and denote by $\mathscr{V}_\mathrm{ch}(M;\theta )$ the set of Hermitian vector bundles $\bundle{E}$ endowed with a Hermitian connection such that $\left< \mathrm{ch}(\bundle{E}),\theta \right> \neq 0$. Then the $L^{n/2}$--(respectively the \emph{conformal)--curvature gap of the conformal manifold $(M^n,[g])$ w.r.t.~the homology class} $\theta $ is defined by
\[
\ccg (M_{[g]};\theta ):=\left[ \inf _{\bundle{E}\in \mathscr{V}_\mathrm{ch}(M;\theta )}\int _M|R^\bundle{E}| ^{n/2}_{g,op}\cdot \mathrm{vol}_g\right] ^{\frac{2}{n}}\in [0,\infty ].
\]
The $[g]$--\emph{conformal curvature gap} $\ccg (M_{[g]})$ is the infimum of $\ccg (M_{[g]};\theta )$ over all $\theta \in \widetilde{H}_{2*}(M;\Q )$ where $\widetilde{H}_*(M;\Q )$ means reduced homology.  Moreover, if $M$ is even--dimensional, define
\[
\begin{split}
\ccg (M)&:=\sup _{[g]}\ccg (M_{[g]})\in [0,\infty ],\\
\cgv (M)&:=\sup _{[g]}\ccg  (M_{[g]};[M])\in [0,\infty ]
\end{split}
\]
where $\cgv (M)$ is the \emph{curvature gap volume} of $M$.
\end{defn}
Note that the previous definition makes sense for $\theta =0$ because the infimum over the empty set is $+\infty $. Conversely, if $\theta \neq 0$, then $\ccg (M_{[g]};\theta )<+\infty $ which means $\ccg (M_{[g]})<\infty $ as soon as $\widetilde{H}_{2*}(M;\Q )\neq \{ 0\} $. It is possible to extend this definition to homology classes of odd degree like it is done for the K--area in \cite{Gr01} respectively in chapter \ref{chp2}. However, for simplicity we pay attention only to even classes and will mainly focus on even dimensional manifolds. It is unclear at the moment if $\cgv (M)$ respectively $\ccg (M)$ have interesting values in dimension bigger than four, even $\cgv (S^{2m})<+\infty $ seems to be a nontrivial question for $m>2$. However, $\cgv $ and $\ccg $ have interesting applications for manifolds of dimension four. 

In $2$--dimensions the conformal curvature gap does not depend on the conformal class, because $|R^\bundle{E}|_{g,op}\cdot \mathrm{vol}_g$ is given by the $2$--form $|R^\bundle{E}|_{op}$ where $|R^\bundle{E}_{v,w}|_{op}(x)\in [0,\infty )$ denotes the maximal absolute eigenvalue of $R^\bundle{E}_{v,w}\in \mathrm{End}(\bundle{E}_x)$ for a positively oriented base $v,w\in T_xM$. Hence,
\[
\cgv  (M)=\ccg (M)=\ccg  (M_{[g]})=\ccg (M_{[g]};[M]) 
\]
for any conformal class on $M^2$ and moreover, $\cgv  (M)=0 $ if $M$ is a closed oriented surface of positive genus (in case $\chi (M)\leq 0$, $M$ has infinite K--area, and therefore $\ccg (M_{[g]})=0$). A standard exercise shows
\[
\Z \ni \left| \int c_1(\bundle{E})\right|\leq \frac{\mathrm{rk}(\bundle{E})}{2\pi }\cdot \int |R^\bundle{E}|_{op}
\]
for all $\bundle{E}$ on a surface $M$, i.e.~the rank of $\bundle{E}$ must increase with the same amount as  $\int |R^\bundle{E}|_{op}$ decreases to zero for nontrivial bundles. In order to compute the curvature gap volume of $S^2$ we use the previous lemma as well as the fact that $\yam (S^2_{[g]})=8\pi $ for any conformal class. In fact, consider the modified scalar curvature in lemma \ref{lem1}, then $\yam ^\bundle{E}(S^2_{[g]})\leq 0$ for any conformal class $[g]$ and bundle $\bundle{E}$ with $c_1(\bundle{E})\neq 0$. Thus, inequality (\ref{ineq12}) yields
\[
0\geq \sup _\bundle{E}\yam ^\bundle{E}(S^2_{[g]})\geq \sup _\bundle{E}\left( \yam (S^2_{[g]})-4\int |R^\bundle{E}|_{op}\right) =8\pi -4\cdot \cgv (S^2),
\]
i.e.~$\cgv (S^2)\geq 2\pi $ (the supremum is taken over all $\bundle{E}$ with $\left< \mathrm{ch}(\bundle{E}),[S^2]\right> \neq 0$). Conversely, if $g_0$ is the standard metric on $S^2$ and $\bundle{E}:=\spinor ^+S^2$ is equipped with the Levi--Civita connection of $g_0$, then $\int |R^\bundle{E}|_{op}=2\pi $ and $c_1(\bundle{E})=\pm 1$ which proves
\[
\cgv  (S^2)= 2\pi .
\]
Since Hermitian vector bundles on $S^2$ are classified by the first Chern class and the rank we can summarize the above results as follows:
\begin{prop}
\begin{enumerate}
\item[a)] A Hermitian bundle $\bundle{E}\to S^2$ is trivial if and only if it admits a Hermitian connection with curvature $\int |R^\bundle{E}|_{op}<2\pi $. Moreover, the constant $2\pi $ can not be improved.
\item[b)] Let $M$ be a closed oriented surface with $M\not \approx S^2$. Then for any $\epsilon >0$ there is a Hermitian bundle $\bundle{E} $ with connection which satisfies $c_1(\bundle{E})\neq 0$ and $\int |R^\bundle{E}|_{op}<\epsilon $, in this case $\mathrm{rk}(\bundle{E})>\frac{2\pi }{\epsilon }$.
\end{enumerate}
\end{prop}
In dimension $n>2$ the conformal curvature cap depends on the conformal class and it is less well behaved than the $L^\infty $--curvature gap respectively the K--area. However, a few methods from K--area apply to the conformal curvature gap quantity.
\begin{prop}
Let $f:(\tilde M^n,\tilde g)\to (M,g)$ be a $k$--fold Riemannian covering of oriented closed manifolds, then
\[
k^{-2/n}\cdot \ccg (\tilde M_{[\tilde g]};f^! \theta )\leq \ccg (M_{[g]};\theta )\leq \ccg (\tilde M_{[\tilde g]};f^! \theta )
\]
for all $\theta \in H_{2*}(M;\Q )$. Hence,  $\cgv (M)\leq \cgv (\tilde M)$.
\end{prop} 
\begin{proof}
The first inequality follows from the pull back of vector bundles. Since $\left< \mathrm{ch}(f^*\bundle{E}),f^! \theta \right> =\pm k\cdot \left< \mathrm{ch}(\bundle{E}),\theta \right> $ and
\[
\int _{\tilde M}|R^{f^*\bundle{E}}|_{\tilde g,op}^{n/2}\cdot \mathrm{vol}_{\tilde g}=\int _{\tilde M}\left[ |R^\bundle{E}|_{g,op}\circ f\right]^{n/2} \mathrm{vol}_{\tilde g}=k\cdot \int _M|R^\bundle{E}|_{g,op}^{n/2}\cdot \mathrm{vol}_g
\]
we deduce $\ccg (\tilde M_{[\tilde g]};f^! \theta )\leq k^{2/n}\ccg (M_{[g]};\theta )$. The other inequality follows from the push forward construction of vector bundles. We assume without loss of generality that $f$ is a normal covering, otherwise choose a finite normal covering $h:N\to \tilde M$ such that $f\circ h:N\to M$ is a normal covering (cf.~proof of proposition \ref{proposition7}). Suppose that $\bundle{E}$ is a bundle on $\tilde M$ with $\left< \mathrm{ch}(\bundle{E}),f^! \theta \right> \neq 0$, then $f_!\bundle{E}\to M$ determined by
\[
(f_!\bundle{E})_x=\bigoplus _{y\in f^{-1}(x)}\bundle{E}_y
\]
satisfies $\left< \mathrm{ch}(f_!\bundle{E}),\theta \right> \neq 0$ (cf.~proof of proposition \ref{proposition7}). $f^*f_!\bundle{E}$ is isomorphic to $\bigoplus _\sigma \sigma ^*\bundle{E}$ where the sum is taken over all isometric deck transformations $\sigma :(\tilde M,\tilde g)\to (\tilde M,\tilde g)$. Hence, we obtain
\[
\begin{split}
k\cdot \int _M|R^{f_!\bundle{E}}|^{n/2}_{g,op}\cdot \mathrm{vol}_g&=\int _{\tilde M}|R^{f^*f_!\bundle{E}}|_{\tilde g,op}^{n/2}\cdot \mathrm{vol}_{\tilde g}\\
&\leq \sum _\sigma \int _{\tilde M}|R^\bundle{E}|_{g,op}^{n/2}\circ \sigma \cdot \mathrm{vol}_{\tilde g}=k\cdot \int _{\tilde M}|R^\bundle{E}|_{\tilde g,op}^{n/2}\cdot \mathrm{vol}_{\tilde g}
\end{split}
\]
which yields $\ccg (M_{[g]};\theta )\leq \ccg (\tilde M_{[\tilde g]};f^!\theta )$. Note that the above inequality follows from the pointwise estimate
\[
|R^{\bigoplus \sigma ^*\bundle{E}}|_{\tilde g,op}^{n/2}(x)= \max _\sigma |R^\bundle{E}|_{\tilde g,op}^{n/2}(\sigma (x))\leq \left[ \sum _\sigma |R^\bundle{E}|_{\tilde g,op}^{n/2}\circ \sigma \right] (x).
\] 
The second statement is obvious because $f^![M]=[\tilde M]$. 
\end{proof}
\begin{rem}
\label{ccg_est}
Suppose that $ \theta ,\eta \in H_{2*}(M;\Q )$ and denote by $\ke _\mathrm{ch}(M_g;\theta )$ the K--area for the Chern character, then standard exercises show:
\begin{enumerate}
\item[(i)]
\[
\ccg (M_{[g]};\theta )\leq \frac{n(n-1)}{2}\frac{\mathrm{Vol}(M,g)^{2/n}}{\ke _\mathrm{ch} (M_{g};\theta )}.
\]
\item[(ii)] $\ccg (M_{[g]})\leq \frac{n(n-1)}{2}\cg (M_g)\cdot \mathrm{Vol}(M,g)^{2/n}$ for all metrics $g$.
\item[(iii)] For conformal classes $[g]$ and $[h]$ on $M$, there is a constant $C\geq 1$ with
\[
C^{-1}\cdot \ccg  (M_{[g]})\leq \ccg (M_{[h]} )\leq C\cdot \ccg  (M_{[g]} ).
\]
\item[(iv)] If $a,b\in \Q $, then $\mathscr{V}_\mathrm{ch}(M;a\cdot \theta +b\cdot \eta )\subseteq \mathscr{V}_\mathrm{ch}(M;\theta )\cup \mathscr{V}_\mathrm{ch}(M;\eta )$, i.e.
\[
\ccg (M_{[g]};a\cdot \theta +b\cdot \eta )\geq \min \{ \ccg (M_{[g]};\theta ),\ccg (M_{[g]};\eta )\} .
\]
\end{enumerate}
\end{rem}
Hence, $\ccg (M_{[g]};\theta )>0$ does not depend on the choice of the conformal class $[g]$ and
\[
\mathscr{P}_{2*}(M;\Q ):=\{ \theta \in H_{2*}(M;\Q )|\ \ccg (M_{[g]};\theta )>0\} \subseteq \mathscr{H}_{2*}(M;\Q )
\]
is a well defined subspace of $H_{2*}(M;\Q )$ closely related to the subspace $\mathscr{H}_*(M;\Q )$ determined by the K--area homology. However, the functorial properties of $\mathscr{P} _{2*}$ are an open problem. To be precise the induced homomorphism on homology of a (continuous) map $f:M\to N$ restricts to a homomorphism $f_*:\mathscr{H}_*(M;\Q )\to \mathscr{H}_*(N;\Q )$, but it is unknown what happens for $\mathscr{P}_{2*}$. In particular, it is unclear if $\ccg (M_{[g]})>0$ depends on the differentiable structure or only on the homotopy type of $M$. Nevertheless, we conclude analogously to lemma \ref{lem_gap} and theorem \ref{thm_gap} that
\[
\ccg (M_{[g]})=\min _{\theta \in \widetilde{H}_{2*}(M;\Q )}\ccg (M_{[g]};\theta ),
\]
i.e.~$\ccg (M_{[g]})>0$ if and only if $\mathscr{P} _{2*}(M;\Q )=\widetilde{H}_{2*}(M;\Q )$.
\begin{thm}
Suppose that $\bundle{E}\to M^n$ is a Hermitian vector bundle endowed with a Hermitian connection. If
\[
\int |R^\bundle{E}|_{g,op}^{n/2}\cdot \mathrm{vol}_g<\ccg (M_{[g]})^{n/2}
\]
for some metric $g$, then $\bundle{E}$ is stably rational trivial. 
\end{thm}
In some cases we can deduce $\ccg (M_{[g]})>0$ respectively $\cgv (M)>0$ on manifolds with positive scalar curvature. Lemma \ref{lem1} together with inequality (\ref{ineq12}) yield the estimate of Yamabe invariants by conformal curvature gaps. In order to cover the spin$^c$ case as well, we define a norm on the deRham cohomology $H^2_{dR}(M)$ which depends only on the conformal class $[g]$. If $\omega $ is a $2$--form, we denote by $|\omega |_{g,1}$ the function on $M^n$ given by
\[
| \omega  |_{g,1}(x):=\inf _{e_1,\ldots ,e_n}\sum _{i<j}|\omega (e_{i},e_{j})|
\]
where the infimum is taken over all orthonormal bases $e_1,\ldots ,e_n\in T_xM$. Then $|\, .\, |_{g,1}(x)$ determines a norm for  each $\Lambda ^2T_x^*M$ with the property that $| R^\bundle{E}| _{g,op}=|\mathbf{i}R^\bundle{E}|_{g,1}$ for Hermitian line bundles $\bundle{E}$. Moreover,
\[
|\omega |_g\leq |\omega |_{g,1}\leq \sqrt{\Bigl[ \frac{n}{2}\Bigl] }\, \cdot |\omega |_g
\]
where $|\, .\, |_g(x)$ is the standard norm on $\Lambda ^2T_x^*M$ induced by the scalar product $g_x$. For the second inequality we use the fact that each $2$--form $\omega \in \Lambda ^2T_x^*M$ can be written as $\omega =\sum _i\lambda _ie_{2i-1}^*\wedge e_{2i}^*$ for a suitable orthonormal basis $e_1^*,\ldots ,e_n^*$ of $T^*_xM$. In fact, a standard exercise shows $|\omega |_{g,1}=\sum _i|\lambda _i|$ whereas $|\omega |_g^2=\sum _i\lambda _i^2$. The above inequality is of importance in estimating the operator norm of Clifford multiplication with $\omega $:
\[
|\gamma (\omega )|_{op}=|\omega |_{g,1}\leq \sqrt{\Bigl[\frac{n}{2}\Bigl]}\, \cdot |\omega |_g
\]
and of course to estimate $|R^\bundle{E}|_{g,op}$ by $|\mathbf{i}R^\bundle{E}|_g$ if $\bundle{E}$ is a line bundle, here $\mathbf{i}R^\bundle{E}$ is regarded as a real $2$--form on $M$. We introduce the following $L^{n/2}$--norm on $2$--forms of $M$:
\[
\| \omega \| _{[g]}:=\left[ \int _M|\omega |_{g,1}^{n/2}\cdot \mathrm{vol}_g\right] ^\frac{2}{n}
\]
which depends only on the conformal class $[g]$. Suppose that $\theta \in H^2_{dR}(M)$ is represented by the closed $2$--form $\omega $, then
\[
\| \theta \| _{[g]}:=\inf _{\alpha \in \Omega ^1(M)}\| \omega +\mathrm{d}\alpha \| _{[g]}
\]
determines a norm on $H^2_{dR}(M)$. 
\begin{prop}
\label{pro3}
Suppose that $M^{n}$, $n=2m$, is a spin$^c$--manifold with associated class $c\in H^2(M;\Z )$. Then
\[
\yam (M_{[g]})\leq 4\pi  \cdot \| c_\R  \| _{[g]}+4\cdot \ccg (M_{[g]}; (\widehat{\mathbf{A}}(M)\cdot e^{-c/2})\cap [M])
\]
for any conformal class on $M$ where $c_\R \in H^2_{dR}(M)$ means the image of $c$ under $H^2(M;\Z )\to H^2_{dR}(M)$. Hence, if $\yam (M_{[g]})>4\pi \cdot \| c_\R \| _{[g]}$ for some conformal class, then
\[
\ccg (M_{[h]};(\widehat{\mathbf{A}}(M)\cdot e^{-c/2})\cap [M])>0 .
\]
for all conformal classes $[h]$.
\end{prop}
\begin{proof}
Let $\spinor _cM$ be the spin$^c$ bundle associated to $c$ and consider the twisted Dirac bundle $\spinor _cM\otimes \bundle{E}$, then
\[
(\dirac ^\bundle{E})^2=\nabla ^*\nabla +\frac{\mathrm{scal}_g}{4}+\frac{\mathbf{i}}{2}\gamma (\omega ) +\sum _{i<j}^n\gamma (e_i)\gamma (e_j)\otimes R^\bundle{E}_{e_i,e_j}
\]
where $\gamma (.)$ means Clifford multiplication and $\omega $ is the curvature of the associated line bundle, i.e.~$\omega $ is a representative of $2\pi c_\R $. We have already observed that $|\gamma (\omega )|_{op}= |\omega |_{g,1}$, i.e.~we conclude from the estimate in the proof of lemma \ref{lem1}  that $\mathrm{scal}^\bundle{E}_g:=\mathrm{scal}_g-2\cdot | \omega | _{g,1}-4|R^\bundle{E}|_{g,op}$ is a modified scalar curvature and if $\mathrm{ind}\ \dirac ^\bundle{E}\neq 0$, there is no metric which satisfies $\mathrm{scal}_g^\bundle{E}>0$. Thus, for any bundle $\bundle{E}$ with $\left< \mathrm{ch}(\bundle{E})\cdot \widehat{\mathbf{A}}(M)\cdot e^{-c/2},[M]\right> \neq 0$, we obtain from $\yam ^\bundle{E}(M_{[g]})\leq 0$ and (\ref{ineq12}) the estimate
\[
\begin{split}
\yam (M_{[g]})&\leq \Bigl\| 2\cdot |\omega |_{g,1}+4\cdot |R^\bundle{E}|_{g,op}\Bigl\| _{L^{n/2}}\\
&\leq 2\cdot \| \omega \| _{[g]} +4\left[\int _M |R^\bundle{E}|_{g,op}^{n/2}\cdot \mathrm{vol}_g\right] ^{2/n}.
\end{split}
\]
Take the infimum over all bundles $\bundle{E}\in \mathscr{V}_\mathrm{ch}(M;(\widehat{\mathbf{A}}(M)\cdot e^{-c/2})\cap [M])$ and over all $\omega $ which represent $2\pi c_\R $ provides the claim. Note for each closed $\omega $ with $[\omega ]=2\pi c_\R $, there is a connection on the associated line bundle with curvature $2$--form $\omega $.
\end{proof}
\begin{cor}
Suppose that $M^{2m}$ is a closed spin manifold such that either $m\leq 2$ or the $\widehat{A}$--class satisfies $\widehat{\mathbf{A}}(M)=1$, then
\[
\yam (M)\leq 4\cdot \cgv (M).
\]
\end{cor}
We can use the above proposition to compute the conformal curvature gap for a couple of conformal manifolds. For instance, if $g_0$ denotes the standard metric on $S^{2m}$, then remark \ref{rem2} shows $\ccg  (S^{2m}_{[g_0]})\leq \frac{m(2m-1)}{2}\omega ^{1/m}_{2m}$ and the above proposition together with $\yam (S^{2m}_{[g_0]})=2m(2m-1)\omega _{2m}^{1/m}$ yield the opposite inequality which means:
\[
\ccg (S^{2m}_{[g_0]})=\frac{m(2m-1)}{2}\omega ^{1/m}_{2m}=\frac{\yam (S^{2m}_{[g_0]})}{4}=\frac{\yam (S^{2m})}{4}.
\]
Here, $\omega _{2m}$ means the volume of the $2m$--dimensional standard sphere. Moreover, we have already observed that $\yam (S^2)=4\cdot \cgv (S^2)$. The existence of self dual connections will also provide $\yam (S^4)=4\cdot \cgv (S^4)=8\sqrt{6}\, \pi $, but the value $\cgv (S^{2m})$  remains an open problem for $m>2$. Another example is obtained for the standard conformal class  $[g_0\oplus g_0]$ on the manifold $N:=S^{2m}\times S^{2m}$ where $g_0$ is the standard metric on  $S^{2m}$. Since $g:=g_0\oplus g_0$ is Einstein, we obtain
\[
\yam (N_{[g]})=4m(2m-1)\omega ^{1/m}_{2m}.
\]
where $\omega _{2m}$ means again the volume of the standard unit sphere $S^{2m}$. On the other hand if $\pi _j:N=S^{2m}\times S^{2m}\to S^{2m}$, $j=1,2$, are the projections and $\spinor ^+$ the positive complex spinor bundle of $S^{2m}_{g_0}$, the bundle $\bundle{E}=\pi ^*_1\spinor ^+\otimes \pi _2^*\spinor ^+ $ has curvature $|R^\bundle{E}|_{g,op}=m(2m-1)$ and satisfies $\left< \mathrm{ch}(\bundle{E}),[N]\right> \neq 0$. Thus, $\ccg (N_{[g]},[N])\leq m(2m-1)\omega ^{1/m}_{2m}$ together with the opposite inequality from proposition \ref{pro3} supplies
\[
 \ccg (N_{[g]};[N])= \frac{\yam  (N_{[g]})}{4}=m(2m-1)\omega ^{1/m}_{2m}.
\]
This equality and the results in the section below (theorem \ref{thm256}) yield an estimate of the curvature gap volume of $S^2\times S^2$
\[
4\pi \leq \cgv (S^2\times S^2)\leq 4\sqrt{3}\, \pi .
\]
In fact $4\pi <\cgv (S^2\times S^2) $, because $16\pi <\yam (S^2\times S^2)$ by one of the main results in \cite{BoWaZi}. Moreover, the above estimates, the remark below and the bundle $\pi _2^*(\spinor ^+)$ yield
\[
\ccg (N_{[g]})=\ccg (N_{[g]};1\times [S^{2m}])=\frac{m(2m-1)}{2}\omega ^{1/m}_{2m},
\] 
i.e.~$\ccg (N_{[g]})=2\pi $ if $m=1$. Here we use that the value $\ccg (N_{[g]})=\ccg (N_{[g]};\theta _i)$ is achieved by some $\theta _i$ if $\theta _1,\theta _2 ,\theta _3$ is a basis of $\widetilde{H }_{2*}(N;\Q )$. Hence,
\[
2\pi \leq \ccg (S^2\times S^2)\leq 4\sqrt{3}\, \pi .
\]
\begin{rem}
Let $(X^k\times Y^l,g\oplus h)$ be a Riemannian product and $\theta \in H_{2*}(Y;\Q )$, then
\[
\ccg (Y_{[h]};\theta )\leq \mathrm{Vol}(Y,h)^{\frac{2k}{l(k+l)}}\mathrm{Vol}(X,g)^{-\frac{2}{k+l}}\ccg (X\times Y_{[g\oplus h]};1\times \theta ).
\]
\end{rem}
\begin{proof}
$i_x:Y\to  X\times Y$ denotes the inclusion map $y\mapsto (x,y)$. If $f:X\times Y\to [0,\infty )$ is a continuous function, there is some $x\in X$ with
\[
\int _Y(f\circ i_x)\mathrm{vol}_h\leq \int _Y(f\circ i_z)\mathrm{vol}_h
\] 
for all $z\in X$. Hence, Fubini's theorem yields for this $x\in X$
\[
\int _Y(f\circ i_x)\mathrm{vol}_h\leq \frac{1}{\mathrm{Vol}(X,g)}\int _{X\times Y}f\cdot \mathrm{vol}_{g\oplus h}.
\]
Suppose $\bundle{E}\to X\times Y$ is a bundle with $\left< \mathrm{ch}(\bundle{E}),1\times \theta \right> \neq 0$, then choose the point $x\in X$ as above for the continuous function $|R^\bundle{E}|^{l/2}_{g\oplus h,op}$. The pull back $i^*_x\bundle{E}$ satisfies $\left< \mathrm{ch}(i^*_x\bundle{E}),\theta \right> \neq 0$ and $|R^{i^*_x\bundle{E}}|_{h,op}\leq |R^\bundle{E}|_{g\oplus h,op}\circ i_x$. Hence,
\[
\begin{split}
\ccg (Y_{[h]};\theta )&\leq \left( \int _Y|R^{i^*_x\bundle{E}}|_{h,op}^{l/2}\mathrm{vol}_h\right) ^{2/l}\leq \left( \int _Y|R^\bundle{E}|^{l/2}_{g\oplus h,op}\circ i_x\mathrm{vol}_h\right) ^{2/l}\\
&\leq \left( \frac{1}{\mathrm{Vol}(X,g)} \int _{X\times Y}|R^\bundle{E}|_{g\oplus h,op}^{l/2}\mathrm{vol}_{g\oplus h}\right) ^{2/l}\\
&\leq \left( \frac{\mathrm{Vol}(Y,h)^{k/l}}{\mathrm{Vol}(X,g)} \int _{X\times Y}|R^\bundle{E}|_{g\oplus h,op}^{\frac{k+l}{2}}\mathrm{vol}_{g\oplus h}\right) ^{\frac{2}{k+l}}. \qedhere
\end{split}
\]

\end{proof}
\section{$L^2$--curvature gaps and Yamabe invariants on $4$--manifolds}
\label{sec244}
Throughout this section $M$ is a closed oriented $4$--dimensional manifold. Since the values for the conformal curvature gap and the curvature gap volume do not depend on the orientation of $M$ we assume without loss of generality that $b_+(M)\geq b_-(M)$ where $b_+$ and $b_-$ denote the dimension of the space of self dual respectively anti--self dual harmonic $2$--forms on $M$. We observed that
\[
\yam (M)\leq 4\cgv (M)
\]
if $M$ is spin. The example $M=\C P^2$ will show that this inequality does not hold without the spin assumption. However, the following lemma provides suitable generalizations of this inequality.
\begin{lem}
Let $[g]$ be a conformal class on $M$, then
\[
\ccg (M_{[g]};\eta +[M])\leq \ccg (M_{[g]};[M])\leq \cgv (M)
\]
for all $\eta \in H_2(M;\Q )$.
\end{lem} 
\begin{proof}
Suppose that $\bundle{E}$ satisfies $\left< \mathrm{ch}(\bundle{E});[M]\right> \neq 0$, then $\bundle{F}:=\bundle{E}\oplus \bundle{E}^*$ has trivial first Chern class because $c_1(\bundle{E})=-c_1(\bundle{E}^*)$ but $\mathrm{ch}_2(\bundle{F})=2\mathrm{ch}_2(\bundle{E})$. Thus,
\[
\left< \mathrm{ch}(\bundle{F}),\eta +[M]\right> =2\left< \mathrm{ch}(\bundle{E}),[M]\right> \neq 0
\]
and $|R^\bundle{F}|_{g,op}=|R^\bundle{E}|_{g,op}$ complete the proof.
\end{proof}
Hence, if $c\in H^2(M;\Z )$ is associated to a spin$^c$ structure on $M$, then proposition \ref{pro3} yields for any conformal class $[g]$ on $M$:
\begin{equation}
\label{ineq365}
\yam (M_{[g]})\leq 4\pi \| c_\R \| _{[g]}+4\cgv (M).
\end{equation}
Of course, if $c^2\neq \tau (M)$ for the signature $\tau (M)$, the Atiyah--Singer index theorem shows $\yam (M_{[g]})\leq 4\pi \| c_\R \| _{[g]}$ for any conformal class (cf.~\cite{GurLe2}). Note that
\[
\widehat{\mathbf{A}}(M)\cdot e^{-c/2}\cap [M]=\frac{1}{8}\left( c^2 -\tau (M)\right) -\eta +[M].
\]
where $\eta \in H_2(M;\Q )$ means the Poincar\'e dual of $c_\Q /2$, i.e.~$\yam (M_{[g]})\leq 4\pi \| c_\R \| _{[g]}$  follows easily from proposition \ref{pro3}. However, in case $c^2=\tau (M)$ we obtain a couple of estimates for the curvature gap volume.

We compute the first precise values for the complex projective plane. Let $[g_{FS}]$ be the conformal class of the Fubini--Study metric on $\proj{\C }^2$ and $c\in H^2(\proj{\C };\Z )$ be a generator. Consider the line bundle $\bundle{E}$ associated to $c$ which is endowed with the connection that has harmonic curvature $2$--form $\mathbf{i}R^\bundle{E}=\omega $, then $\omega $ is self dual which means
\[
\int |R^\bundle{E}|^2_{g,op}\cdot \mathrm{vol}_g=\int |\omega |_{g,1}^2\cdot \mathrm{vol}_g\leq 2\int |\omega |^2_g\cdot \mathrm{vol}_g=2\int \omega \wedge \omega =8\pi ^2\int c^2=8\pi ^2.
\]
This proves $\ccg (\C P^2_{[g_{FS}]};\eta )\leq 2\sqrt{2}\, \pi $ for $0\neq \eta \in H_2(\C P^2;\Q )$ as well as $\ccg (\C P^2_{[g_{FS}]};[\C P^2])\leq 2\sqrt{2}\, \pi $. Conversely, consider the spin$^c$ structure associated to $c$, then $\| c_\R \| ^2 _{[g_{FS}]}\leq 2c^2=2$ and proposition \ref{pro3} yield
\[
\yam (\proj{\C }^2_{[g_{FS}]})=12\sqrt{2}\, \pi \leq 4\sqrt{2}\, \pi + 4\cdot \ccg (\proj{\C }^2_{[g_{FS}]};[\proj{\C }^2]-\eta /8)
\]
where $\eta $ is the Poincar\'e dual of $c_\Q /2$ (cf.~\cite{LeB5,GurLe2}). Thus, the above lemma supplies
\[
\ccg (\proj{\C }^2_{[g_{FS}]};[\proj{\C }^2])\geq \ccg (\proj{\C }^2_{[g_{FS}]};[\proj{\C }^2]-\eta /8) \geq 2\sqrt{2}\, \pi .
\]
Since $[\C P^2]$ and $[\C P^2]-\eta /8$ determine a basis of $\widetilde{H}_{2*}(\C P^2;\Q )$ we conclude from  (iv) in remark \ref{ccg_est}:
\begin{cor}
\[
\ccg (\proj{\C }^2_{[g_{FS}]})=\ccg (\proj{\C }^2_{[g_{FS}]};\eta )=\ccg (\proj{\C }^2_{[g_{FS}]};[\proj{\C }^2] )=2\sqrt{2}\, \pi .
\]
\end{cor}

By the above definitions the curvature gap volume supplies a suitable curvature gap quantity on $4$--manifolds with Betti number $b_2=0$. In fact, if $b_2=0$ and $\cgv (M)=0$, then there are bundles $\bundle{E}\to M$ which are arbitrary flat in the $L^{2}$--norm of the curvature. Conversely, if $\cgv (M)>0$, then for all $0<\epsilon <\cgv (M)$ there is a conformal class $[g]$ on $M$ with the following property: Any Hermitian bundle $\bundle{E}\to M$ with curvature
\[
\int |R^\bundle{E}|^2_{g,op}\cdot \mathrm{vol}_g<(\cgv (M)-\epsilon )^2
\] 
is stably rational trivial. By the definitions this observation remains true for manifolds with $b_2\neq 0$ if we replace $\cgv (M)$ by $\ccg (M)$. Hence, we are interested in estimating $\ccg (M)$ from below by $\cgv (M)$:
\begin{prop}
\label{prop243}
Suppose that $M^4$ satisfies $b_{-}=0$, then
\[
\ccg (M_{[g]})\geq \frac{1}{2}\ccg (M_{[g]};[M]).
\]
In particular, $\ccg (M)\in [\frac{1}{2},1]\cdot \cgv (M)$.
\end{prop}
\begin{proof}
It suffices to show $\ccg (M_{[g]};\theta )\geq \frac{1}{2}\ccg (M_{[g]};[M])$ for $0\neq \theta \in H_2(M;\Q )$. Suppose $\bundle{E}\in \mathscr{V}_\mathrm{ch}(M;\theta )$, then either $\left< \mathrm{ch}(\bundle{E}),[M]\right >\neq 0$ or $\left< \mathrm{ch}(\bundle{E}\otimes \bundle{E}),[M]\right> \neq 0$. Here we use that the cup product is positive definite. In fact, if $\left< \mathrm{ch}(\bundle{E}),[M]\right >=0$ and $\left< \mathrm{ch}(\bundle{E}\otimes \bundle{E}),[M]\right> = 0$, then $\bundle{E}$ is stably rational trivial, i.e.~$\bundle{E}\notin  \mathscr{V}(M;\theta )$. Thus, the claim follows from $|R^{\bundle{E}\otimes \bundle{E}}|_{g,op}\leq 2|R^\bundle{E}|_{g,op}$.
\end{proof}
\begin{rem}
A counterexample to the last proposition without the assumption $b_-=0$ is easily found. For instance, consider the manifold $S^2\times T^2$, then $\ccg (S^2\times T^2)=0$ because the class $1\times [T^2]$ has infinite K--area. Conversely,
\[
0<\yam (S^2\times T^2)\leq 4\cdot \cgv (S^2\times T^2)
\]
shows the necessity of $b_-=0$ respectively $b_+=0$.
\end{rem}
\begin{lem}
\label{lem255}
Assume that $M$ satisfies $b_-=0$ and there is a class $c\in H^2(M;\Z )$ with $c^2=1$ (this is true for $0<b_+\leq 4$). Then $\cgv (M)\leq 2\sqrt{2}\, \pi $. In particular,
\[
\cgv (\proj{\C }^2)=\ccg (\proj{\C }^2)=\ccg (\proj{\C }^2_{[g_{FS}]})=2\sqrt{2}\, \pi .
\]
\end{lem}
\begin{proof}
The statement for $b_+\leq 4$ follows from Lemma 2.1 in \cite{MilHu}. Consider the class $c\in H^2(M;\Z )$ with $\int c^2=1$, and let $\bundle{E}$ be the complex line bundle with $c_1(\bundle{E})=c$. Choose a connection on $\bundle{E}$ which has $g$--harmonic curvature $2$--form $\mathbf{i}R^\bundle{E}=\omega $, then $[\omega ]=2\pi c_\R $ and $\omega \in \Omega ^{2,+}(M)$. Hence,
\[
\int |R^\bundle{E}|_{g,op}^2\cdot \mathrm{vol}_g=\int |\omega |_{g,1}^2\cdot \mathrm{vol}_g\leq 2\int |\omega |^2_g\cdot \mathrm{vol}_g=2\int \omega \wedge \omega =8\pi ^2
\]
together with $\left< \mathrm{ch}(\bundle{E}),[M]\right> =1$ completes the proof. The equality for the complex projective space follows from the above corollary.
\end{proof}
\begin{thm}
\label{thm256}
$\ccg (M)$ and $\cgv (M)$ are finite with
\[
\ccg (M)\leq \cgv (M)\leq 2\pi \cdot \sqrt{6\cdot \max \{ 2b_-,1\} }.
\]
\end{thm}
\begin{proof}
In order to show the estimate we use the existence of self dual connections on vector bundles proved by Taubes in \cite{Taub}. In fact, let $\bundle{E}\to M$ be a Hermitian bundle of rank $2$ with $c_1(\bundle{E})=0$ and $c_2(\bundle{E})=-\max \{ 2b_-,1\} $. Note that such a bundle can be obtained as pull back $f^*\spinor ^+$ where $f:M\to S^4$ is a map with $\deg f=\max \{ 2b_-,1\} $ and $\spinor ^+$ is the positive half spinor bundle on $S^4$. Hence, $\bundle{E}$ satisfies $\left< \mathrm{ch}(\bundle{E}),[M]\right> \neq 0$. Moreover, $\bundle{E}$ is a $\mathrm{SU}(2)$--bundle and to any metric on $M$, we can choose a self--dual $\mathrm{SU}(2)$--connection on $\bundle{E}$ (cf.~\cite[Theorem 1.1]{Taub}). Because $R^\bundle{E}_{X,Y}$ is trace free and $\mathrm{rk}(\bundle{E})=2$, we obtain
\[
\begin{split}
|R^\bundle{E}|_{g,op}^2&= \biggl[ \sum _{i<j}|R^\bundle{E} _{e_i,e_j}|_{op} \biggl] ^2\leq 6\sum _{i<j}|R^\bundle{E}_{e_i,e_j}|_{op}^2=\\
&=-3\sum _{i<j}\mathrm{tr}\left( R^\bundle{E}_{e_i,e_j}\circ R^\bundle{E}_{e_i,e_j}\right) =3 |R^\bundle{E}|^2_g=3|R^{\bundle{E},+}|_g^2.
\end{split}
\]
Thus, $R^{\bundle{E},-}=0$ and integration yield
\[
\begin{split}
\int _M|R^\bundle{E}|^2_{g,op}\cdot \mathrm{vol}_g&\leq 3\int _M\left[ |R^{\bundle{E},+}|^2_g-|R^{\bundle{E},-}|^2_g\right] \cdot \mathrm{vol}_g\\
&=-3\int \mathrm{tr}(R^\bundle{E}\wedge R^\bundle{E})=-24\pi ^2 c_2(\bundle{E})=24\pi ^2\max \{ 2b_-,1\} 
\end{split}
\]
which completes the proof.
\end{proof}
In general this is a very rough inequality as the example $M=\proj{\C }^2$ shows. In fact, if $b_->2$, Taubes existence theorem on self dual connections yields the better estimate
\[
\cgv (M)\leq 2\pi \cdot \sqrt{6\cdot \lceil 4b_-/3\rceil}
\]
where $\lceil a\rceil $ means the smallest integer greater than or equal to $a\in \R $. However, we believe that $\cgv (M)\leq \cgv (S^4) $ holds for any $4$--manifold which makes the improved inequality obsolete. Nevertheless, the estimate in the theorem becomes optimal for spin manifolds $M$ with $b_2=0$ and $\yam (M)=\yam (S^4)=8\sqrt{6}\, \pi $. Here we use the inequality $\yam (M)\leq 4\cgv (M)$.
\begin{cor}
If $[g_0]$ is the standard conformal class on $S^4$, then
\[
\cgv (S^4)=\ccg (S^4_{[g_0]})=2\sqrt{6}\, \pi .
\]
Moreover,
\[
\cgv (S^3\times S^1\# \cdots \# S^3\times S^1)=2\sqrt{6}\, \pi .
\] 
\end{cor}
\begin{proof}
The second statement follows from the above theorem, $\yam (M)\leq 4\cgv (M)$ for spin manifolds and
\[
\yam (S^3\times S^1\# \cdots \# S^3\times S^1)=\yam (S^4)=8\sqrt{6}\, \pi 
\]
(cf.~\cite{Koba1,Sch2}). Hence, it remains to show $\ccg (S^4_{[g_0]})=2\sqrt{6}\, \pi $ for the standard conformal class. Let $g_0$ be the standard metric on $S^4$, then the above theorem and proposition \ref{pro3}  yield with $c=0$:
\[
2\sqrt{6}\, \pi \geq \cgv (S^4)\geq \ccg (S^4_{[g_0]})\geq \frac{1}{4}\yam (S^4_{[g_0]})=\frac{8\sqrt{6}\, \pi }{4}=2\sqrt{6}\, \pi .
\]
Note that the 't Hooft construction (cf.~\cite{AHS}) yields an explicit construction of a self--dual $\mathrm{SU}(2)$ connection on $\spinor ^+\to S^2$ which can be used to reprove $\ccg (S^4_{[g]})\leq 2\sqrt{6}\, \pi $ for any conformal class $[g]$.
\end{proof}
Suppose that $M$ satisfies $b_-=0$ and $c\in H^2(M;\Z )$ is associated to a spin$^c$--structure on $M$. Then the harmonic $2$--form $\eta $ representing $c_\R $ has to be self dual, i.e.~
\[
\| c_\R \| _{[g]}^2\leq \int |\eta |_{g,1}^2\mathrm{vol}_g\leq 2\int |\eta |_{g}^2\mathrm{vol}_g=2\int \eta \wedge \eta =2c^2.
\]
Hence, inequality (\ref{ineq365}) yields
\[
\yam (M_{[g]})\leq 4\pi \sqrt{2c^2}+4\cgv (M)  
\]
for any conformal class $[g]$ on $M$. If $c^2>b_+$, we even have $\yam (M_{[g]})\leq 4\pi \sqrt{2c^2}$ which was proved by Grusky and LeBrun in \cite{GurLe2}. In the case $c^2=b_+$, we obtain a couple of lower bounds for $\cgv (M)$. For instance: If $N^4$ is a spin manifold with $b_2(N)=0$ and $\yam (N)\geq 12\sqrt{2}\, \pi $, then
\[
M_k:=\underbrace{\proj{\C }^2\# \cdots \# \proj{\C }^2}_{k-times}\# N
\]
satisfies $\yam (M_k)\geq 12\sqrt{2}\, \pi $ (cf.~\cite{Koba1,pre_AmDH,LeB5}). If $k>0$, then there is a class $a\in H^2(M_k;\Z )$ with $a^2=1$ and moreover, we can choose a spin$^c$ structure on $M_k$ whose associated class satisfies $c^2=b_+(M_k)=k$. Hence, we obtain from lemma \ref{lem255} and the above inequality
\[
3\sqrt{2}\, \pi -\sqrt{2k}\, \pi \leq \cgv (M_k)\leq \biggl\{ {2\sqrt{2}\, \pi \ ,\atop 2\sqrt{6}\, \pi \ ,} \ \ {k>0\atop k=0}
\]
for all $k$ which proves $\cgv (M_k)>0$ as long as $k<9$. Hence, proposition \ref{prop243} yields $\ccg (M_k)>0$ if $k<9$. 

\section{$L^{n/2}$--gaps of the Weyl tensor}
In 4--dimensions there are lots of optimal results estimating certain $L^2$--norms of the Weyl tensor (cf.~\cite{Gur1,LeB3,ABKS,Itoh} to name a few). However to our knowledge, in dimension $n>4$ there are no precise results on this subject. In order to conclude our estimates below, we use an observation made by Bourguignon in \cite{Bourg} that the curvature operator of the Bochner--Weitzenb\"ock formula on $\Omega ^m(M^{2m})$ depends only on the scalar curvature and on the Weyl tensor. Akutagawa et al.~showed in \cite{ABKS} that the $L^{n/2}$--norm of the Weyl tensor can be arbitrarily large even for conformal classes with Yamabe invariant sufficiently close to $\yam (M)$. Below, we will show the counterpart to their result for manifolds with Betti number $b_m(M^{2m})\neq 0$.

We denote by $W$ the Weyl tensor of a Riemannian manifold $(M^n,g)$ considered as endomorphism on $\Lambda ^2TM$ where the sign of $W$ and the Riemannian curvature operator $R$ are chosen in such a way that $R-W=\frac{\mathrm{scal_g}}{n(n-1)}$ for Einstein metrics. Let $\lambda _-(W) $ and $\lambda _+(W)$ be the real valued functions on $M$ which in $x\in M$ determine the minimal respectively maximal eigenvalue of $W_x\in \mathrm{End}(\Lambda ^2T_xM)$. The values
\[
\mathscr{W}_\pm (M_{[g]})=\mathscr{W}_\pm (M_g):=\left[ \int _M|\lambda _\pm (W)|^{n/2}\cdot \mathrm{vol}_g\right] ^{2/n}
\]
depend only on the conformal class $[g]$ and vanish iff $[g]$ is (locally) conformally flat. Since $W$ is trace free, a standard argument shows with $c(n)=\frac{(n+1)(n-2)}{2}$
\[
c(n)^{-1}\mathscr{W}_-(M_{[g]})\leq \mathscr{W}_+(M_{[g]})\leq c(n)\mathscr{W}_-(M_{[g]}).
\]
Note that one can also estimate the standard $L^{n/2}$--form of the Weyl tensor $\left[ \int |W|^{n/2}_g \mathrm{vol}_g\right] ^{2/n} $ from below and above by $\mathscr{W}_\pm (M_{[g]})$. However, in order to get optimal inequalities we state our main results only for the conformal invariants $\mathscr{W}_\pm (M_{[g]})$. We consider the following modified scalar curvatures:
\[
\begin{split}
&\mathrm{scal}_g^{W,-}=\mathrm{scal}_g+n(n-1)\lambda _-(W)\\
& \mathrm{scal}_g^{W,+}=\mathrm{scal}_g-(n-1)(n-2)\lambda _+(W)
\end{split}
\]
and the corresponding modified Yamabe invariants
\[
\yam _\pm ^W(M_{[h]}):=\inf _{g\in [h]}\frac{\int \mathrm{scal}_g^{W,\pm }\cdot \mathrm{vol}_g}{\mathrm{Vol}(M,g)^{(n-2)/n}}.
\]
\begin{thm}
Suppose that $M^{2m}$ is a closed oriented manifold with Betti number $b_m\neq 0$, then
\[
\yam ^W _\pm (M):=\sup _{[g]}\yam ^W _\pm (M_{[g]})\leq 0
\]
with $\yam _- ^W(M_{[g]})=0$ if $(M^{2m},g)$ is a locally symmetric Einstein space of nonnegative scalar curvature. Moreover, $\yam _+^W(M_{[g]})=0$ holds for the complex projective plane $\proj{\C}^2$ endowed with the Fubini--Study metric.
\end{thm}
This theorem follows immediately by the method of modified scalar curvature and lemma \ref{lem_weyl} below. Because $\yam ^W_\pm (M_{[g]})>0$ yields a metric with $\mathrm{scal}_g^{W,\pm }>0$, the estimates in lemma \ref{lem_weyl} imply that all harmonic $m$--forms must be trivial which is a contradiction to $H^m(M;\R )\neq \{ 0\} $. Thus, we conclude from the theorem and inequality (\ref{ineq12}) :
\begin{cor}
Suppose that $M^n$, $n=2m$, is a closed oriented manifold with Betti number $b_m>0$. Then for any conformal class $[g]$ on $M$:
\[
\yam (M_{[g]})\leq n(n-1)\mathscr{W}_-(M_{[g]})\quad \text{and}\quad \yam (M_{[g]})\leq (n-1)(n-2)\mathscr{W}_+(M_{[g]}).
\]
with equality for $\mathscr{W}_-$ if $(M,g)$ is a locally symmetric Einstein space of nonnegative scalar curvature, and equality for $\mathscr{W}_\pm $ if $(M,g)$ is the standard complex projective plane. 
\end{cor}
That this lower bound of the conformal Weyl invariant is not a local statement can easily be seen considering the standard sphere $S^{2m}$. More generally, if $(M^m,g)$ and $(N^n,h)$ are closed Riemannian manifolds of constant sectional curvature $K_g=1$ and $K_h=-1$, the product $(M\times N,g\oplus h)$ is conformally flat: $\mathscr{W}_\pm (M\times N_{[g\oplus h]})=0$, but the Yamabe invariant of $[g\oplus h]$ satisfies
\[
\yam (M\times N_{[g\oplus h]})\in \Biggl\{ \begin{array}{ll}(-\infty ,0)&m<n\\\{ 0\} &m=n\\
 (0,\infty )&m>n.
\end{array} 
\]
Moreover, $b_j(M)=0$ for all $0<j<m$ supplies that the $\frac{m+n}{2}$--Betti number of $M\times N$ can be positive only if $m\leq n$ which shows the significance on the assumption of the Betti number. 

In \cite{ABKS}, Akutagawa et al.~proved that $\mathscr{W}_{-}(M_{[g]})$ and $\mathscr{W}_+(M_{[g]})$ can be arbitrary large even for conformal classes $[g]$ with $\yam (M_{[g]})\geq \yam (M)-\epsilon $. The above corollary is therefore the counterpart to their result: $\mathscr{W}_\pm (M_{[g_i]})$ is bounded from below by a positive constant if $\{ [g_i]\} $ is a \emph{Yamabe sequence} (cf.~\cite{ABKS}). In order to be precise, Akutagawa et al.~define a Yamabe sequence $\{ [g_i]\} $ to be a sequence of conformal structures $[g_i]$ with
\[
\lim _{i\to \infty }\yam (M_{[g_i]})=\yam (M),
\]
and consider invariants closely related to
\[
\mathscr{W}_\pm (M):=\inf _{\{ [g_i]\}}\liminf _{i\to \infty }\mathscr{W}_\pm (M_{[g_i]})
\]
where the infimum is taken over all Yamabe sequences. Akutagawa et al.~prove that for any $\epsilon >0$ there is a conformal class $[g]$ with $\yam (M_{[g]})\geq \yam (M)-\epsilon $ and $\mathscr{W}_\pm (M_{[g]})\geq \frac{1}{\epsilon } $. Conversely, the above results and
\[
\yam (\proj{\C }^2)=\yam (\proj {\C }^2_{[g_{FS}]})=12\sqrt{2}\, \pi 
\]
(cf.~\cite{LeB5,GurLe2}) yield the following counterpart:
\begin{thm}
Suppose that $M^{2m}$ is closed with Betti number $b_m>0$, then
\[
\yam (M)\leq 2m(2m-1)\mathscr{W}_-(M)\quad \text{and}\quad \yam (M)\leq 2(m-1)(2m-1)\mathscr{W}_+(M)
\]
with equality in both cases if  $M=\proj{\C }^2$ or $M=T^{2m}$.
\end{thm}
The main results in this section follow from the estimates in the lemma below. The key observation for the lemma is due to Bourguignon \cite{Bourg}, although he did not compute the explicit curvature expression. Moreover, we should point out that weaker versions of the result below have been stated by Lafontaine in \cite[theorem 6]{Laf1} respectively by Itoh in \cite[section 3]{Itoh}. However, Lafontaine and Itoh used an incorrect coefficient for the scalar curvature term, because on the standard sphere of constant sectional curvature $K$, the Bochner--Weitzenb\"ock formula on $l$--forms is given by
\[
\Delta =\nabla ^*\nabla +K\cdot l(n-l).
\]
\begin{lem}
\label{lem_weyl}
Let $(M^n,g)$, $n=2m$, be a Riemannian manifold, then the Bochner--Weitzen\-b\"ock formula on $m$--forms is given by
\[
\Delta =\nabla ^*\nabla +\frac{n}{4(n-1)}\mathrm{scal}_g+\frak{R}^W
\]
where $\frak{R}^W$ depends only on the Weyl tensor and is bounded from below by
\[
\frak{R}^W\geq  m^2\lambda _- (W)\quad \text{and}\quad \frak{R}^W\geq -m(m-1)\lambda _+ (W) .
\]
Moreover, if $(M,g)$ is a closed locally symmetric Einstein space with $\mathrm{scal}_g\geq 0$ and $H^m(M;\R)\neq \{ 0\}$, then $\mathrm{scal}_g=-n(n-1)\lambda _- (W)$, i.e.~$m^2\lambda _-(W)$ is the minimal eigenvalue of $\frak{R}^W$. The equality $(n-1)(n-2)\lambda _+(W)=\mathrm{scal}_g$ is satisfied for the complex projective space $\proj{\C }^2$ endowed with the Fubini--Study metric, in this case $-2\lambda _+(W)$ is the minimal eigenvalue of $\frak{R}^W$.
\end{lem}
\begin{proof}
It is sufficient to consider the statement for complex valued forms. The bundle $\Lambda ^*_\C T^*M$ is (at least locally) given by $\spinor M\otimes \spinor M$ where $\spinor M$ means the (locally existent) irreducible complex spinor bundle. Since the Dirac structure on $\Lambda ^*_\C T^*M$ coincides with the twisted Dirac bundle structure $\spinor M\otimes \spinor M$, we obtain for an orthonormal basis $e_1,\ldots ,e_{n}$ of $TM$:
\[
\Delta =D^2=\nabla ^*\nabla +\frac{\mathrm{scal}_g}{4}+\frac{1}{2}\sum _{i,j}\gamma (e_i)\gamma (e_j)\otimes R^s _{e_i,e_j}.
\]
If $R:\Lambda ^2TM\to \Lambda ^2TM$ denotes the Riemannian curvature operator, the curvature on $\spinor M$ satisfies $R^s_{X,Y}=-\frac{1}{2}\gamma (R(X\wedge Y))$. Moreover, if $\mathrm{Ric}:TM\to TM$ is the Ricci endomorphism and $\mathrm{Ric}_2:\Lambda ^2TM\to \Lambda ^2TM$ the derivative extension to $\Lambda ^2TM$, i.e.~
\[
\mathrm{Ric}_2(X\wedge Y)=\mathrm{Ric}(X)\wedge Y+X\wedge \mathrm{Ric}(Y),
\]
then the curvature operator decomposes as follows
\[
R=-\frac{\mathrm{scal}_g}{(n-1)(n-2)}\cdot \mathrm{Id}+\frac{1}{n-2}\mathrm{Ric}_2+W.
\]
Consider the linear operator $L(\eta )=-\sum _i\gamma (e_i)\eta \gamma (e_i)$ on $\Lambda ^*_\C T^*M$ (cf.~\cite[\S 5  Chap.II]{LaMi}), then $L$ acts trivial on $\Lambda ^m_\C T^*M$ and moreover, $(\varphi \otimes \psi )\gamma (\sigma ^t) =\varphi \otimes \gamma (\sigma )\psi $ for  $\sigma \in \Cl (TM)$ and $\varphi \otimes \psi \in \spinor M\otimes \spinor M$. Choose an eigenbasis $e_1,\ldots ,e_{n}$ of $\mathrm{Ric}$ to eigenvalues $\lambda _1,\ldots ,\lambda _{n}$, then $\gamma (e_i\wedge e_j)=\gamma (e_i)\gamma (e_j)+\delta _{ij}$ yields
\[
\begin{split}
2\sum _{i,j}\gamma (e_i)\gamma (e_j)\otimes R^s _{e_i,e_j}&=-\frac{n\cdot \mathrm{scal}_g}{(n-1)(n-2)}+\frac{2}{n-2}\sum _{i}\lambda _i\\
&\qquad \qquad -\sum _{i,j}\gamma (e_i)\gamma (e_j)\otimes \gamma (W(e_i\wedge e_j))\\
&=\frac{\mathrm{scal}_g}{n-1}+4\cdot \frak{R}^W.
\end{split}
\]
In this case we use the fact that
\[
\sum _{i,j}(\lambda _i+\lambda _j)\gamma (e_i)\gamma (e_j)\otimes \gamma (e_i)\gamma (e_j)\quad \text{and}\quad  \sum _{i,j}\gamma (e_i)\gamma (e_j)\otimes \gamma (e_i)\gamma (e_j)
\]
act trivial on $\Lambda ^m_\C T^*M$ because $L=0$ on $\Lambda ^m_\C T^*M$. This proves the first part of the proposition where
\[
\frak{R}^W=-\frac{1}{4}\sum _{i,j=1}^{n}\gamma (e_i)\gamma (e_j)\otimes \gamma (W(e_i\wedge e_j)).
\]
In order to compute the lower bound for $\frak{R}^W$ we consider modifications of $W$ by a function $\alpha :M\to \R $. In particular, suppose that $\overline{W}=W+\alpha \cdot \mathrm{Id}$ satisfies $\overline{W}\geq 0$ or $\overline{W}\leq 0$. We conclude by the above arguments
\[
\frak{R}^{\overline{W}}=-\frac{1}{4}\sum _{i,j=1}^{n}\gamma (e_i)\gamma (e_j)\otimes \gamma (\overline{W}(e_i\wedge e_j))=\frak{R}^W+\frac{n}{4}\alpha 
\]
on $\Lambda ^m_\C T^*M$, hence, it suffices to compute a lower bound of $\frak{R}^{\overline{W}}$. Suppose at first that $\overline{W}\geq 0$ and choose an orthonormal eigenbasis $\eta _k\in \Lambda ^2TM$ of $\overline{W}$ to eigenvalues $\tau _k\geq 0$, then
\[
\frak{R}^{\overline{W}}=-\frac{1}{2}\sum _k\tau _k\cdot \gamma (\eta _k)\otimes \gamma (\eta _k).
\]
Since $\tau _k\geq 0$ for all $k$ and the square of skew--Hermitian endomorphisms is nonpositive, the operator
\[
\begin{split}
\frak{C}&=\sum _k\tau _k\cdot (\gamma (\eta _k)\otimes \id +\id \otimes \gamma (\eta _k))^2\\
&=-4\cdot \frak{R}^{\overline{W}}+\sum _k\tau _k(\gamma (\eta _k)^2\otimes \id +\id \otimes \gamma (\eta _k)^2),
\end{split}
\]
is nonpositive. The first Bianchi identity for $W$ and $\overline{W}$ yields $\sum _k \tau _k\eta _k\wedge \eta _k=0$. Hence, $\mathrm{tr}(\overline{W})=\sum \tau _k=\frac{n(n-1)}{2}\alpha $ and $\gamma (\eta _k)^2=-|\eta _k|^2+\gamma (\eta _k\wedge \eta _k)$ supply
\[
\frak{C}=-4\frak{R}^{\overline{W}}-n(n-1)\alpha . 
\]
Thus, $\frak{C}\leq 0$ proves
\[
\frak{R}^W=\frak{R}^{\overline{W}}-\frac{n}{4}\alpha \geq -\frac{n^2}{4}\alpha=-m^2\alpha 
\]
for any function $\alpha $ such that $W+\alpha \cdot \mathrm{id}\geq 0$. Conversely, suppose that $\overline{W}\leq 0$ and choose as before an orthonormal eigenbasis $\eta _k$ of $\overline{W}$ to eigenvalues $\tau _k\leq 0$. Consider now the operator
\[
\frak{C}:=\sum \tau _k(\gamma (\eta _k)\otimes \id -\id \otimes \gamma (\eta _k))^2=4\frak{R}^{\overline{W}}+n(n-1)\alpha ,
\] 
then $\tau _k\leq 0$ yields $\frak{C}\geq 0$, i.e.
\[
\frak{R}^W=\frak{R}^{\overline{W}}-\frac{n}{4}\alpha \geq \frac{n(n-2)}{4}\alpha =m(m-1)\alpha . 
\]
Hence, we conclude the lower bounds of $\frak{R}^W$ by choosing $\alpha =-\lambda _\pm (W)$. Suppose now that $(M,g)$ is closed and locally symmetric with $R=\frac{\mathrm{scal}_g}{n(n-1)}+W\geq 0$, then
\[
\frak{R}^W\geq m^2\lambda _-(W)\geq -\frac{n\cdot \mathrm{scal}_g}{4(n-1)} .
\]
But if $H^m(M;\R )\neq \{ 0\} $, the Bochner--Weitzenb\"ock formula yields a nontrivial parallel $m$--form $\theta  $ with
\[
\frac{n\cdot \mathrm{scal}_g}{4(n-1)}\theta +\frak{R}^W\theta =0
\]
proving the statement for locally symmetric Einstein spaces. A standard exercise shows that the Weyl tensor of the Fubini--Study metric $g$ on $\proj{\C }^m$ satisfies
\[
2m(2m-1)\lambda _-(W)=-\mathrm{scal}_g\quad \text{and}\quad m(2m-1)\lambda _+(W)=(m-1)\mathrm{scal}_g
\]
which means that $6\lambda _+(W)=\mathrm{scal}_g$ on $\proj{\C }^2$.
\end{proof}
\chapter{Upper bounds of scalar curvature}
\section{Scalar curvature bounds by Riemannian functionals}
\label{sec_functionals}
Let $M$ be a closed smooth manifold and $\met (M)$ be the set of smooth Riemannian metrics on $M$. A \emph{Riemannian functional} on $M$ is a $\mathrm{Diff}(M)$--invariant map $\mu :\met (M)\to \R $, i.e.~$\mu (f^*g)=\mu (g)$ for any diffeomorphism $f:M\to M$ and any metric $g$ (cf.~\cite{Bes}). We say that $\mu $ is \emph{scaling invariant of weight} $k\in \Z $ if $\mu (c\cdot g)=c^k\cdot \mu (g)$ for positive constants $c$ and all $g\in \met (M)$. The scalar curvature map $\mathrm{scal}:\met (M)\to C^\infty (M)$ defines a scaling invariant Riemannian functional of weight $-1$ as follows
\[
\scal :\met (M)\to \R \ ,\ g\mapsto \min _{x\in M}\ \mathrm{scal}_g(x).
\]
The order on $\R $ yields a partial order on the set of Riemannian functionals where $\mu _1\geq \mu _2$ means that $\mu _1(g)\geq \mu _2(g)$ for all $g\in \met (M)$. If $\mu _1\geq \mu _2$, we say that $\mu _1$ is an \emph{upper bound} of $\mu _2$. This yields the following definition:
\begin{defn}
A scaling invariant Riemannian functional $\mu $ of weight $-1$ is called \emph{upper bound of the scalar curvature} iff $\mu \geq \scal $.  Moreover, we say that an upper bound  $\mu \geq \scal $ is of
\begin{itemize}
\item \emph{typeI}, if for all $\epsilon >0 $ there is a metric $g $ with $\mathrm{Vol}(M,g)=1$ and  $\scal (g)>\mu (g)-\epsilon $. 
\item \emph{typeII},  if there is a metric $g$ with $\scal (g)=\mu (g)$.
\item \emph{typeIII}, if up to scaling by constants and up to isometry, there is a unique metric $g$ with $\scal (g)=\mu (g) $. 
\end{itemize}
\end{defn}
Obviously, the different notions of upper bounds satisfy
\[
\{ \text{typeIII}\} \subseteq \{ \text{typeII}\} \subseteq \{ \text{typeI}\}
\]
Considering the Riemannian functional $\mu =0$ rephrases the question of positive scalar curvature which has been extensively studied. In fact, $\mu =0$ is an upper bound of the scalar curvature iff $M$ does not have a metric of positive scalar curvature and it is a typeII upper bound if additionally $M$ admits a Ricci flat metric. The upper bound $\mu =0$ can be further improved in case $\yam (M)<0$. In fact, suppose that $M$ is closed and orientable of dimension $n$, then the Yamabe invariant of a conformal class $[g]$ on $M$ is given by
\[
\yam (M_{[g]}):=\inf _{h\in [g]}\frac{\int \mathrm{scal}_h\cdot \mathrm{vol}_h}{\mathrm{Vol}(M,h)^{(n-2)/n}}
\]
and the Yamabe invariant (respectively the sigma invariant) of $M$ is $\yam (M)=\sup _{[g]}\yam (M_{[g]})$. Define
\[
\mu (g):=\yam (M)\cdot \mathrm{Vol}(M,g)^{-2/n},
\]
then $\mu $ is a scaling invariant Riemannian functional of weight $-1$ and it is a typeI upper bound of the scalar curvature on $M$ if $\yam (M)\leq 0$. This follows immediately from remark \ref{rem_yamabe} for $F\equiv 0$ and the definition of $\yam (M)$. Note that for each $\epsilon >0$ there is a conformal class $[g]$ with $\yam (M_{[g]})>\yam (M)-\epsilon $, i.e.~the metric $h\in [g]$ with constant scalar curvature and $\mathrm{Vol}(M,h)=1$ satisfies $\mathrm{scal}_h=\scal (h)>\mu (h)-\epsilon $. Moreover, we deduce from results in \cite[Chp.~4]{Bes} that $\mu $ is a typeII upper bound if $\yam (M)\leq 0$ and there is an Einstein conformal class $[g]$ with $\yam (M_{[g]})=\yam (M)$, in this case the Einstein metric $h\in [g]$ satisfies $\mathrm{scal}_h=\mu (h)$. The upper bound becomes typeIII iff there is a unique conformal class $[g]$ with $\yam (M_{[g]})=\yam (M)$.

This approach fails in general for manifolds with $\yam (M)>0$ as the following example shows. Consider a product manifold like $M=S^2\times S^1$, then $\yam (M)=\yam (S^3)=6(2\pi ^2)^{2/3}$ (cf.~Schoen \cite{Sch2} and Kobayashi \cite{Koba1}). The product metric $h:=g_0\oplus \mathrm{d}t^2$ has scalar curvature $\mathrm{scal}_h=2$ for any line segment $\mathrm{d}t^2$ ($g_0$ is the standard metric on $S^2$), i.e.~increasing the size of $S^1$ keeps the scalar curvature constant while the volume of $(M,h)$ goes to infinity and so $\mu (h)\to 0$ proves that $\mu $ can not be an upper bound of the scalar curvature. This leads us to consider the question of upper bounds for the scalar curvature on closed manifolds $M$ with $\yam (M)>0$. We have observed that in general
\[
\mu (g)=C\cdot \mathrm{Vol}(M,g)^{-2/n}
\] 
will not be an upper bound of the scalar curvature for arbitrary constants $C$ if $M^n$ has positive Yamabe invariant. In fact, the example $S^2\times S^1$ suggests to control the areas of $g$ rather than the volume. 

An upper bound of the scalar curvature  on spin manifolds $M^n$ is determined by Gromov's K--area (cf.~\cite{Gr01} and chapter \ref{chp2}), here we choose the Riemannian functional
\[
\mu _\ke (g)=\frac{n(n+1)^3}{2\cdot \ke (M_g)}
\] 
where $\ke (M_g)$ denotes the total K--area of the Riemannian manifold $M_g=(M,g)$. There is no known example for which the universal constant $n(n+1)^3/2$ determines a typeI upper bound $\mu _\ke $. However, on certain manifolds one can improve the constant to obtain better statements. For instance, on a rational homology sphere $M^n$
\[
\mu (g):=\frac{2n(n-1)}{\ke (M_g)}
\]
supplies an upper bound of the scalar curvature. Moreover, $\mu $ is a type II upper bound on quotients of $S^n$ where $\mathrm{scal}_g=\mu (g)$ is achieved for metrics of constant sectional curvature. Here, we simply use the fact that on a rational homology sphere, the (ordinary) K--area and the K--area for the Chern character coincide (cf.~chapter \ref{chp2}), and then we apply proposition \ref{prop_scalar}, remark \ref{rem243} and corollary \ref{cor345}. 

We can further improve the above functional to get better upper bounds of the scalar curvature. If $\bundle{S} $ is a Dirac bundle, $\frak{R}^{\bundle{S}}\in \Gamma (\mathrm{End}(\bundle{S}))$ means the twist curvature endomorphism in the Bochner--Weitzenb\"ock formula for the Dirac operator of $\bundle{S}$:
\[
\D ^2=\nabla ^*\nabla +\frac{\mathrm{scal}_g}{4}+\frak{R}^\bundle{S}.
\]
Suppose that $M$ is even dimensional and define
\[
\mu _\ce (g):=4\cdot \inf _\bundle{S}\left( -\min _{x\in M} \frak{R}^{\bundle{S}} _\mathrm{min}(x)\right)  
\]
where the infimum is taken over the nonempty set of $\Z _2$--graded Dirac bundles $\bundle{S}$ with $\mathrm{ind}(\D )\neq 0$ and $\frak{R}^\bundle{S} _\mathrm{min}:M\to (-\infty ,0] $ means the minimal eigenvalue function of $\frak{R}^\bundle{S}\in \Gamma (\mathrm{End}(\bundle{S}))$. Note that this is an $L^\infty $--norm on trace free endomorphisms of $\bundle{S}$ and $\mathrm{tr}(\frak{R}^\bundle{S})=0$. Then $\mu _\ce $ is a Riemannian functional which is scaling invariant of weight $-1$. Moreover, if $M$ is closed and oriented, the Atiyah--Singer index theorem shows that $\mu _\ce $ is an upper bound of the scalar curvature on $M$. In order to get a similar functional for odd dimensional $M$, consider the above functional for the product manifold $M\times S^1$, i.e.~we define
\[
\mu _\ce (g):=\inf _{\mathrm{d}t^2}\mu _\ce (g\oplus \mathrm{d}t^2).
\]
Then $\mu _\ce $ is an upper bound of the scalar curvature on odd dimensional oriented closed manifolds. If $M$ is a spin manifold, a simple exercise shows $\mu _\ke \geq \mu _\ce $. We summarize these observations in the theorem below. The second statement in the theorem is due to the fact that the bundle of exterior forms is a $\Z _2$--graded bundle with twist curvature $ \frak{R}^{\Lambda ^*M } _\mathrm{min}= -\mathrm{scal}_g/4$ if $g$ is locally symmetric of nonnegative Ricci curvature (cf.~section \ref{notes_bochner}). 
\begin{thm}
Let $M$ be a closed oriented manifold, then $\mu _\ce \geq \scal $. Moreover, if $(M,g)$ is a locally symmetric space of positive Ricci curvature and Euler characteristic $\chi (M)\neq 0$, then $\scal (g)=\mu _\ce (g)$ (i.e.~$\mu _\ce $ is a typeII upper bound of the scalar curvature). 
\end{thm}
Another way to obtain upper bounds of the scalar curvature is to consider area extremal metrics. Suppose that $g_0\in \met (M)$ is a metric of non--negative scalar curvature, then $g_0$ is called \emph{area extremal} (cf.~\cite{Gr01,GoSe1}) if for any $g\in \met (M)$ the inequalities
\begin{equation}
\label{area1}
\begin{split}
|v\wedge w|_g&\geq |v\wedge w|_{g_0}  \quad \forall v,w\in TM\\
\mathrm{scal}_g&\geq \mathrm{scal}_{g_0}
\end{split}
\end{equation}
imply $\mathrm{scal}_g=\mathrm{scal}_{g_0}$. Moreover, we say that $g_0$ is \emph{strict area extremal} if (\ref{area1}) implies additionally $g=g_0$. For instance, if $M$ is a manifold which does not admit a metric of positive scalar curvature, any Ricci flat metric $g_0$ on $M$ is area extremal. Examples of strict area extremal metrics are locally symmetric spaces of nontrivial Euler characteristic and positive Ricci curvature.
\begin{thm}[\cite{GoSe1}]Suppose that $(M^n,g_0)$, $n\geq 3$, is a locally symmetric space of positive Ricci curvature and with Euler characteristic $\chi (M)\neq 0$ or Kervaire semi characteristic $\sigma (M)\neq 0$, then $(M,g_0)$ is strict area extremal.
\end{thm}
The standard sphere was the first area extremal example of positive scalar curvature and discovered by Llarull in \cite{Llarull} although his proof is incomplete for odd dimensional spheres. This gap was closed by Kramer in \cite{Kra}. Min--Oo considered in \cite{MinO2} a similar question for Hermitian symmetric space. Goette and Semmelmann generalized the previous results in \cite{GoSe1,GoSe2} to manifolds with nonnegative curvature operator and nontrivial Euler characteristic as well as to K\"ahler manifolds with nonnegative Ricci curvature. Note that symmetric spaces $G/H$ of compact type have nonnegative curvature operator and that $\chi (G/H)\neq 0$ if and only if $\mathrm{rk}(G)=\mathrm{rk}(H)$. We should also mention that the proof in \cite{GoSe1} is incomplete if $\mathrm{rk}(G)-\mathrm{rk}(H)=1$, but the statement in \cite{GoSe1} is correct if one assumes additionally $\sigma (G/H)\neq 0$. Goette considered in \cite{Go1} area extremal examples of homogeneous spaces.
  
Let $M$ be a closed manifold and $g_0\in \met (M)$ be an area extremal metric, then $f^*g_0$ is area extremal for all diffeomorphism $f:M\to M$ and hence, $g_0$ defines a Riemannian functional $\mu _{g_0}:\met (M)\to [0,\infty )$ as follows:
\[
\mu _{g_0}(g):=\left( \max _{x\in M}\mathrm{scal}_{g_0}(x) \right) \cdot \inf _{f}\max _{v,w\in TM} \frac{|\mathrm{d}f(v)\wedge \mathrm{d}f(w)|_{g_0}}{|v\wedge w|_g}
\]
where the infimum is taken over all diffeomorphisms $f:M\to M$ and $v\wedge w\neq 0$. Obviously, $\mu $ is a Riemannian functional with $\mu _{g_0}(cg)=c^{-1}\mu _{g_0}(g)$. 
\begin{prop}
If $g_0$ is area extremal, then $\mu _{g_0}\geq \scal $ with $\scal (g_0)=\mu _{g_0}(g_0)$ if $g_0$ has constant scalar curvature. Hence, an area extremal metric of constant scalar curvature determines a typeII upper bound of the scalar curvature. Moreover, if $(M,g_0)$ is a closed oriented locally symmetric space of positive Ricci curvature and Euler characteristic $\chi (M)\neq 0$, then
\[
\mu _{g_0}\geq  \mu _\ce  \geq \scal 
\]
with $\mu _{g_0}(g_0)=\mu _\ce (g_0)=\scal (g_0)=\mathrm{scal}_{g_0}$.
\end{prop}
\begin{proof}
The first part follows from the definitions. Hence, it remains to show that $\mu _{g_0}\geq \mu _\ce $ if $(M,g_0)$ is locally symmetric with $\chi (M)\neq 0$. Let $g$ be a metric on $M$ and consider the formal decomposition of the Dirac bundle $\Lambda ^*_\C M=\spinor M\otimes \spinor M$ where the (twisting) second factor is endowed with the Levi--Civita connection of $f^*g_0$. Then the twisted curvature endomorphism is bounded by (cf.~\cite{GoSe1} or section \ref{notes_bochner}):
\[
 \frak{R}^{\Lambda ^*M }\geq-\frac{ \mathrm{scal}_{g_0}}{4}\cdot \max _{v,w\in TM}\frac{|v\wedge w|_{f^*g_0}}{|v\wedge w|_g}
\]
which proves the assertion because the Dirac operator associated to $\Lambda ^*M$ has nontrivial index and this inequality is true for any diffeomorphism $f$.
\end{proof}
\begin{conj}
If $g_0$ is strict area extremal of constant scalar curvature, then $\mu _{g_0}$ is a typeIII upper bound of the scalar curvature.
\end{conj}
The Riemannian functionals obtained from the K--area and area extremal metrics $g_0$ are inverse area functions, i.e.
\begin{equation}
\label{eq_area}
|v\wedge w|_g\geq |v\wedge w|_h\quad \forall v,w\in TM\quad \Longrightarrow \quad \mu _\bullet (g)\leq \mu _\bullet (h)
\end{equation}
for all $g,h$ and $\mu _\bullet \in \{ \mu _\ke ,\mu _{g_0}\} $. If $\mu $ is a Riemannian functional of weight $-1$ which satisfies (\ref{eq_area}) for all $g$, $h$, then $\mu $ is nonnegative, i.e.~$\mu (g)\geq 0$ for all $g$. This leads us to consider the set
\[
\mathpzc{SC_{area}}(M)=\{ \mu \ |\ \mu (c\cdot g)=c^{-1}\mu (g) , \mu \geq \scal , \mu \text{ satisfies (\ref{eq_area})}\} 
\]
consisting of Riemannian functionals which are upper bounds for the scalar curvature on $M$ and satisfy (\ref{eq_area}) for arbitrary $g,h\in \met (M)$. The example $\mu _\ke $ shows that $\mathpzc{SC_{area}}(M)$ is nonempty if $M$ is a closed spin manifold. The Riemannian functional $\mu _\ce $ does not seem to satisfy (\ref{eq_area}) which means that in the non--spin case, $\mathpzc{SC_{area}}(M)\neq \emptyset $ turns out to be a nontrivial question. Nevertheless, one should be able to show $\mathpzc{SC_{area}}(M)\neq \emptyset $ for a variety of spin$^c$--manifolds. 
Note that $\mathpzc{SC_{area}}(M)$ is a partially ordered convex subset of the vector space consisting of Riemannian functionals with weight $-1$. Moreover, $\scal \notin \mathpzc{SC_{area}}(M)$ because (\ref{eq_area}) is not satisfied but the above results supply a map
\[
\{ \text{area extremal metrics on $M$}\} \to \mathpzc{SC_{area}}(M),\ g_0\mapsto \mu _{g_0}.
\] 
However,  it is still an open problem which manifolds of positive Yamabe invariant admit area extremal metrics. In order to deal with this question, the following proposition may be useful. 
\begin{prop}
Suppose there are $\mu \in \mathpzc{SC_{area}}(M)$ and $h\in \met (M)$ with $\mathrm{scal}_h=\mu (h)$, then $h$ is area extremal of nonnegative constant scalar curvature.
\end{prop}
In general one should not expect the existence of a typeII upper bound $\mu \in \mathpzc{SC_{area}}(M)$. For instance, the manifold $M=T^n\# T^n$ satisfies $\yam (M)=0$ if $n\geq 3$, but $M$ does not have a metric with vanishing scalar curvature and hence, $\mu (g)>\scal (g)$ for all $g$ and $\mu \in \mathpzc{SC_{area}}(M)$. Here, we use that  $M$ is an enlargeable spin manifold and on these manifolds, each metric of nonnegative scalar curvature must be flat.   
\section{Conform area extremal metrics} 
\label{sec_conf_extremal}
We can generalize the notion of area extremality given in (\ref{area1}) as follows. A Riemannian metric $g_0$ with $\mathrm{scal}_{g_0}\geq 0$ is said to be \emph{conform area extremal} if for any $g\in \met (M)$ and continuous function $\rho :M\to (0,\infty )$, the inequalities
\begin{equation*}
\begin{split}
|v\wedge w|_g&\geq \rho \cdot |v\wedge w|_{g_0}  \quad \forall v,w\in TM\\
\mathrm{scal}_g&\geq \frac{\mathrm{scal}_{g_0}}{\rho }
\end{split}
\end{equation*}
imply $\rho \cdot \mathrm{scal}_g=\mathrm{scal}_{g_0}$. Moreover, $g_0$ is \emph{strict conform area extremal} if additionally $\rho =\mathrm{const}$ and $g=\rho \cdot g_0$. Thus, each (strict) conform area extremal metric is (strict) area extremal but not vice versa. In order to simplify this definition we introduce the area dilatation function of $g_0$ by $g$:
\[
\mathrm{area}\bigl( {\textstyle \frac{g_0}{g}}\bigl)  :M\to (0,\infty ),\ x\mapsto \max _{v,w\in T_xM}\frac{|v\wedge w|_{g_0}}{|v\wedge w|_g}.
\]
Thus, $g_0$ is conform area extremal if $\mathrm{scal}_{g_0}\geq 0$ and for all $g\in \met (M)$ with
\begin{equation}
\label{conf_extremal}
\mathrm{scal}_g\geq \mathrm{scal}_{g_0}\cdot \mathrm{area}\bigl( {\textstyle \frac{g_0}{g}}\bigl) ,
\end{equation}
equality must hold.  Moreover, $g_0$ is strict conform area extremal if (\ref{conf_extremal}) for $g\in \met (M)$ implies $g=c\cdot g_0$ for some constant $c>0$. If we replace in (\ref{conf_extremal}) the function $\mathrm{area}\bigl( {\textstyle \frac{g_0}{g}}\bigl) $ by its maximum, we recover the notion of area extremality.

Any metric $g_0$ with $\mathrm{scal}_{g_0}\geq 0$ determines a modified scalar curvature (cf.~section \ref{sec2}) by
\[
\overline{\mathrm{scal}}^{g_0}:\met (M)\to C^{0,\alpha } (M) ,\ g\mapsto \mathrm{scal}_g-\mathrm{scal}_{g_0}\cdot \mathrm{area}\bigl( {\textstyle \frac{g_0}{g}}\bigl)  \ .
\]
Hence, if
\[
\overline{\yam }^{g_0}(M):=\sup _{[g]}\overline{\yam }^{g_0}(M_{[g]}),\qquad \overline{\yam }^{g_0}(M_{[g]}):=\inf _{h\in [g]}\frac{\int \overline{\mathrm{scal}}^{g_0}_h\cdot \mathrm{vol}_h}{\mathrm{Vol}(M,h)^{(n-2)/n}}
\]
denotes the modified Yamabe invariant for $\overline{\mathrm{scal}}^{g_0}$ as defined in section \ref{sec2} we can rephrase conform area extremality as follows:
\begin{prop}
Suppose that $g_0$ has nonnegative scalar curvature, then $g_0$ is conform area extremal iff $\overline{\yam }^{g_0}(M)= 0$. Moreover, $g_0$ is strict conform area extremal iff $\overline{\yam }^{g_0}(M_{[g]})\leq 0$ for all $[g]$ with strict inequality for all $[g]\neq [g_0]$.
\end{prop} 
\begin{proof}
Let $g_0$ be conform area extremal, then $\overline{\mathrm{scal}}^{g_0}(g)\geq 0$ for some $g$ implies equality by definition. Hence, corollary \ref{cor_yam} yields $\overline{\yam }^{g_0}(M_{[g]})\leq 0$ for all $[g]$. Conversely, if $\overline{\yam }^{g_0}(M_{[g]})\leq 0$ for all $[g]$, then remark \ref{rem_yamabe} shows that for each metric $g$ on $M$ with $\mathrm{Vol}(M,g)=1$ and
\[
\overline{\mathrm{scal}}^{g_0}(g)\geq \overline{\yam }^{g_0}(M_{[g]})\in (-\infty ,0]
\]
equality has to hold. Hence, $\overline{\mathrm{scal}}^{g_0}(g)\geq 0$ implies equality, i.e.~$g_0$ is conform area extremal. Moreover, the same argument shows that for $ \overline{\yam }^{g_0}(M_{[g]})<0$, there is a point $x$ with $\overline{\mathrm{scal}}^{g_0}(g)(x)<0$ which means that $\overline{\mathrm{scal}}^{g_0}(g)=0$ is only possible in case $\overline{\yam }^{g_0}(M_{[g]})=0$. Thus, suppose that $\overline{\yam }^{g_0}(M_{[g]})<0$ for all $[g]\neq [g_0]$, then $\overline{\mathrm{scal}}^{g_0}(g)= 0$ yields $[g]=[g_0]$ and remark \ref{rem_yamabe} supplies $g=c\cdot g_0$ for some constant $c$. Conversely, if $g_0$ is strict conform area extremal and $[g]\neq [g_0]$, then for any $h\in [g]$, there is a point $x\in M$ with $\overline{\mathrm{scal}}^{g_0}_h(x)<0$. Now choose the metric $h\in [g]$ with $\mathrm{Vol}(M,h)=1$ and constant modified scalar curvature (cf.~proposition \ref{prop_itoh}) proving $\overline{\yam }^{g_0}(M_{[g]})=\overline{\mathrm{scal}}^{g_0}_h=\mathrm{const}<0$.
\end{proof}
Hence, there is no conform area extremal metric on $M$ if $\overline{\yam }^{g_0}(M)>0$ for all $g_0$. In fact, the quantity $\inf _{g_0}\overline{\yam }^{g_0}(M)$ may be very useful to exclude conform area extremal metrics. Note that for any metric $g_0$ with $\mathrm{scal}_{g_0}\geq 0$ we know $\overline{\yam }^{g_0}(M_{[g_0]})=0$ from remark \ref{rem_yamabe} which means
\[
\overline{\yam }^{g_0}(M)\in [0,\yam (M)].
\]
The main purpose of this section is to prove the following generalization of the above results on area extremal metrics:
\begin{thm}
Let $(M^m_0,g_0)$, $m>2$, be an orientable locally symmetric space of positive Ricci curvature and Euler characteristic $\chi (M_0)\neq 0$, then $(M_0,g_0)$ is strict conform area extremal. Moreover, if $(T^k,h)$ is a flat torus and $(F,b)$ is a Ricci flat spin manifold with nonvanishing $\widehat{A}$--genus, then
\[
(M_0\times T^k\times F,g_0\oplus h\oplus b)
\]
is conform area extremal. 
\end{thm}
The first part of this theorem was already proved by the author in \cite{List9}. Note that $(M_0\times T^k\times F,g_0\oplus h\oplus b)$ can not be strict conform area extremal in case $k+\dim F>0$ because $g\oplus c_1 h\oplus c_2 b$ satisfies (\ref{conf_extremal}) for all $c_1,c_2\geq 1$.  The theorem follows immediately from theorem \ref{area_thm} below considering $M=M_0$ respectively $M=M_0\times T^k\times F$ and using the results in chapter \ref{chp2}.  In order to state theorem \ref{area_thm} we use the K--area respectively the K--area homology as it is defined in chapter \ref{chp2} and consider the transfer homomorphism $f^!:H_{n-k}(N;\Q )\to H_{m-k}(M;\Q )$ for maps $f:M^m\to N^n$ given by the composition of
\[
H_{n-k}(N;\Q )\to H^k(N;\Q )\stackrel{f^*}{\to } H^k(M;\Q )\to H_{m-k}(M;\Q ) 
\]
where the first and last homomorphism are determined by Poincar\'e duality on $N$ respectively on $M$. Remember that a smooth map $f:M\to N$ is called \emph{spin map} iff $f^*w_2(N)=w_2(M)$ for the second Stiefel--Whitney classes. Note that $\H _*(M;\Q )$ is defined as the set of rational homology classes with finite K--area which means $\widehat{\mathbf{A}}(M)\cap f^!(\eta )\notin  \H _*(M;\Q )$ for $\eta \in H_*(N;\Q )$ is equivalent to say that $\widehat{\mathbf{A}}(M)\cap f^!(\eta )\in H_*(M;\Q )$ has infinite K--area. 
\begin{defn}
\label{dilatation}
We denote the \emph{$\Lambda ^k$--dilatation function} of a smooth map $f:(M,g)\to (N,h)$ by $\mathrm{dil}_{k}(f):M\to [0,\infty )$ where
\[
\mathrm{dil}_{k}(f)(x):=\max _{v_1,\ldots ,v_k\in T_xM}\frac{|\mathrm{d}f(v_1)\wedge \cdots \wedge \mathrm{d}f(v_k)|_h}{|v_1\wedge \cdots \wedge v_k|_g}=\max _{0\neq \eta \in \Lambda ^kT_xM}\frac{|\mathrm{d}f^{\Lambda ^k}(\eta )|_h}{|\eta |_g}.
\]
In fact, $\mathrm{dil}_2(f)=\mathrm{area}\bigl( {\textstyle \frac{f^*h}{g}}\bigl) $.
\end{defn}
\begin{thm}
\label{area_thm}
Let $(M_0^m,g_0)$, $m\geq 3$, be an oriented closed and connected Riemannian manifold with nonnegative curvature operator, positive Ricci curvature and Euler characteristic $\chi (M_0)\neq 0$. Suppose that $(M^n,g)$ is an oriented closed Riemannian manifold and $f:M\to M_0$ is a smooth spin map with $\widehat{\mathbf{A}}(M)\cap f^!(1 )\notin \H _*(M;\Q )$ for a generator $1\in H_0(M_0;\Z ) $. If $g$ satisfies
\begin{equation}
\label{in_conf_ex}
\mathrm{scal}_g\geq \mathrm{dil}_{2}(f)\cdot \mathrm{scal}_{g_0}\circ f
\end{equation}
then equality holds and in case $\mathrm{scal}_g>0$,
\[
f:(M,\mathrm{dil}_{2}(f)\cdot g)\to (M_0,g_0)
\]
is a Riemannian submersion. In fact, if $n=m$ and $g$ satisfies (\ref{in_conf_ex}), then $\mathrm{dil}_2(f)$ is constant and $f:(M,\mathrm{dil}_2(f)\cdot g)\to (M_0,g_0)$ is a Riemannian covering. Moreover, these statements are also true in case $m=\dim M_0=2$ if we replace $\mathrm{dil}_2(f)$ everywhere by $\mathrm{dil}_1(f)^2$.
\end{thm}
\begin{proof}
We start with the case that $n=\dim M$ is even and use the results in section \ref{notes_bochner}. Since $f$ is a spin map, the second Stiefel--Whitney classes are related by $w_2(TM)=f^*w_2(TM_0)$, i.e.~the bundle $\bundle{F}=TM\oplus f^*TM_0$ admits a spin structure. The associated irreducible complex spinor bundle $\spinor \bundle{F}$ is naturally $\Z _2$--graded by the volume form of $\Cl _\C (\bundle{F})$. Using the embedding $\Cl _\C (TM)\hookrightarrow \Cl _\C (\bundle{F})$ and a connection induced from a connection on $\bundle{F}$, $\spinor \bundle{F}$ is a Dirac bundle. Moreover, if $\bundle{E}\to M$ is a Hermitian vector bundle with connection, the tensor product $\spinor \bundle{F}\otimes \bundle{E}$ is a $\Z _2$--graded Dirac bundle with Dirac operator $\D $. The index of $\D ^+:\Gamma (\spinor ^+\bundle{F}\otimes \bundle{E})\to \Gamma (\spinor ^-\bundle{F}\otimes \bundle{E})$ is given by  (cf.~\cite{LaMi} theorem 13.13 and proposition 11.24)
\[
\begin{split}
\mathrm{ind}\ \D ^+&=(-1)^k\left< \chi (f^*TM_0)\cdot \widehat{\mathbf{A}}(f^*TM_0)^{-1}\cdot \widehat{\mathbf{A}}(TM)\cdot \mathrm{ch}(\bundle{E}), [M]\right> \\
&=(-1)^k\chi (M_0)\left< \widehat{\mathbf{A}}(M)\cdot \mathrm{ch}(\bundle{E})\cdot f^*\omega , [M]\right> \\
&=(-1)^k\chi (M_0)\left< \mathrm{ch}(\bundle{E}), \widehat{\mathbf{A}}(M)\cap f^! (1)\right> 
\end{split}
\] 
where $2k=\dim M_0$ and $\omega \in H^{2k}(M_0)$ means the orientation class. Let $\nabla ^\bundle{F}=\nabla \oplus f^*\nabla ^0$ be the connection on $\bundle{F}$ induced from the Levi--Civita connection on $TM$ and the Levi--Civita connection $\nabla ^0$ on $TM_0$, then $\frak{R}^0$ denotes the twist curvature endomorphism of the Dirac bundle $\spinor \bundle{F} $ as in section \ref{notes_bochner}, i.e.
\[
\D ^2=\nabla ^*\nabla +\frac{\mathrm{scal}_g}{4}+\frak{R}^0\otimes \mathrm{id}+\sum _{i<j}\gamma (e_i)\gamma (e_j)\otimes R^\bundle{E}_{e_i,e_j}.
\]
Moreover, inequality (\ref{in_conf_ex}) and proposition \ref{main_prop} yield $\mathrm{scal}_g/4+\lambda _-(\frak{R}^0)\geq 0$ where $\lambda _-(\frak{R}^0)$ means the pointwise minimal eigenvalue of $\frak{R}^0$. Since the homology class $\widehat{\mathbf{A}}(M)\cap f^! (1)$ has infinite K--area, there exists for each $\epsilon >0$  a twist bundle $\bundle{E}$ with $\| R^\bundle{E}\| _g<\epsilon $ and corresponding Dirac operator $\mathrm{ind}(\D ^+)\neq 0$. Thus, proposition \ref{lem1342} proves $\mathrm{scal}_g/4+\lambda _-(\frak{R}^0)=0$ and
\[
0=4\lambda _-(\frak{R}^0)+\mathrm{scal}_g\geq 4\lambda _-(\frak{R}^0)+ \mathrm{dil}_2(f)\cdot \mathrm{scal}_{g_0}\circ f\geq 0
\]
supplies $\mathrm{scal}_g=\mathrm{dil}_2(f)\cdot \mathrm{scal}_{g_0}\circ f$. Since $\mathrm{scal}_g>0$ implies $\mathrm{dil}_2(f)>0$, proposition \ref{main_prop} shows that $f:(M,\mathrm{dil}_2(f)\cdot g)\to (M_0,g_0)$  is a Riemannian submersion. 

If $n=\dim M$ is odd, we apply the even--dimensional case to the Riemannian manifold $(\tilde M,\tilde g):=(M\times S^1,g\oplus \mathrm{d}t^2)$. Consider the map $\tilde f:M\times S^1\to M_0, (p,x)\mapsto f(p)$ and the projection $\pi :\tilde M\to M$, then
\[
\mathrm{scal}_{\tilde g}=\mathrm{scal}_g\circ \pi \ , \qquad \mathrm{dil}_2(\tilde f)\cdot \mathrm{scal}_{g_0}\circ \tilde f=\left[\mathrm{dil}_2(f)\cdot \mathrm{scal}_{g_0}\circ f\right] \circ \pi 
\]
and
\[
\begin{split}
\infty =\ke (M_g; \widehat{\mathbf{A}}(M)\cap f^!(1))&\stackrel{def}{=}\sup _{\mathrm{d}t^2}\ke (M_g\times S^1_{\mathrm{d}t^2};\widehat{\mathbf{A}}(M)\cap f^! (1) \times [S^1])\\
&= \ke (\tilde M_{\tilde g}; \widehat{\mathbf{A}}(\tilde M)\cap \tilde f ^!(1)).
\end{split}
\]
shows that $(\tilde M,\tilde g)$ satisfies the assumptions of the theorem for even dimensional manifolds. Thus, assuming (\ref{in_conf_ex}) and $\mathrm{scal}_g>0$, $\tilde f:(\tilde M,\mathrm{dil}_2(\tilde f)\cdot \tilde g)\to (M_0,g_0)$ is a Riemannian submersion  and $\mathrm{dil}_2(\tilde f)=\mathrm{dil}_2(f)\circ \pi $ supplies the claim for odd dimensional $M$.

Now, suppose that $n=m$, then $n$ is obviously even. Define the open set $U:=\{ p\in M|\ \alpha (p)>0\} $ where $\alpha :=\mathrm{dil}_2(f)$ for notational simplicity. Then $U$ is nonempty and $f:M\to M_0$ is surjective because $f^!(1)\neq 0$ yields $f_*f^!=\deg (f)\neq 0$. The above results show that the map $f:(U,\alpha \cdot g)\to (M_0,g_0)$ is a local isometry if we assume inequality (\ref{in_conf_ex}). We use $\mathrm{scal}_g=\alpha \cdot \mathrm{scal}_{g_0}\circ f$ to prove that $\alpha $ has to be constant which implies that $U=M$ and that $f:(M,\alpha g)\to (M_0,g_0)$ must be a Riemannian covering. We first note that $\mathrm{scal}_0\circ f>0$ implies that $\alpha $ is smooth on all of $M$ not only on $U$. Moreover, the scalar curvature of $g$ and $h:=\alpha g$ are on $U$ related by
\[
\mathrm{scal}_h=\frac{1}{\alpha }\mathrm{scal}_g+\frac{n-1}{\alpha ^2}\delta \mathrm{d}\alpha -\frac{(n-1)(n-6)}{4\alpha ^3}|\mathrm{d}\alpha |^2_g.
\]
Since $h=f^*g_0$ on $U$, the scalar curvature of $h$ is given by $\mathrm{scal}_h=\mathrm{scal}_0\circ f$ and we conclude from $\mathrm{scal}_g=\alpha \cdot \mathrm{scal}_{g_0}\circ f$ on $U$:
\[
\frac{n-1}{\alpha ^2}\delta \mathrm{d}\alpha -\frac{(n-1)(n-6)}{4\alpha ^3}|\mathrm{d}\alpha |^2_g=0.
\]
Thus, since $\alpha $ is smooth on $M$ and $\alpha =0$ on $M-U$, the following equation holds on all of $M$ for all $k\geq 1$:
\[
\alpha ^{k}\delta \mathrm{d}\alpha -\frac{n-6}{4}\alpha ^{k-1}|\mathrm{d}\alpha |^2=0.
\]
Integrate over $M$ w.r.t.~the volume form of $g$ yields for all $k\geq 1$
\[
\left( k-\frac{n-6}{4}\right) \int\limits _M\alpha ^{k-1}|\mathrm{d}\alpha |^2=0
\]
and hence, $\mathrm{d}\alpha =0$ shows $\alpha =\mathrm{dil}_2(f)=\mathrm{const}$ which completes the proof (recall that $\alpha \not\equiv 0$ since $\deg (f)\neq 0$ means $U\neq \emptyset $).
\end{proof}

\section{Notes on curvature in Bochner--Weitzenb\"ock formulas}
\label{notes_bochner}
A result by Bourguignon states that a closed manifold $M$ admits a metric of strictly positive scalar curvature or any metric of nonnegative scalar curvature is Ricci flat. This is the starting point in Gromov and Lawson's proof that enlargeable spin manifolds do not carry metrics of positive scalar curvature. In their proof Gromov and Lawson assume the existence of positive scalar curvature and then they consider Dirac bundles $\spinor M\otimes \bundle{E}$ for a suitable twist bundle $\bundle{E}$ with curvature norm less than $\min \mathrm{scal}_g$. This yields a contradiction if the Dirac operator associated to $\spinor M\otimes \bundle{E}$ has nontrivial index. Below we show that the Gromov and Lawson argument also works without knowing Bourguignon's result. In fact, lemma \ref{lem1342} will be essential to conclude the main statement in section \ref{sec_conf_extremal}.  

Let $\bundle{S}$ be a (complex) Dirac bundle on a closed Riemannian manifold $(M^n,g)$, then $\frak{R}^{\bundle{S}}$ denotes the twist curvature endomorphism which appears in the Boch\-ner--Weitzenb\"ock formula:
\[
\D ^2=\nabla ^*\nabla +\frac{\mathrm{scal}_g}{4}+\frak{R}^{\bundle{S} }.
\] 
If $\bundle{E}\to M$ is a Hermitian vector bundle, i.e.~$\bundle{E}$ is a complex vector bundle endowed with a Hermitian metric and a metric connection, then $\bundle{S}\otimes \bundle{E}$ is a Dirac bundle and
\[
\frak{R}^{\bundle{S}\otimes \bundle{E} }=\frak{R}^{\bundle{S}}\otimes \mathrm{id}+\sum _{i<j}\gamma (e_i)\gamma (e_j)\otimes R^\bundle{E}_{e_i,e_j}
\]
where $\gamma (.)$ means Clifford multiplication on $\bundle{S}$ and $e_1,\ldots ,e_n$ is an orthonormal basis of $TM$. We denote by $\| R^\bundle{E}\| _{g}$ the $L^\infty $--operator norm of $R^\bundle{E}$ defined as in chapter \ref{chp2}, i.e.
\[
\| R^\bundle{E}\| _{g}:=\max _{0\neq v\wedge w\in \Lambda ^2TM}\frac{|R^\bundle{E}_{v,w}|_{op}}{|v\wedge w|_g}
\]
where $|\, .\, |_{op}$ means the operator norm on the fibers of $\mathrm{End}(\bundle{E})$. Moreover, $\lambda _-(.)$ denotes the (pointwise) minimal eigenvalue of selfadjoint endomorphism, in fact, $\lambda _-(\frak{R}^{\bundle{S} }):M\to \R $ is the function which assigns to $p\in M$ the minimal eigenvalue of $\frak{R}^{\bundle{S} }(p)\in \mathrm{End}(\bundle{S}_p)$. In the case that $\frac{\mathrm{scal}_g}{4}+\frak{R}^\bundle{S}\geq 0$ and $\bundle{E}$ is a bundle with $\epsilon $--flat curvature, we cannot directly apply the integrated Bochner--Weitzenb\"ock formula on $\bundle{S}\otimes \bundle{E}$  because we only have $\frac{\mathrm{scal}_g}{4}+\frak{R}^{\bundle{S}\otimes \bundle{E}}\geq -\epsilon $. Now, the following proposition deals with this problem of almost nonnegative curvature.  
\begin{prop}
\label{lem1342}
In the above situation suppose that $\frac{\mathrm{scal}_g}{4}+\lambda _-(\frak{R}^\bundle{S})\geq 0$, then there is a constant $C=C(g,n) $ with the following property: For any sufficiently small $\epsilon >0$ and Hermitian bundles $\bundle{E}$ with $\| R^\bundle{E}\| _g<\epsilon $ and $\mathrm{ind}(\D ^{\bundle{S}\otimes \bundle{E}})\neq 0$ one has
\[
\int \left[\frac{\mathrm{scal}_g}{4}+\lambda _-(\frak{R}^\bundle{S})\right] \cdot \mathrm{vol}_g\leq C\cdot \sqrt{\epsilon }.
\]
In particular, if for all $\epsilon >0$ there is a twist bundle $\bundle{E}$ with $\mathrm{ind}(\D ^{\bundle{S}\otimes \bundle{E}})\neq 0$ and $\| R^\bundle{E}\| _g<\epsilon $, then $\frac{\mathrm{scal}_g}{4}+\lambda _-(\frak{R}^\bundle{S})=0$.
\end{prop}
\begin{proof}
Note that we cannot use Sobolev embeddings for sections of $\bundle{S}\otimes \bundle{E}$ because the embedding constants depend on the choice of the twist bundle $\bundle{E}$. We assume without loss of generality $\mathrm{Vol}(M,g)=1$ and integrate with respect to the volume form of $g$. Define the function $\alpha :=\frac{\mathrm{scal}_g}{4}+\lambda _-(\frak{R}^\bundle{S}):M\to [0,\infty )$ for notational simplicity. Let $\phi \in \ker (\D ^{\bundle{S}\otimes \bundle{E}})$ be nontrivial, then the local Bochner--Weitzenb\"ock formula for $\D ^{\bundle{S}\otimes \bundle{E}}$ and a standard exercise show
\begin{equation}
\label{local_boc}
\begin{split}
0&=\frac{1}{2}\delta _g\mathrm{d}|\phi |^2+|\nabla ^{\bundle{S}\otimes \bundle{E}}\phi |^2+\frac{\mathrm{scal}_g}{4}|\phi |^2+\left< \frak{R}^{\bundle{S}\otimes \bundle{E}}\phi ,\phi \right>\\
&\geq \frac{1}{2}\delta _g\mathrm{d}|\phi |^2+|\nabla ^{\bundle{S}\otimes \bundle{E}}\phi |^2+\alpha |\phi |^2-\frac{n(n-1)}{2}\| R^\bundle{E}\| _g\cdot |\phi |^2.
\end{split}
\end{equation}
Hence, multiplication of this inequality by $|\phi |^2$ and using $\alpha \geq 0$ supplies
\[
\int \Bigl| \nabla |\phi |^2\Bigl |^2\leq n(n-1)\| R^\bundle{E}\| _g\int |\phi |^4<n(n-1)\epsilon \int |\phi |^4.
\]
By a constant rescaling of $\phi $ we assume $\int |\phi |^2=1$, i.e.~the Poincar\'e inequality for $(M,g)$ yields a constant $C'=\frac{C}{n(n-1)}>0$ (independent on $\phi $ and $\epsilon $) with
\[
\int (|\phi |^2-1)^2\leq C'\cdot \int \Bigl| \nabla |\phi |^2\Bigl| ^2< C\epsilon \int |\phi |^4=C\epsilon \left( 1+\int (|\phi |^2-1)^2\right) .
\]
Hence, if $C\epsilon \leq 1/2$, then
\[
\int (|\phi |^2-1)^2\leq 2C\epsilon 
\]
which implies that $|\phi |^2=1+h$ for a $L^1$--function $h$ with $\| h \|_{L^1}\leq \sqrt{2C\epsilon }$. Integration of (\ref{local_boc}) supplies for $\epsilon \leq \min \{ 1,\frac{1}{2C}\} $:
\[
\begin{split}
0&\geq \int \left[ |\nabla ^{\bundle{S}\otimes \bundle{E}}\phi |^2+\alpha (1+h)-\frac{n(n-1)}{2}\epsilon |\phi |^2\right]\\
&\geq -\max \alpha \cdot \| h\| _{L^1}-\frac{n(n-1)}{2}\epsilon +\int \alpha \\
& \geq - \left[ \max \frac{\mathrm{scal}_g}{4}\sqrt{2C}+\frac{n(n-1)}{2}\right] \cdot \sqrt{\epsilon }+\int \alpha  ,
\end{split}
\]
here we use $\lambda _-(\frak{R}^\bundle{S})\leq 0$, i.e.~$0\leq 4\cdot \max \alpha \leq \max \mathrm{scal}_g$. Thus, the assumption $\alpha \geq 0$ and the fact that $\alpha $ is continuous provide the claim.
\end{proof}

In the remainder of this section we show the main inequality to deduce theorem \ref{area_thm}. In order to get a simple expression for the index of the Dirac operator in terms of characteristic classes, we are using the approach presented in \cite{GoSe1}. We assume that $(M_0^m,g_0)$ is an oriented manifold with non--negative Riemannian curvature operator $\curv{R}^0:\Lambda ^2TM_0\to \Lambda ^2TM_0$, in particular $g_0$ has non--negative sectional curvature. Suppose $(M,g)$ is an oriented manifold of dimension $n$ and $f:M\to M_0$ is a spin map, then the vector bundle
\[
\bundle{F}:=TM\oplus f^*TM_0
\]
admits a spin structure. The Levi--Civita connection of $g$ and the Levi--Civita connection of $g_0$ induce the connection $\nabla ^\bundle{F}=\nabla \oplus f^*\nabla ^0$ on $\bundle{F}$ which is Riemannian with respect to $g\oplus f^*g_0$. The complex Clifford bundle of $\bundle{F}$ is given by
\[
\Cl _\C (\bundle{F})=\Cl _\C (TM)\widehat{\otimes } f^*\Cl _\C (TM_0).
\]
Since the $\mathrm{SO}$--frame bundle of $\bundle{F}$ has structure group $\mathrm{SO}(n)\times \mathrm{SO}(m)$, the structure group of a spin structure on $\bundle{F}$ is reducible to
\[
\mathrm{Spin}(n)\cdot \mathrm{Spin}(m):=\mathrm{Spin}(n)\times \mathrm{Spin}(m)/\{ \pm 1\} \subseteq \mathrm{Spin}(m+n).
\]
Let $\spinor \bundle{F}$ be the complex spinor bundle of $\bundle{F}$ induced by a choice of the spin structure and the tensor product of the complex spin representations. The connection $\nabla ^\bundle{F}$ lifts uniquely to a Riemannian connection $\overline{\nabla }$ on $\spinor \bundle{F}$. For each point $p\in M$ there are neighborhoods $p\in U\subseteq M$ and $f(p)\in V\subseteq M_0$ in such a way that the spinor bundle decomposes as
\begin{equation*}
\spinor \bundle{F}_{|U}=\spinor U\otimes f^*(\spinor V)
\end{equation*}
(note that we do not assume $\spinor \bundle{F}$ to be irreducible at this point). In particular, if $M_0$ is spin, then $M$ is spin and we conclude $\spinor \bundle{F}=\spinor M\widehat{\otimes }f^*\spinor M_0$ as $\Cl _\C (\bundle{F})$ module. Since the forthcoming computations are of local nature we use the notation $\spinor \bundle{F}=\spinor M\otimes f^*\spinor M_0$ even if $M_0$ is not spin. Because $\spinor \bundle{F}$ is a $\Cl _\C (\bundle{F})$--module, $\spinor \bundle{F}$ becomes a complex Dirac bundle over $M$ if we use the imbedding
\[
\Cl _\C (TM)\hookrightarrow \Cl _\C (TM)\otimes \mathbbm{1}\subseteq \Cl _\C (\bundle{F}).
\]
The corresponding Dirac operator will be denoted by
\[
\overline{\dirac }:=\sum _{j=1}^n\gamma (e_j)\overline{\nabla }_{e_j}
\]
where $e_1,\ldots ,e_n$ is an orthonormal basis of $T_pM$. Since the connection $\overline{\nabla }$ preservers the decomposition of the formal tensor product $\spinor \bundle{F}=\spinor M\otimes f^*\spinor M_0$, i.e.~$\overline{\nabla }=\nabla \otimes \mathbbm{1}+\mathbbm{1}\otimes f^*\nabla ^0$, we can use \cite[Ch.~II Thm.~8.17]{LaMi} to compute the Bochner--Weitzenb\"ock formula of $\overline{\dirac }$:
\begin{equation*}
\overline{\dirac }^2 =\overline{\nabla }^* \overline{\nabla }+\frac{\mathrm{scal}_g}{4}+\frak{R}^{0}
\end{equation*}
where $\frak{R}^0$ is defined by
\[
\frak{R}^0=\frac{1}{2}\sum _{i,j=1}^n\gamma (e_i)\gamma (e_j)\otimes R_{e_i,e_j}^0
\]
and $R^0$ means the curvature of $f^*\nabla ^0$.
\begin{defn}
Suppose $(V,\langle .,.\rangle _V)$ and $(W,\langle .,.\rangle _W)$ are inner product spaces and $\beta :V\to W$ is a linear transformation. Then $\beta $ is a \emph{homothetic injection} if there is a constant $c>0$ such that
\begin{equation}
\label{homothetic}
\left< v,\hat v\right> _V=c \cdot \left< \beta (v),\beta (\hat v) \right> _W 
\end{equation}
for all $v,\hat v\in V$. Let $V=\ker (\beta )\oplus V^\prime $ be the orthogonal decomposition of $V$ w.r.t.~the inner product. Then $\beta $ is said to be a \emph{homothetic surjection} if (\ref{homothetic}) holds for all $v,\hat v\in V^\prime $. A homothetic isomorphism is also called \emph{homothety}.
\end{defn}
\begin{prop} 
\label{main_prop}
Let the Riemannian curvature operator of $(M_0,g_0)$ be non--negative and $\dim M_0\geq 3$, then the curvature endomorphism $\frak{R}^0$ is at each point bounded as follows
\begin{equation}
\label{ineq}
\frak{R}^0\geq -\mathrm{dil}_2(f)\cdot \frac{\mathrm{scal}_{g_0}\circ f}{4}
\end{equation}
with $\mathrm{dil}_2(f)$ given in definition \ref{dilatation}. If $\mathrm{Ric}(g_0)$ is positive definite at $f(p)\in M_0$ and $-\frac{1}{4}\mathrm{dil}_2(f)(p)\cdot \mathrm{scal}_{g_0}( f(p))$ is the minimal eigenvalue of $\frak{R}^0$ at $p\in M$, then $f_*:T_pM\to T_{f(p)}M_0$ is a homothetic surjection or $\mathrm{dil}_2(f)(p)=0$. 

In particular, if $\mathrm{Ric}(g_0)>0$ on $M_0$ and $U\subseteq M$ denotes the interior of all points  $p\in M$ where the minimal eigenvalue of $\frak{R}^0$ is $-\frac{1}{4}\mathrm{dil}_2(f)(p)\cdot \mathrm{scal}_{g_0}(f(p))$ and $\mathrm{dil}_2(f)(p)>0 $, then
\[
f:(U,\mathrm{dil}_2(f)\cdot g) \to (M_0,g_0)
\]
is a Riemannian submersion (not necessarily surjective). 
\end{prop}
Note that $\mathrm{dil}_2(f)$ is smooth on $U$, while of course $U$ could be empty in this proposition. Recall that if $\mathrm{dil}_2(f)$ vanishes at $p$, then the image of $f_*:T_pM\to T_{f(p)}M_0$ is at most one dimensional and $\mathrm{Im}(f_*)$ is trivial if and only if $\mathrm{dil}_1(f)=0$. The assumption $\dim M_0\geq 3$ can be omitted in the above proposition if we replace the function $\mathrm{dil}_2(f)$ everywhere by $\mathrm{dil}_1(f)^2$ (in case $\dim M_0=2$, $\Lambda ^2TM_0$ has rank one which will not be enough to show that $f_*$ is a homothetic surjection if $\mathrm{dil}_2(f)>0$). The following two examples provide non--constant maps which are not homothetic surjections but satisfy all the assumptions of the proposition except $\mathrm{Ric}(g_0)>0$ respectively $\mathrm{dil}_2(f)>0$ ($T^n$ means the standard $n$--dimensional flat torus, $c:[0,2\pi )\to S^n$ is a simple closed geodesic):
\[
f:T^n\to S^n,\ (t_1,\ldots ,t_n)\mapsto c(t_1)\ , \qquad \widetilde{f}:T^n\to T^n\times S^1, \ p\mapsto (p,t_0).
\]
In both cases $\frak{R}^0$ vanishes and inequality (\ref{ineq}) is an equality. Although $\mathrm{Ric}(g_{S^n})$ is positive definite, $f$ is nowhere a homothetic surjection since $\mathrm{dil}_2(f)=0$. Moreover, $\widetilde{f}$ satisfies $\mathrm{dil}_2(f)\equiv 1$, but $\widetilde{f}$ is nowhere a homothetic surjection since $\mathrm{Ric}(g_0)=0$. In fact, in both cases we have $U=\emptyset $.

In order to show proposition \ref{main_prop} we will simplify the curvature expression $\frak{R}^0$. For each point $p\in M$ the map $f:M\to M_0$ induces an isometric isomorphism $\beta _p:f^*\spinor \frak{p}\to \spinor \frak{p}$ where $\frak{p}=T_{f(p)}M_0$ and $\spinor \bundle{F}_{p}=\spinor T_pM\otimes f^*\spinor \frak{p}$. Since the curvature of the connection $\nabla ^0$ is the curvature of the (virtual) spin bundle $\spinor M_0$, we obtain the curvature of $f^*\nabla ^0$
\begin{equation*}
\begin{split}
R^0_{v,w}&=\beta ^{-1}_p\circ R^{0}_{f_*v,f_*w}\circ \beta _p\\
&=-\frac{1}{2}\beta ^{-1}_p\circ \gamma ^0(\curv{R}^0(f_*(v\wedge w)))\circ \beta _p
\end{split}
\end{equation*}
where $v,w\in T_pM$, $\gamma ^0 $ is the Clifford multiplication on $\spinor \frak{p}$ and $\curv{R}^0$ is the Riemannian curvature operator of $(M_0,g_0)$ considered as endomorphism on $\Lambda ^2TM_0$. Thus, the curvature operator $\frak{R}$ is determined by
\begin{equation}
\label{curv11}
\begin{split}
\frak{R}^0&=\frac{1}{2}\sum _{ i,j=1}^{n}\gamma (e_i)\gamma (e_j)\otimes R_{e_i,e_j}^0\\
&=-\frac{1}{4}\sum _{i,j=1}^{n}\gamma (e_i\wedge e_j)\otimes \beta _p^{-1}\gamma ^0\Bigl( \curv{R}^0(f_*e_i\wedge f_*e_j)\Bigl) \beta _p.
\end{split}
\end{equation}
Let $B\in \Gamma (\mathrm{End}(TM))$ be the symmetric positive semi--definite transformation defined by
\begin{equation}
\label{defn_of_B}
\begin{split}
g(BX,Y)&=f^*g_0(X,Y)=g_0(f_*(X),f_*(Y))
\end{split}
\end{equation}
and set $B_k:=\underbrace{B\otimes \cdots \otimes B}_{k-\text{times}}\in \Gamma (\mathrm{End}(\Lambda ^kTM))$, then $B_k$ satisfies
\[
0\leq B_k\leq \mathrm{dil }_k(f) ^{2}
\]
[note that the upper inequalities are sharp in each point, since by definition, $B_k$ has an eigenvalue $\mathrm{dil}_k(f)^2$]. 
\begin{lem}
\label{lem_lin}
Suppose $f:M\to N$ is a differentiable map, $g$ is a Riemannian metric on $M$ and $\overline{g}$ is a Riemannian metric on $N$. Then $g$ and $\overline{g}$ induce isomorphisms $t:\Lambda ^kT^*_pM\rightarrow \Lambda ^kT_pM$, $r:\Lambda ^kT_qN\rightarrow \Lambda ^kT_q^*N$ and the following diagram is commutative:
\[
\begin{xy}
\xymatrix{
\Lambda ^kT_pM \ar[r]^{f_*} \ar[rd]^{B_k}& \Lambda ^kT_{f(p)}N\ar@{.>}[d]^{f^\#}\ar[r]^{r} & \Lambda ^kT_{f(p)}^*N\ar[d]^{f^*}\\
& \Lambda ^kT_pM &\ar[l]^t  \Lambda ^kT_p^*M
}
\end{xy}
\]
In particular, $f^\#:\Lambda ^kT_{f(p)}N\to \Lambda ^kT_pM$ is uniquely determined for each $p\in M$ by $f^\#:=t\circ f^*\circ r$ and satisfies
$f^\#\circ f_*=B_k$ as well as
\[
\overline{g}(f_*v,w)=g(v,f^\#w)
\]
for all $v\in \Lambda ^kT_pM$ and $w\in \Lambda ^kT_{f(p)}N$. Furthermore, define
\[
\breve{B}_k:=f_*f^\# :\Lambda ^kT_{f(p)}N\to \Lambda ^kT_{f(p)}N,
\]
then $\breve{B}_k$ is positive semi--definite and symmetric w.r.t.~$\overline{g}$. Moreover, the non--vanishing eigenvalues of $B_k$ at $p\in M$ coincide with the non--vanishing eigenvalues of $\breve{B}_k$ at $f(p)\in N$ and we have $\breve{B}_k=\underbrace{\breve{B}\otimes \cdots \otimes \breve{B}}_k$ with $\breve{B}=\breve{B}_1$. 
\end{lem}
\begin{proof}
This lemma collects some facts from linear algebra. Note that $\breve{B}$ can be defined analogous to (\ref{defn_of_B}) by
\[
\begin{split}
\overline{g}(\breve{B}x,y)&=g(f^\# x,f^\# y)
\end{split}
\]
for all $x,y\in T_{f(p)}N$. If $v\in T_pM$ is an eigenvector of $B$ to the eigenvalue $\lambda \neq 0$, then (\ref{defn_of_B}) yields that $f_*v\neq 0$. In particular, $f_*v\in T_{f(p)}N$ is an eigenvector of $\breve{B}$ to the eigenvalue $\lambda $:
\[
\breve{B}f_*v=(f_*f^\# )f_*v=f_*(f^\# f_*)v=f_*Bv=\lambda f_*v
\]
and thus, $\lambda $ is an eigenvalue of $\breve{B}$ (appears with the same multiplicity as in $B$). That any nonzero eigenvalue of $\breve{B}$ is an eigenvalue of $B$ follows in the same way.
\end{proof}

Let $h_1,\ldots ,h_s\in \Lambda ^2T_{f(p)}M_0$ be a $g_0$--orthonormal eigenbasis of $\curv{R}^0$ and $\kappa _1,\ldots ,\kappa _s$ be the corresponding eigenvalues (note that $\curv{R}^0$ is symmetric), then $\curv{R}^0\geq 0$ yields
\[
\kappa _j=\left< \curv{R}^0(h_j),h_j\right> \geq 0.
\]
Furthermore, we obtain from (\ref{curv11}) and the symmetry of $\curv{R}^0$ ($e_1,\ldots ,e_n$ is a $g$--orthonormal basis of $T_pM$):
\begin{equation*}
\begin{split}
\frak{R}^0&=-\frac{1}{4}\sum _{i,j=1}^n\sum _{l=1}^sg_0(f_*(e_i\wedge e_j),\curv{R}^0(h_l))\cdot \gamma (e_i\wedge e_j)\otimes \beta _p^{-1}\gamma ^0(h_l) \beta _p\\
&=-\frac{1}{2}\sum _{l=1}^s\kappa _l\gamma (f^\# h_l)\otimes \beta _p^{-1}\gamma ^0(h_l)\beta _p.
\end{split}
\end{equation*}
Let $\alpha \in [0,\infty )$ be non--negative and define
\begin{equation}
\label{defnC}
\frak{C}:=\sum _{l=1}^s\kappa _l[\gamma (f^\# h_l)\otimes \mathrm{Id}+\alpha \cdot \mathrm{Id}\otimes \beta _p^{-1}\gamma ^0(h_l)\beta _p]^2 \in \mathrm{End}(\spinor \bundle{E}_p).
\end{equation}
then a straightforward calculation shows
\[
\frak{C}=\sum _{l=1}^s\kappa _l[(\gamma (f^\# h_l))^2\otimes \mathrm{Id}+\alpha ^2\cdot \mathrm{Id}\otimes \beta _p^{-1}(\gamma ^0(h_l))^2\beta _p]-4\alpha \frak{R}^0.
\]
\begin{lem}
Suppose $\eta $ is a $2$--form, then
\[
\gamma (\eta )^2=\gamma (\eta \wedge \eta )-|\eta |^2
\]
for any Clifford module. Moreover, let $\kappa _1,\ldots ,\kappa _s$ be the eigenvalues of $\curv{R}^0$ and $h_1,\ldots ,h_s$ be the corresponding orthonormal eigenbasis, then $\sum \kappa _l=\frac{1}{2}\mathrm{scal}_0$ and
\[
\sum _{l=1}^s\kappa _lh_l\wedge h_l=0.
\]
\end{lem}
\begin{proof}
The first statement is a straightforward calculation and the second follows from $\mathrm{scal}_{g_0}=2\mathrm{trace}(\curv{R}^0)$. Moreover, the first Bianchi identity yields after applying vectors $x,y,z,t$:
\[
\sum _{l=1}^s\kappa _lh_l\wedge h_l=\sum _l\curv{R}^0(h_l)\wedge h_l=0.
\]
\end{proof}
Using the fact $f^\# (h_l\wedge h_l)=f^\# h_l\wedge f^\# h_l$, this lemma simplifies $\frak{C}$ to:
\begin{equation}
\label{simplC}
\frak{C}=-4\alpha \frak{R}^0-\frac{\alpha ^2}{2}\cdot \mathrm{scal}_{g_0}(f(p))-\sum _{l=1}^s\kappa _l |f^\#h_l|^2_g.
\end{equation}
Since $0\leq B_2\leq \mathrm{dil}_2(f)^2$ and the nonzero eigenvalues of $B_2$ and $\breve{B}_2$ coincide, lemma \ref{lem_lin} yields
\[
g(f^\# h_l,f^\# h_l)=g_0(f_*f^\#h_l,h_l)=g_0(\breve{B}_2h_l,h_l)\leq \mathrm{dil}_2(f)^2.
\]
In particular, $\kappa _j\geq 0$ and $\sum _l\kappa _l=\frac{1}{2}\mathrm{scal}_{g_0}(f(p))$ prove the following inequality
\begin{equation}
\label{ineq_loc}
\sum _{l=1}^s\kappa _l|f^\# h_l|^2_g \leq \frac{1}{2}\mathrm{dil}_2(f)^2\cdot \mathrm{scal}_{g_0}(f(p)).
\end{equation}
This estimate completes the proof for the first part of proposition \ref{main_prop} by the following argument. Consider the definition of $\frak{C}$ in (\ref{defnC}). Since $\gamma (\eta )$ is a skew adjoint action on any Clifford module for arbitrary $2$--forms $\eta $ and $\kappa _l\geq 0$ for all $l=1\ldots s$, we conclude from (\ref{defnC}): $\frak{C}\leq 0$. Hence, set $\alpha :=\mathrm{dil}_2(f):M\to [0,\infty )$ then equation (\ref{simplC}) and inequality (\ref{ineq_loc}) show
\begin{equation}
\label{ineq_fin}
\frak{R}^0\geq -\frac{\alpha }{8}\mathrm{scal}_{g_0}\circ f-\frac{1}{4\alpha }\sum _{l=1}^m\kappa _l|f^\# h_l|^2_g\geq -\frac{\alpha }{4}\mathrm{scal}_{g_0}\circ f
\end{equation}
(note that in case $\alpha =\mathrm{dil}_2(f)=0$ at a point $p\in M$, $f_*:\Lambda ^2T_pM\to \Lambda ^2T_{f(p)}M_0$ has to vanish and thus, $\frak{R}^0$ vanishes at $p\in M$ from (\ref{curv11}), i.e.~this inequality is also true at points where $\alpha = 0$). In order to show the second part of proposition \ref{main_prop} we need the conditions $\mathrm{Ric}(g_0)>0$ and $m=\dim M_0\geq 3$. We have to prove that $f_*:T_pM\to T_{f(p)}M_0$ is a homothetic surjection if $\mathrm{dil}_2(f)(p)>0$. Suppose the minimal eigenvalue of $\frak{R}^0$ at $p\in M$ is $-\frac{1}{4}\alpha \cdot \mathrm{scal}_{g_0}(f(p))$ with $\alpha :=\mathrm{dil}_2(f)(p)>0$.  In this case we obtain equality for at least one nontrivial spinor in (\ref{ineq_fin}) and henceforth, we obtain equality in (\ref{ineq_loc}). In particular, $| f^\# h_l| _g=\alpha $ for all $l$ with $\kappa _l> 0$ which is equivalent to $\breve {B}_2 =\alpha ^2\mathrm{Id}$ on $\mathrm{Im}(\curv{R}^0)\subseteq \Lambda ^2T_{f(p)}M_0$. Let $e_1,\ldots ,e_m\in T_{f(p)}M_0$ be an orthonormal eigenbasis of $\breve{B}=\breve{B}_1$ to eigenvalues $\lambda _1\geq \ldots \geq \lambda _m$. Then $\lambda _i\lambda _j$ ($i\neq j$) are the eigenvalues of $\breve{B}_2$ and $0\leq \breve{B}_2\leq \alpha ^2$ yields $\lambda _i\lambda _j\leq \alpha ^2 $ for all $i,j=1\ldots m$ with $i\neq j$. Moreover, $\breve{B}_2=\alpha ^2$ on $\mathrm{Im}(\curv{R}^0)$ supplies for the curvature of $(M_0,g_0)$ (use the symmetry of $\breve{B}_2=\breve{B}\otimes \breve{B}$):
\[
\begin{split}
\alpha ^2\cdot R^0(e_i,e_j,e_k,e_l)&=\alpha ^2\cdot g_0(\curv{R}^0(e_i\wedge e_j ),e_l \wedge e_k )=g_0(\breve{B}_2\curv{R}^0(e_i\wedge e_j),e_l\wedge e_k)\\
&=g_0(\curv{R}^0(e_i\wedge e_j),\breve{B}e_l\wedge \breve{B}e_k)=\lambda _k\lambda _lR^0(e_i,e_j,e_k,e_l).
\end{split}
\]
Since $\mathrm{Ric}(g_0)>0$ at $f(p)$, for any $k$ there is some $l$ with $R^0(e_l,e_k,e_k,e_l)>0$. Thus, for all $k=1\ldots m$ there is some $l\neq k$ with $\lambda _k\lambda _l=\alpha ^2$. We assumed $\alpha >0$ and $m=\dim M_0\geq 3$. Thus, let $k$ be arbitrary and $ l\neq k$ in such a way that $\lambda _k\lambda _l=\alpha ^2$. Suppose $i\neq k$ as well as $i\neq l$. Then
\[
\lambda _i\lambda _k\leq \alpha ^2,\qquad \lambda _i\lambda _l\leq \alpha ^2
\]
together with $\lambda _k\lambda _l=\alpha ^2$ yields $\lambda _i\leq \alpha $ and since $k$ was arbitrary, we conclude $\lambda _i\leq \alpha $ for all $i=1\ldots m$. Because for any $i$ there is some $j$ with $\lambda _i\lambda _j=\alpha ^2$, we obtain $\lambda _i=\alpha $ for all $i=1\ldots m$ which is equivalent to $\breve{B}=\alpha \mathrm{Id}$. Since the non--vanishing eigenvalues of $B$ and $\breve{B}$ coincide (lemma \ref{lem_lin}), definition \ref{defn_of_B} proves that $f_*:T_pM\to T_{f(p)}M_0$ is a homothetic surjection in case $\mathrm{dil}_2(f)(p)>0$ (we have $f_*$ is surjective, $f^\# $ is injective, $T_pM=\ker (f_*)\oplus V$ and $B_1=\alpha $ on $V=\mathrm{Im}(f^\# )$ as well as $B_1=0$ on $\ker (f_*)$). 

Suppose now that $\mathrm{Ric}(g_0)>0$. Define $U\subseteq M$ to be the interior of all point $p\in M$ where the minimal eigenvalue of $\frak{R}^0$ is $-\frac{1}{4}\alpha \cdot \mathrm{scal}_{g_0}(f(p))$ and where $\mathrm{dil}_2(f)(p)>0$. Then the above considerations show that $f:U\to M_0$ is a submersion and (\ref{defn_of_B}) [$B=0$ on $\ker (f_*) \subseteq TU$ as well as $B_1=\alpha $ on $\ker (f_*)^\perp \subseteq TU$] proves that 
\[
f:(U,\alpha g)\to (M_0,g_0)
\]
is a Riemannian submersion with $\alpha =\mathrm{dil}_2(f)$.

\section{Strictly conform area examples with vanishing Euler characteristic}
\label{strict_conf}
The previous results on strict conform area extremality are based on the Atiyah--Singer index theorem for the signature operator or the Euler characteristic operator. The reason for using only the Euler characteristic operator is that our model spaces $M_0$ have nontrivial signature only in case $\chi (M_0)\neq 0$ by the following remark which follows from theorem 20 in \cite{Brend} and the main result in \cite{CaoChow}.
\begin{rem}[\cite{Brend,CaoChow}]
If $(M_0^n,g_0)$ is closed, simply connected and irreducible with nonnegative curvature operator, then one of the following cases occurs:
\begin{enumerate}
\item[(i)] $M_0$ is homeomorphic to $S^n$.
\item[(ii)] $(M_0,g_0)$ is K\"ahler and biholomorphic to $\proj{\C }^{n/2}$.
\item[(iii)] $(M_0,g_0)$ is a symmetric space of compact type.
\end{enumerate}
\end{rem}
\begin{proof}
If $g_0$ has nonnegative curvature operator, each harmonic form is parallel. For $0<k<n$ any nontrivial parallel $k$--form yields a restriction on the holonomy group of $(M_0,g_0)$. Thus, $\mathrm{Hol}(M_0,g_0)=\mathrm{SO}(n)$ implies vanishing of the Betti numbers $b_1,\ldots ,b_{n-1}$ and $M_0$ is a rational homology sphere. In order to conclude that $M_0$ is homeomorphic to $S^n$ we refer to \cite[theorem 20]{Brend}. The next possible holonomy group in Berger's list is $\mathrm{U}(n/2)$. In this case $(M_0,g_0)$ is a K\"ahler manifold with nonnegative curvature operator and the main result in \cite{CaoChow} shows that $(M_0,g_0)$ is a Hermitian symmetric space or $M_0$ is biholomorphic to $ \proj{\C }^{n/2}$. The only other holonomy group in Berger's list which needs to be considered is $\mathrm{Sp}(1)\cdot \mathrm{Sp}(n/4)$ because otherwise  $(M_0,g_0)$ is a symmetric space or $g_0$ is Ricci flat which together  with nonnegative curvature operator yields $R^{g_0}=0$, a contradiction to the classification of space forms. However, quaternionic K\"ahler manifolds are Einstein, and Einstein spaces with nonnegative curvature operator are locally symmetric (use a generalization of the main theorem in \cite{Tach}). 
\end{proof} 
Using real Dirac bundles and indices of Dirac operators in $\mathrm{KO}$--groups, one can also show conform area extremality of odd dimensional symmetric spaces with positive Ricci curvature if the Kervaire semi--characteristic is nontrivial. However, for index theoretic reasons it seems rather difficult to conclude similar statements for symmetric spaces $G/H$ of compact type if $\mathrm{rk}(G)-\mathrm{rk}(H)>1$. 
\begin{thm}[\cite{List9}]
Suppose that $(M_0^{4n+1},g_0)$ is a closed, oriented and locally symmetric space with positive Ricci curvature and Kervaire semi--charac\-teris\-tic $\sigma (M_0)\neq 0$, then $(M_0,g_0)$ is strict conform area extremal.
\end{thm}
The proof of this theorem follows from the results in section \ref{sec_conf_extremal} and \ref{notes_bochner} using the real $\Cl _1$--Dirac bundle $\Lambda ^*M_0$. Note that this theorem does not apply to $S^{4k-1}$, i.e.~conform area extremality for these spaces needs yet to be investigated. One may try to adapt Kramer's proof  \cite{Kra}  on area extremality of odd dimensional spheres to show the conform area extremality but we will use a slightly different approach. Note that the previous theorem covers the strictly conform area extremality of $\R P^{4n+1}$ which does not follow from theorem \ref{thm343} below because $\R P^{4n+1}$ is not spin if $n\geq 1$. Conversely the spin manifolds $\R P^{4n-1}$ have trivial Kervaire semi--characteristic, but $\R P^{4n-1}$ is spin and therefore, strictly conform area extremal: 
\begin{thm}
\label{thm343}
Let $(M_0^n,g_0)$, $n\geq 3$, be a closed spin manifold of constant sectional curvature $K_{g_0}>0$. Then $(M_0,g_0)$ is strictly conform area extremal. In fact, if $(M,g)$ is a closed spin manifold and $f:M\to M_0$ has nontrivial degree, then
\[
\mathrm{scal}_g\geq n(n-1)\mathrm{dil}_2(f)\cdot K_{g_0}
\]
implies equality and $f:(M,c\cdot g)\to (M_0,g_0)$ is a Riemannian covering for some constant $c>0$. 
\end{thm}
The main ingredient of this theorem will be corollary \ref{cor345} which computes the total K--area of Riemannian manifolds with constant sectional curvature. In order to show corollary \ref{cor345} we use the relative K--area.
\begin{lem}
Let $(M^n,g)$ be a connected closed Riemannian manifold and $\{ p_1,\ldots ,p_s\}$ be a finite set of points on $M$, then for $\theta \in \widetilde{H}_{2*}(M)$:
\[
\ke (M_g,\{ p_1,\ldots ,p_s\} ;j_*\theta )=\ke (M_g,\{ p_1,\ldots ,p_s\} ;\theta )=\ke (M_g;\theta )
\]
where $j:(M,\emptyset )\to (M,\{ p_1,\ldots ,p_s\} )$.
\end{lem}
\begin{proof}
The inequality $\leq $ is obvious by the definitions, i.e.~it remains to show $\geq $. Fix $\delta >0$ in such a way that $B_\delta (p_i)\cap B_\delta (p_j)=\emptyset $ for all $i\neq j$, then there is a constant $C=C(g,\delta )>0$ with the following property.  For all $0<\epsilon <\delta /2$ there are smooth contraction maps $f^{j}:M\to M$ such that $f^{j}\simeq \mathrm{id}_{M}$, $f^{j}(B_\epsilon (p_j))=p_j$, $f^{j}_{|M\setminus B_\delta (p_j)}=\mathrm{id}$ and
\[
\| \mathrm{d}f^{j}\| :=\max _{v\in TM}\frac{|\mathrm{d}f^{j}(v)|_g}{|v|_g}\leq 1+C\cdot \epsilon .
\]
Consider a bundle $\bundle{E}\in \mathscr{V}(M;\theta )$ and define $\bundle{E}':=(f^1)^*\cdots (f^s)^*\bundle{E}$, then $\bundle{E}'\in \mathscr{V}(M,\{ p_1,\ldots ,p_s\} ;\theta )$ and
\[
\| R^{\bundle{E}'}\| _g\leq \| \mathrm{d}f^1\| ^2\cdots \| \mathrm{d}f^s\| ^2\cdot \| R^\bundle{E}\| _g\leq (1+C\cdot \epsilon )^{2s}\| R^\bundle{E}\| _g.
\]
Hence, we obtain for all $0<\epsilon <\delta /2$:
\[
\ke (M_g;\theta )\leq (1+C\epsilon )^{2s}\ke (M_g,\{ p_1,\ldots ,p_s\} ;\theta )
\]
which completes the proof.
\end{proof}
\begin{cor}
\label{cor345}
Let $(M,g)$ be orientable and closed of constant sectional curvature $K\geq 0$, then $\ke (M_g)=2/K$.
\end{cor}
\begin{proof}
The cases $K=0$ and $M=S^{2n}$ are known from chapter \ref{chp2}. Hence, assume that $M$ is odd dimensional of positive sectional curvature. Moreover, proposition \ref{proposition7} supplies the claim if $\ke (S^{2n-1}_{g_0})=2$ holds for the standard sphere $S^{2n-1}_{g_0}$, $n\geq 2$. Now, proposition \ref{prop_scalar} yields $\ke (S^{2n-1}_{g_0})\leq 2$. In order to see the opposite inequality it suffices to show
\[
\ke (S^{2n-1}_{g_0}\times S^1_r,S^{2n-1}\times \{ x\} ;[(S^{2n-1}\times S^1,S^{2n-1}\times \{ x\} )])\geq 2
\]
if $S^1_r$ is sufficiently large. The map
\[
(S^{2n-1}_{g_0}\times [0,2\pi r],S^{2n-1}\times \{ 0,2\pi r\} )\to (S^{2n-1}_{g_0}\times S^1_{r},S^{2n-1}\times \{ x\} )
\]
is a relative diffeomorphism and an isometry, i.e.~the corresponding relative K--areas coincide. Moreover, the usual map \[
f:(S^{2n-1}_{g_0}\times [0,\pi r],S^{2n-1}\times \{ 0,\pi r\} )\to (S^{2n}_{g_0'},\{ p,-p\} )
\]
is a relative diffeomorphism with $f^*g_0'\leq g_0\oplus \mathrm{d}s^2$ if $g_0'$ is the standard metric on $S^{2n}$ and $r\geq 1$. Thus, the previous lemma and remark \ref{rem243} yield
\[
\ke (S^{2n-1}_{g_0}\times [0,\pi r],S^{2n-1}\times \{ 0,\pi r\} )\geq \ke (S^{2n}_{g_0'},\{ p,-p\} )=\ke (S^{2n}_{g_0'})=2.
\]
as long as $r\geq 1$.
\end{proof}
The even dimensional case of theorem \ref{thm343} follows from theorem 3 in \cite{List9}. Hence, we assume that $M$ and $M_0$ are odd dimensional. Consider the manifold $(M_0\times S^1,g\oplus \mathrm{d}t^2)$ and let $e_1^0,\ldots ,e_{n+1}^0$ be an arbitrary orthonormal frame with $e_1^0,\ldots ,e_n^0\in TM_0$ and $e^0_{n+1}\in TS^1$. Then the previous corollary shows that for all $\epsilon >0$ there is a Hermitian bundle $\bundle{E}\to M_0\times S^1$ with Hermitian connection such that $\left< \mathrm{ch}(\bundle{E}),[M_0\times S^1]\right> \neq 0$ and
\[
| R^\bundle{E}_{e_i^0,e_j^0}| _{op}\leq \frac{K_0}{2}+\epsilon \qquad \text{as well as}\qquad |R^\bundle{E}_{e_i^0,e_{n+1}^0}|_{op}\leq \epsilon 
\] 
where $1\leq i\leq j\leq n$ (note that we take the supremum over all line elements on $S^1$ to define the K--area). Consider the map
\[
\tilde f=f\times \mathrm{id}:(M\times S^1, g\oplus r^2\cdot \mathrm{d}t^2)\to (M_0\times S^1,g_0\oplus \mathrm{d}t^2)
\]
where $r:=\max \mathrm{dil}_1(f)$. Since $f$ has nontrivial degree, $\tilde f$ has nonvanishing degree and the bundle $\tilde f^*\bundle{E}\to M\times S^1$ satisfies $\left< \mathrm{ch}(\tilde f^*\bundle{E}),[M\times S^1]\right> \neq 0$. Since $\mathrm{ch}(\tilde f^*\bundle{E})\in H^0(M\times S^1;\Q )\oplus H^{n+1}(M\times S^1;\Q )$ and $M\times S^1$ has trivial $\widehat{A}$--genus, the Dirac operator $\D $ associated to the Dirac bundle $\bundle{S}:=\spinor (M\times S^1)\otimes \tilde f^*\bundle{E}$ has nontrivial index. Consider the metric $\tilde g=g\oplus r^2\cdot \mathrm{d}t^2$ on $M\times S^1$ and let $e_1,\ldots ,e_{n+1}\in T_x(M\times S^1)$ be a $\tilde g$--orthonormal frame with $\tilde f_*(e_i)=\lambda _i(x)\cdot e_i^0\in T_{\tilde f(x)}(M_0\times S^1)$, then
\[
|R^{\tilde f^*\bundle{E}}_{e_i,e_j}|_{op}(x)\leq  \lambda _i(x)\cdot \lambda _j(x)\cdot \left( \frac{K_0}{2}+\epsilon \right) ,\quad |R^{\tilde f^*\bundle{E}}_{e_i,e_{n+1}}|_{op}(x)\leq \epsilon   
\]
for $1\leq i,j\leq n$. Hence, the operator norm of the twisted curvature operator for the Dirac bundle $\bundle{S}$ has a pointwise estimate
\[
|\frak{R}^\bundle{S}|_{op}(x)\leq \sum _{i<j}^{n+1}|R^{\tilde f^*\bundle{E}}_{e_i,e_j}|_{op}(x)\leq \frac{K_0}{4}\sum _{ i,j=1}^n\lambda _i(x)\lambda _j(x)+C\cdot \epsilon 
\]
where $C$ is a constant independent on $\epsilon $. Define the function $\beta :M\times S^1\to [0,\infty )$ by $\beta (x):=\sum _{i<j}^n\lambda _i(x)\lambda _j(x)$, i.e.~$\beta =\mathrm{tr}_g(f^*g_0^{\Lambda ^2})\circ \pi $ if $\pi :M\times S^1\to M$ is the projection and $g_0^{\Lambda ^2}$ means the metric on $\Lambda ^2TM_0$. Obviously, $\beta \leq \frac{n(n-1)}{2}\mathrm{dil}_2(f)\circ \pi $ and equality at the point $x\in M$ implies $\lambda _i(x)=\lambda _j(x)$ for al $1\leq i, j\leq n$ (here we use $\dim M\geq 3$ again). Now we proceed as in the proof of proposition \ref{lem1342}. Set $4\alpha :=\mathrm{scal}_{\tilde g}-2K_0\beta $ for notational simplicity, then $\alpha \geq 0$ on $M\times S^1$ by assumption.  Suppose $\phi \in \ker \D $ with $\int |\phi |^2=\int 1$ where integration is done w.r.t.~the volume form of $\tilde g$. Then the above estimate yields
\[
\begin{split}
0&=\frac{1}{2}\delta _{\tilde g}\mathrm{d}|\phi |^2+|\nabla \phi |^2+\frac{\mathrm{scal}_{\tilde g}}{4}|\phi |^2+\left< \frak{R}^\bundle{S}\phi ,\phi \right> \\
&\geq \frac{1}{2}\delta _{\tilde g}\mathrm{d}|\phi |^2+|\nabla \phi |^2+\alpha |\phi |^2-C\epsilon|\phi |^2. 
\end{split}
\]
Multiply the inequality by $|\phi |^2$ and use $\alpha \geq 0$, then
\[
\int \left| \nabla |\phi |^2\right|^2\leq 2C\epsilon \int |\phi |^4.
\]
Hence, the Poincar\'e inequality on $(M\times S^1,g\oplus r^2\mathrm{d}t^2)$ yields
\[
\int \left( |\phi |^2-1\right) ^2\leq C'\int \left| \nabla |\phi |^2\right|^2\leq 2\epsilon C\cdot C'\int |\phi |^4=\epsilon \tilde C\left[ \int 1+\int \left( |\phi |^2-1\right) ^2\right]
\]
which supplies $|\phi |^2=1+h$ for a $L^1$--function $h$ with $\| h\| _{L^1}\leq \bar C \sqrt{\epsilon }$ if $\epsilon $ is sufficiently small. Integration of the above inequality shows
\[
0\geq \int \left[ |\nabla \phi |^2+\alpha (1+h)-C\epsilon (1+h)\right]\geq \int \alpha -O(\sqrt{\epsilon }),
\]
i.e.~$\alpha \geq 0$ and the limit case $\epsilon \to 0$ implies $\alpha =0$ and therefore, $\beta =\frac{n(n-1)}{2}\mathrm{dil}_2(f)\circ \pi $ as well as $\mathrm{dil}_2(f)g=f^*g_0$. Hence, the conformal argument  at the end of section \ref{sec_conf_extremal} shows $\mathrm{dil}_2(f)=\mathrm{const}>0$ and $f:(M,\mathrm{dil}_2(f)\cdot g)\to (M_0,g_0)$ has to be a Riemannian covering.
\begin{rem}
One should to able to generalize theorem \ref{thm343} to get an analogous result to theorem 3 in \cite{List9}. Moreover, we expect the theorem to hold if the model spaces are products of odd dimensional manifolds with constant sectional curvature. 
\end{rem}

\bibliographystyle{abbrv}
\bibliography{bibliothek}

\noindent Mathematisches Institut, Albert-Ludwigs-Universit\"at Freiburg, Eckerstra\ss e 1, 79104 Freiburg, Germany\\
{\it E-mail address:} mario.listing@math.uni-freiburg.de
\end{document}